\documentclass[reqno]{amsart}
\usepackage{epsfig,latexsym,amsfonts,amssymb,amsmath,amscd}
\usepackage{amsthm}
\usepackage[mathscr]{eucal}
\usepackage{mathrsfs}
\usepackage{mathbbol}
\usepackage{oldgerm,units}

\usepackage{epsfig,latexsym,amsfonts,amssymb,amsmath,amscd,graphics,epic}
\usepackage{amsfonts,amssymb,amsmath,amscd,amsthm}
\usepackage[mathscr]{eucal}
\usepackage{mathrsfs}
\usepackage{oldgerm,units}
\usepackage{wrapfig,epsfig}

\usepackage{ifthen}
\usepackage{mathbbol}

\usepackage{amsthm}
\usepackage[mathscr]{eucal}
\usepackage{mathrsfs}
\usepackage{mathbbol}
\usepackage{oldgerm,units}
\usepackage{wrapfig}

\newtheorem{theorem}{Theorem}[section]

\newtheorem{conjecture}[theorem]{Conjecture}

\newcommand{\Real}{\mathbb R}

\newcommand{\Int}{\mathbb Z}
\newcommand{\Net}{\mathbb N}

\newcommand{\one}{\mathbb{1}}
\newcommand{\zero}{\mathbb{0}}

\newcommand{\trop}[1]{\mathcal{#1}}

\newcommand{\tG}{\trop{G}}

\newcommand{\tI}{\trop{I}}

\newcommand{\tR}{\trop{R}}

\newcommand{\tZ}{\trop{Z}}

\newcommand{\al}{\alpha}
\newcommand{\bt}{\beta}
\newcommand{\gm}{\gamma}

\newcommand{\lm}{\lambda}
\newcommand{\Lm}{\Lambda}

\newcommand{\htA}{\hat{\tA}}

\newcommand{\id}{\inva{\aad}}

    \ifx\proof\undefined
    \newenvironment{proof}{
    \smallskip
    \noindent\emph{Proof.}}{\hfill\(\Box\)
    \bigskip
    } \fi

\newcommand{\bfem}[1]{\textbf{\emph{#1}}}

\newcommand{\ifdef}[3]{\ifthenelse{\equal{#1}{true}}{#2}{#3}}

 \input amssymb.sty
\def\epsc{\preceq_{\operatorname{comp}} }

\def\FR{R}
\def\intt{\operatorname{dom}}

\def\tran{\operatorname{t}}

\newcommand{\Per}[1]{ \left| #1 \right |  _{\operatorname{per}}}

\def\id{\operatorname{id}}

\newcommand{\ds}[1]{\ {#1} \ }

\def\vmap{\vartheta}
\def\vrp{\varphi}

\def\mcF{\mathcal F}
\def\mcS{\mathcal S}
\def\mcP{\mathcal P}
\def\mcC{\mathcal C}
\def\mcI{\mathcal I}
\def\mcL{\mathcal L}
\def\mcZ{\mathcal Z}

\def\Prim{\mathcal P}

\def\BB{\{0,1 \}}
\def\hell{\hat\ell}

\input xy
\xyoption{all}

\def\tlR{\widetilde{R}}
\def\htR{\widehat{R}}
\def\htPhi{\widehat{\Phi}}
\def\htG{\widehat{G}}
\def\htA{\widehat{A}}

\def\FuncalSR{\operatorname{Fun} (\tSS,R)}
\def\FuncalSpR{\operatorname{Fun} (\tSS',R)}
\def\FunSR{\operatorname{Fun} (\tSS,R)}
\def\nuge{\ge_\nu}

\def\htf{\widehat{f}}
\def\htg{\widehat{g}}
\def\hth{\widehat{h}}
\def\htbfa{\widehat{\bfa}}

\def\htgm{\widehat{\gm}}

\def\bzero{{\bf 0}}

\def\ldR{(\R, L, (\nu_{m,\ell} ))}
\def\ldsR{(\R, L, s, (\nu_{m,\ell} )) }
\def\ldsPR{(\R, L, L_+, s,(\nu_{m,\ell} ))}

\def\ldsLRpr{(\R', L',   s', (\nu'_{m,\ell} ))}

\def\({\left(}
\def\){\right)}

\def\bbN{\mathbb N}
\def\bbR{\mathbb R}
\def\bbQ{\mathbb Q}

\def\bbC{\mathbb C}
\def\permanent{layered permanent}

\def\tlR{\widetilde{R}}

\def\olR{\overline{R}}

\def\Real{\mathbb R}
\def\Q{\mathbb Q}
\def\mfa{\mathfrak a}

\def\ltw{0.7\textwidth}

\def\beginA {\pSkip \qquad \begin{minipage}{\ltw}}
\def\endA{\end{minipage} \pSkip}

\newcommand\boxtext[1]{\pSkip \qquad \qquad \qquad \framebox{\parbox{\ltw}{#1}}\pSkip}

\newcommand{\xl}[2]{\,\,{^{[#2]}}{#1}\,}

\newcommand{\Det}[1]{ \left|{#1}\right|}
\newcommand{\spol}[2]{{#1}_{[#2]}}

\newcommand{\etype}[1]{\renewcommand{\labelenumi}{(#1{enumi})}}
\def\eroman{\etype{\roman}}
\def\Lay{\operatorname{Lay}}

\def\corn{{\operatorname{corn}}}
\def\rat{\operatorname{rat}}

\def\ealph{\etype{\alph}}

\def\tG{\mathcal{G}}
\def\tSS{\mathcal{S}}

\def\bfi{{\bf i}}
\def\bfj{{\bf j}}
\def\bfk{{\bf k}}
\def\bfa{{\bf a}}
\def\bfb{{\bf b}}
\def\bfc{{\bf c}}
\def\one{\mathbb{1}}
\def\zero{\mathbb{0}}
\def\pSkip{\vskip 1.5mm \noindent}
\def\res{\Re}
\def\les{\frak L}
\def\mcR{\mathcal R}

\def\csupp{\operatorname{csupp}}

\def\nucong{\cong_\nu}
\def\gnu{>_\nu}

\def\pipe{{\underset{{\ \, }}{\mid}}}
\def\pipeGS{{\underset{\operatorname{\, gs }}{\mid}}}
\def\pipeL{{\underset{{L}}{\mid}}}

\def\pipel{{\underset{{\ell}}{\mid}}}

\def\pipeWl{{\underset{{\ell}}{\mid}}}

\def\pipeWL{{\underset{L}{\mid}}}
\def\lmodg{\mathrel  \pipeGS \joinrel \joinrel \joinrel =}
\def\lmodL{\mathrel  \pipeL \joinrel \joinrel =}

\def\lmodl{\mathrel  \pipel \joinrel   =}
\def\pipe1{{\underset{{1}}{\mid}}}
\def\lmod1{\mathrel  \pipe1  \joinrel \joinrel =}

\def\lmodWl{\mathrel  \pipeWl   \joinrel \joinrel =}
\def\lmodWL{\mathrel  \pipeWL   \joinrel \joinrel =}

\def\sig{\sigma}
\def\la{\lambda}
\def\al{\alpha}
\def\bt{\beta}

\def\semiring0{semiring$^\dagger$}
\def\domain0{domain$^\dagger$}
\def\domains0{domains$^\dagger$}
\def\semirings0{semirings$^\dagger$}
\def\semifield0{semifield$^\dagger$}
\def\semifields0{semifields$^\dagger$}

\def\tSS{S}
\def\tZ{\mathcal{Z}}

\def\Net{\mathbb N}

\def\lv{\operatorname{s}}

\def\Fun{{\operatorname{Fun}}}

\def\one{\mathbb 1}
\def\zero{\mathbb 0}
\def\rone{\one_R}
\def\rzero{\zero_R}

\def\lone{{1}}
\def\lzero{{0}}

\newtheorem{thm}{Theorem} [section]
\newtheorem{exampl}[thm]{Example}

\newtheorem*{thm*}{Theorem}
\newtheorem{cor}[thm]{Corollary}
\newtheorem{lem}[thm]{Lemma}
\newtheorem{prop}[thm]{Proposition}
\newtheorem*{claim*} {Claim}
\newtheorem{dig}[thm]{Digression}

\newtheorem{acknowledgment*}[thm] {Acknowledgment}

\newtheorem{Note}[thm]{Note}

    \newtheorem*{remarks*} {Remarks}
 \newtheorem*{remark*}{Remark}
 \newtheorem{defn}[thm]{Definition}
\newtheorem{construction}[thm]{Construction}
\newtheorem{rem}[thm]{Remark}

 \newcommand{\cR}{\mathcal{R}}

\renewcommand{\L}{L}
\newcommand\zL{\L^0}

 \renewcommand{\sectionmark}[1]{}

\newcommand{\ve}{\varepsilon}
\newcommand{\iy}{\infty}

\renewcommand{\a}{\alpha}
\def\Lz{L}

\def\tldR{\widetilde R}
\def\tlds{\tilde s}

\def\barL{\bar L}
\def\R {R}
\def\barR{\overline {R}}

\begin{document}

\title[Layered tropical mathematics]
{Layered tropical mathematics} 
\author[Z. Izhakian]{Zur Izhakian}
\address{Department of Mathematics, Bar-Ilan University, Ramat-Gan 52900,
Israel} \email{zzur@math.biu.ac.il}
\author[M. Knebusch]{Manfred Knebusch}
\address{Department of Mathematics,
NWF-I Mathematik, Universit\"at Regensburg 93040 Regensburg,
Germany} \email{manfred.knebusch@mathematik.uni-regensburg.de}
\author[L. Rowen]{Louis Rowen}
 \address{Department of Mathematics,
 Bar-Ilan University, 52900 Ramat-Gan, Israel}
 \email{rowen@macs.biu.ac.il}

\thanks{The research of the first and third authors is supported  by the
Israel Science Foundation (grant No.  448/09).}

\thanks{The research of the first author also was conducted under the auspices of the
Oberwolfach Leibniz Fellows Programme (OWLF), Mathematisches
Forschungsinstitut Oberwolfach, Germany}
\thanks{This research of the second author was supported in part by
 the Gelbart Institute at
Bar-Ilan University, the Minerva Foundation at Tel-Aviv
University, the Department of Mathematics   of Bar-Ilan
University,   the Emmy Noether Institute at Bar-Ilan University,
and the Mathematisches Forschungsinstitut Oberwolfach.}

\thanks{\textbf{Acknowledgement:} We thank Prof.~S.~Shnider and T.~Perri for many helpful comments
 on a draft of this manuscript.}


\subjclass[2010]{Primary 11C08, 13B22, 16D25; Secondary 16Y60,
14T05. }

\date{\today}


\keywords{Supertropical algebra, ordered groups, ordered
semirings, layered domains, tangible, ghost, truncation, layered
supervaluations, layered functions, polynomials, primary
polynomials, Nullstellensatz, layered varieties, resultants,
layered derivatives.}


\begin{abstract} Generalizing supertropical algebras, we present
a ``layered'' structure, ``sorted'' by a semiring which permits
varying ghost layers, and indicate how it is more amenable than
the ``standard'' supertropical construction  in factorizations of
polynomials, description of varieties, properties of the
resultant, and for mathematical analysis and calculus, in
particular with respect to multiple roots of polynomials. Explicit
examples and comparisons are given for various sorting semirings
such as the natural numbers and the positive rational numbers, and
we see how this theory relates to some recent developments in the
tropical literature such as ``characteristic 1,''
``analytification,'' and ``hyperfields.''
\end{abstract}

\maketitle

\setcounter{tocdepth}{1} {\small \tableofcontents}

\numberwithin{equation}{section}

\section{Introduction}\label{sec:Introduction} Tropical geometry, a rapidly
growing area expounded  for example in \cite{Gat,IMS,L,MS,SS}, has
been based on two main approaches. Primarily, tropical curves have
been defined as domains of non-differentiability of polynomials
over the max-plus algebra, and also in terms of valuation theory
applied to curves over Puiseux series. Unfortunately, semirings
such as the max-plus algebra possess a limited algebraic structure
theory, and also do not reflect the valuation-theoretic properties
intrinsic in tropical mathematics (cf.~\cite{Pay2}, for example),
thereby forcing researchers to turn to combinatoric arguments.

To remedy this situation, the first author introduced a
modification in \cite{I} of idempotent semirings, which evolved to
the supertropical semiring, which we call here the
\textbf{standard} supertropical
 semiring.
Its theory is far more compatible with algebraic structure theory
and valuation theory than the max-plus algebra, and has been
investigated in a sequence of papers including
\cite{IzhakianRowen2007SuperTropical} and
\cite{IzhakianRowen2008Resultants}, focusing on various
fundamental properties involving roots of polynomials,
\cite{IzhakianRowen2008Matrices} --
\cite{IzhakianRowen2010MatricesIII} and  \cite{IKR2} which concern
matrices, as well as
 \cite{IKR1} and \cite{IKR3}, which deal directly with supertropical
 valuation theory. The basic idea is to introduce another ``ghost'' copy $\Real^{\nu}$ of the
 max-plus algebra $\Real_{\operatorname{max,+}}$ (graded by   $\{1,
\infty \}$,
 where by definition $1 \cdot 1 = 1 $ and every other sum and product is $\infty$),
 which provides a semiring $R$ that is a cover of the  max-plus algebra  $\Real_{\operatorname{max,+}}$ in which we can ``resolve'' additive
 idempotents, in the sense that $a+a = a^\nu$ instead of $a+a =
 a.$ \footnote{One can  think of the ghost elements as uncertainties in
 classical algebra
 arising from adding two Puisseux series whose lowest order terms have the same
 degree.}
 This modification permits us to detect corner roots of polynomials
 in terms of the algebraic structure by means of ghosts. Comparing this construction with the ``characteristic 1''
 approach of \cite[Definition 2.7]{CC}, we have $1+1+1 = 1+1$ rather than $1+1 =
 1$.

 Although the standard
supertropical
 semiring permits one to
 define tropical varieties algebraically  as roots of
 polynomials (in a certain sense), and
 is quite successful in working with matrices,
it is not so compatible with other basic notions such as
multiplicity of
  roots, and  difficulties are encountered in attempting
 to establish a useful intrinsic differential calculus on the supertropical structure.
 The standard supertropical theory also has
 other drawbacks: Unique
 factorization of polynomials fails, and some of its basic
 verifications are made via ad hoc arguments.

In this paper we remedy many of these drawbacks
 by introducing a new  structure ``sorted'' by a (partially) ordered semiring~$L$
  that refines the
semiring $R$ further. This approach introduces different (possibly
infinitely many) ghost layers, thereby enabling one intrinsically
to handle considerably more mathematical concepts in the tropical
environment. One could view $R$ as graded by the semiring $L$, but
we prefer the term ``layer'' or ``sort'' instead of ``grade''
because the customary decomposition $R = \bigoplus _{\ell \in L}
R_\ell$ is strengthened to the partition $R = \dot\bigcup _{\ell
\in L} R_\ell$. Intuitively, the ghost layers now  indicate  the
number of monomials defining a tangible corner root.  This paper
has four main objectives:

\begin{itemize}
\item Introduce the layered structure and develop its basic
properties, in analogy with the supertropical theory developed
previously. This includes a description of polynomials and their
behavior as functions, and the foundations of an intrinsic
algebraic geometry in terms of ``layering maps.'' \pSkip

\item Indicate how the layered structure extends the scope of the
supertropical theory, as well as the max-plus theory. For example,
we treat multiple roots by means of layers.  This enables one to
obtain decisive results about resultants. The layered theory gives
us a precise description of the resultant in Theorem (although a
new intricacy develops with resultants of primary polynomials).

One of our main earlier supertropical results, concerning the
multiplicative properties of the resultant (of two polynomials),
cf.~\cite{IzhakianRowen2008Resultants}, had been
 obtained previously only for $\nu$-equivalence. The current layered approach
 gives a considerably stronger result (even for the supertropical case). \pSkip

\item Show how certain supertropical proofs actually become more
natural and more accessible in the layered theory. \pSkip

\item Relate these various concepts to notions already existing
in the tropical literature. In particular, Definition~\ref{analyt}
is the layered version of the analytification defined in
\cite{Pay2}.
\end{itemize}

One can view the various choices of the sorting set $L$ as
different stages of degeneration of algebraic geometry, where the
crudest (for $L = \{ 1 \}$) is obtained by passing directly to the
max-plus algebra, refined somewhat (for $L = \{ 1 , \infty\}$) by
the supertropical theory, and further by the layered theory.
 But taking the sorting set  $L
= \Net$ (the positive natural numbers) yields better factorization
properties for polynomials (although there still are
counterexamples to unique factorization), and enables one to work
with multiple roots and with derivatives, as seen in
Example~\ref{integ}.  But in order to have access to integration,
we need to take $L \supseteq \Q _{>0}$  (the positive rational
numbers). Recently, Sheiner has further applied the layered
construction to preserve information lost from the original
algebraic setting, as outlined in Examples~\ref{comp05} and
~\ref{comp06}, along the lines of \cite{Par}.

 Since our goal here
is to show how the layered theory enables us to learn much more
about the algebraic structure by means of standard techniques in
commutative algebra, we do not handle the most general situation,
but focus instead on the most important special cases equipped to
handle the vast majority of the applications to tropical geometry.
A more formal categorical picture is given in \cite{IKR4}.

Here is a survey of the main results of this paper.  In order not
to be forced at the outset to adjoin an extra (minimal) zero
element to the max-plus algebra, we find it convenient to work
with semirings without zero, which we denote as
\textbf{\semirings0}. (Since semirings lack additive inverses, the
zero element loses much of its significance in semiring theory.)
This may seem like a relatively minor matter, since one can always
adjoin $\zero$ formally at any stage. However, we shall see that
the placement of a zero element involves some basic issues in the
theory (including the inclusion of a special 0-layer), which we
discuss briefly in \S\ref{adj00}.

In \S\ref{back} we introduce the basic algebraic structures  --
ordered monoids and bipotent semirings. These two notions are
essentially equivalent, as seen in Corollary~\ref{ordm} and
Remark~\ref{monsem}, and each language has its particular
advantages. Whereas ordered groups arise as targets of valuations,
and have well-studied completions described in
Remark~\ref{linalg2}, semirings permit the introduction of notions
familiar from classical algebra such as polynomials, matrices, and
modules, and provide the framework for our theory.

In \S\ref{basic}, we generalize the standard supertropical
semiring to
  a \semiring0 with different   layers \textbf{sorted}
  by elements of an arbitrary (partially) ordered \semiring0 $L$. In most applications,
  we build $R$ by taking copies of an ordered monoid $\tG$. The familiar max-plus algebra is recovered
by taking $L=\{1\}$, whereas the standard supertropical structure
is obtained when $L=\{1, \infty\}$, where $R_1$ and $R_\infty$ are
two copies of $\tG$, with $R_1$ identified with the
\textbf{tangible} copy of $\tG$ and $R_\infty$ being
  the \textbf{ghost} copy.
Other useful choices of $L$
include
  $\{1,2,\infty\},$
$\Net$, $\Q _{>0}$, $\Real_{>0}$, and the corresponding semirings
with 0 adjoined. The 1-layer, which is the part of significance in
the ``mainstream'' tropical theory, is a multiplicative monoid
corresponding to the tangible elements in the standard
supertropical theory, and the $\ell$-layers for $\ell
>1$ correspond to the ghosts in the standard supertropical theory.

After laying out the basic definitions and the main motivating
example (cf. Construction~\ref{defn5}), we
 see how this layered structure arises quite naturally in a unified  axiomatic framework,   focusing
  on \textbf{layered \domains0}, the
analog of supertropical domains
\cite{IzhakianRowen2007SuperTropical}. Although this is the case
of interest in most tropical applications, we also present a more
general layered version of supertropical semirings, which requires
the somewhat more intricate Definition~\ref{defn2} and includes
Example~\ref{moveideal}. The benefit is that one   now has a much
broader pool of examples, such as finite tropical structures,
given in Example~\ref{trun0001}. We  introduce the ``surpassing
relation'' (\S\ref{surp1}), generalizing ``ghost surpassing'',
which plays such a fundamental role in the standard supertropical
theory.

 In \S\ref{trunc1} we
consider the method of ``truncating'' the sorting \semiring0 $L$
to pass from a given layered structure to less refined structures
(including the standard supertropical structure, and even the
max-plus structure). This is a special case of a \textbf{layered
homomorphism} of layered \domains0, discussed briefly in~
\S\ref{homomorph}. It turns out that any ``natural'' layered
homomorphism is given in terms of its action on tangible elements,
cf.~Theorem~\ref{tandet}.

 As in the tropical and ``standard''
supertropical theories, one  then proceeds to the polynomial
\semiring0.   
Polynomials provide a major motivation, since one can define
varieties in terms of roots of polynomials (via the layers). Since
it is convenient to view polynomials as functions, as well as
variants such as Laurent polynomials, we take an
  excursion  in \S\ref{Fun5} to study the function \semiring0 $\Fun (\tSS,R')$
from an arbitrary set~$\tSS$ to an arbitrary extension $R'$ of a
layered \domain0 \ $R$. In Theorem~\ref{test1}, we show how one
explicit extension, the completion of the $1$-divisible closure of
the layered $1$-\semifield0 of fractions of $R$, already tests
when two polynomials are equal as functions on all extensions of
$R$. Although a self-contained proof is given, this is really a
consequence of model theory, and one way of viewing polynomials is
as formulae in the appropriate first-order language.

In \S\ref{Zar00}, the  layered theory yields
 a basic Zariski-type correspondence between tropical geometry
and ideals of polynomial semirings, by means of \textbf{layering
maps}, enabling us to formulate a layered version of Hilbert's
Nullstellensatz in Theorem~\ref{Null2}. Layering maps also provide
a layered Zariski topology, leading to \textbf{layered varieties},
which are expected to play a major role in understanding the
underlying algebraic geometry.

The study of polynomials and their corner roots  is one area in
which the layered theory has a distinct advantage, which we
describe in detail in \S\ref{polyone} for polynomials in one
indeterminate. First of all, we call a monic polynomial
$a$-\textbf{primary} if   $a$ is the only  corner root of $f$, up
to $\nu$-equivalence, cf.~Lemma~\ref{aprim}. (In other words, the
variety of a primary polynomial is trivial.)
 As noted by
Sheiner~\cite{Erez1} and quoted in Theorem~\ref{Erez1} (under the
assumption that the ``sorting \semiring0'' $L$ is a semifield),
every polynomial can be factored uniquely into primary polynomials
that correspond to its corner roots. An explicit computation of
layers of evaluations of a polynomial is given in
Corollary~\ref{fact2}. In the traditional algebraic setting over a
field, the primary polynomials  in one indeterminate are precisely
the powers of linear polynomials, i.e., $(\la - a)^m$. Over the
max-plus algebra, the primary polynomials have the form $\la ^m +
a^m.$

In the layered theory, primary polynomials   are more varied as
polynomials, and are  the key to the theory of polynomials in one
indeterminate, providing counterexamples to unique factorization,
but also yielding much information via a transition to the
classical polynomial ring, cf.~Proposition~\ref{sort3}.

We then turn to resultants, obtaining some of the main new results
of this paper. Surprisingly, despite
\cite[Theorem~4.12]{IzhakianRowen2008Resultants}, we see in
Example~\ref{multresult01} that the resultant $|\res(f,g)|$ is not
multiplicative, in the sense that $a$-primary  polynomials $f,g,$
and $h$ need not satisfy
$$|\res(f,gh)| = |\res(f,g)| |\res(f,h)|.$$ But
the
  resultant is multiplicative up to $\nu$-value, cf.~Theorem~\ref{multresult5}, and
the
  resultant is multiplicative  modulo primary polynomials, cf.~   Theorem~\ref{multresult413}.
  These positive results are enough
to
  prove that the resultant is multiplicative in many cases, including the standard
  supertropical theory, cf.~(Theorem~\ref{multresult6}). The positive results for primary polynomials are
  proved by considering the
  esoteric
  behavior of the matrix permanent, which affects the ghost layers of resultants of
  primary polynomials.

 A brief discussion of
differentiation and integration is given in \S\ref{diff},
including the intriguing fact  (Corollary~\ref{sep1}) that the
 layered derivative of a separable polynomial is
 separable. These results open the door to layered discriminants,
 which are inaccessible in the unlayered theory. The sort of an arbitrary separable polynomial is obtained in
 Theorem~\ref{discsort}, thereby enabling us to identify separable
 polynomials without computing their roots.

In \S\ref{majex} we consider the main examples for the sorting
\semiring0 $L$ and discuss their respective advantages and
disadvantages in connection with supporting an intrinsic algebraic
theory for tropical mathematics.  The ability to detect multiple
roots requires $R$ to have an extra (finite) ghost layer.

Unique factorization into irreducibles fails in the standard
tropical and supertropical theories. Taking $L = \Net$ yields
enough refinement to permit us to utilize some tools of
mathematical analysis, as indicated above. Taking $L = \Q _{>0}$
enables us to factor polynomials in one indeterminate into primary
factors, and ``almost'' restores unique factorization in one
indeterminate, as indicated below and explained in detail in
\cite{Erez1}. (Unique factorization in several indeterminates
still fails in certain situations, but for the geometric reason
that certain varieties can be decomposed non-uniquely as unions of
hypersurfaces, even when one takes multiple roots into account).
One also can integrate polynomials, as observed in
Example~\ref{ration}.

 Although in our
applications the sorting \semiring0 $L$ is almost always   totally
ordered, one can make do with a directed partial order. This more
general approach is sketched in Appendix A, where we also
formulate and apply the definition of ``symmetry'' given in
\cite{AGG} to the layered structure, by means of a ``negation
map'' on  ~$L$.

The theory works more generally (and perhaps more aesthetically)
 when the sorting set $L$ is merely
a pre-ordered monoid (not necessarily a \semiring0), as indicated
in Appendix B.


Recently, Viro~\cite{V} has introduced an algebraic approach based
on ``hyperfields,'' which are sets with multi-valued operations.
The ghost layers can be viewed in this context, in the following
way: Suppose $f$ and $g$ are two Puisseux series with respective
value $c$ and $d$. If $c <d$ then $v(f-g) = c.$ But if $c=d,$ then
$f-g$ could have any value $\ge c,$ so in our context one could
consider $\xl{c}{2} : = c+c$ to have the possibility of taking on
all values $\ge c.$ We examine this connection further in
Example~\ref{doubntrunc1} of Appendix A.

 In \cite{CC} a general categorical
geometric theory was outlined for ``characteristic 1.'' In this
paper, we intersperse analogs to that paper, to indicate briefly
how one may obtain analogous results in the layered theory.
Likewise, we indicate how the theory relates to \cite{Par},
\cite{Pay1}, and \cite{Pay2}.

\section{Background}\label{back}

 Recall that a (multiplicative) \textbf{monoid} $M = (M, \cdot \ )$ is a semigroup together with a unit
element $1 = 1_M$, and a \textbf{monoid homomorphism} $\vrp: M \to
M'$ satisfies $\vrp(1) = 1$ and $\vrp(ab) = \vrp(a)\vrp(b)$ for
all $a,b \in M$. A  \textbf{monoid ideal} $A \triangleleft M$ is a
subset
 $A$ for which $\ell a \in A$ and $a \ell  \in A$ for each $a\in A$ and $\ell \in
M.$

The monoid $M$ is \textbf{cancellative} if $a b = a c$ implies
$b=c.$ There is a well-known localization procedure with respect
to multiplicative subsets of Abelian monoids, described in
\cite{CHWW}; if $M$ is cancellative, then localizing with respect
to all of $M$ yields its \textbf{group of fractions}. We say that
a monoid $M$ is $\Net$-\textbf{cancellative} if $a^n   = b^n $ for
some $n\in \Net$ implies $a=b.$

\begin{lem}\label{cannet} Any ordered, cancellative monoid is
$\Net$-cancellative.\end{lem}
\begin{proof} If $a^n = b^n$ for $a<b$ then $a^{n} \le a^{n-1}b
\le b^n = a^n$ implies  $ a^{n-1} a = a^{n}  = a^{n-1}b,$ so $a =
b.$
\end{proof}

\begin{lem}\label{inf1} In any cancellative $\Net$-cancellative monoid $L \ne \{ 1 \}$, the
powers of any element $a \ne 1$ are all distinct. In particular,
$L$ is infinite.
\end{lem}
\begin{proof} If $a^i = a^j$ for $i<j$ then $a^{j-i}a^i = 1a^i,$
implying $a^{j-i} = 1 = 1^{j-i},$ and thus $a = 1.$
\end{proof}

 We say that a monoid $M$ is $\Net$-\textbf{divisible} if for
each $a\in M$ and $m \in \Net$ there is $b\in M$ such that $b^m =
a.$ For example, $(\Q, +)$ is $\Net$-divisible.

\begin{rem}\label{divclo} The customary way of embedding an Abelian  monoid $M$ into
an $\Net$-divisible monoid, is to adjoin $\root m \of a$ for each
$a \in M$ and $m \in \Net,$ and define
$$\root m \of a \root n \of b = \root {mn} \of {a^n b^m}.$$

When $M$ is $\Net$-cancellative, one can reduce modulo the
equivalence relation
$$ \root m \of a  \equiv \root n \of b \quad \text{iff}\quad a^n =
b^m,$$ thereby yielding an $\Net$-cancellative $\Net$-divisible
monoid $M'$; furthermore,  $M'$ is a group if $M$ is a group.
\end{rem}

\begin{rem}\label{linalg} For any
$\Net$-divisible, $\Net$-cancellative group $(G, \cdot)$, we can
uniquely define $n$-th roots for any $n \in \Net$, and thus we can
uniquely define arbitrary rational powers of elements of $G$. In
this way, $G$ becomes a vector space over $\Q$, where we rewrite
the operation of $G$ as addition and define $\frac mn \cdot a$ to
be $a^{m/n}.$ Indeed, $$\left(\frac {m_1}{n_1}+\frac
{m_2}{n_2}\right)\cdot a = a^{\frac {m_1}{n_1}+\frac {m_2}{n_2}} =
a^{\frac {m_1}{n_1}} a^{\frac {m_2}{n_2}} = \frac {m_1}{n_1}\cdot
a +\frac {m_2}{n_2}\cdot a,$$ and the other identifications are
analogous. Thus, we may apply linear algebra techniques to the
group $G$.
\end{rem}

\subsection{Ordered groups and monoids}

The passage to the max-plus algebra in tropical mathematics is
done via ordered groups and, more generally, ordered monoids.

\subsubsection{Ordered groups}

The notion of a  ``pre-ordered group'' is quite well known; it
satisfies the following property: \begin{equation}\label{ogr} a
\le b \quad \text{implies}\quad ga \le gb\text{ and }ag \le
bg,\end {equation} for all elements $a,b,g$.

\begin{rem}\label{linalg2} Any ordered $\Net$-divisible, $\Net$-cancellative Abelian group $G$ can be viewed as a metric space
(where we define $d(a,b)$ to be $ab^{-1}$),  and thus completed in
the usual way,  as described in \cite{CG} and \cite[p.~196]{Ke}.
This can be done using general model-theoretic methods,
\cite[p.~116]{M} and \cite[pp.~35,36]{Sa}, since the theory of
ordered $\Net$-divisible Abelian groups is model complete.

More specifically, let us sketch how one can work directly with
Cauchy sequences in the case under consideration here, following
ideas given in \cite{Ho}. There is the difficulty of defining
convergence to 0, since in general the group $G$ might lack the
archimedean property of the real numbers. This is treated in depth
in \cite[\S2]{Ho}. As a special case, one can specify a collection
$\mcF$  of sets $ \mcS$ of subsets of $G$ having the property that
$\bigcap _{S \in \mcS} = \{0\}$ for each~$\mcS$ in~$\mcF$, and
define a sequence $(a_i): = \{ a_1, a_2, \dots \}$ to be
\textbf{$\mcF$-Cauchy} if there is some $\mcS\in \mcF$ satisfying
the property that for each $S \in \mcS$ there is $m = m(S)$
depending on $S$ such that $a_i a_j^{-1} \in S$ for all $i,j>m.$
Intuitively, this means that $a_i a_j^{-1}$ is ``small'' whenever
$i,j$ are ``sufficiently large.'' An $\mcF$-\textbf{null sequence}
is an $\mcF$-Cauchy sequence satisfying the property that for each
$S \in \mcS$ there is $m = m(S)$ depending on $S$ such that $a_i
\in S$ for all $i>m.$

Here, we define $S_{a, \varepsilon} := \{\a a: \a <
\varepsilon\},$ and take $\mcF$ to be the collection of sets
$$\mcS_a = \{S_{a, \varepsilon}: 0< \varepsilon \in \Q
\},$$ for  $ a \in G.$ If two sequences are $\mcF$-Cauchy with
respect to $\mcS_a$ and $\mcS_b$ respectively, then their product
is $\mcF$-Cauchy with respect to $\mcS_{ab}$. Thus, as in the
general theory of Cauchy sequences, the set of $\mcF$-Cauchy
sequences is a pre-ordered group, having a subgroup comprised of
the set of $\mcF$-null sequences, and the quotient group is an
ordered group~$\htG$ which we call the $\mcF$-\textbf{completion}
of $G$.


Now writing the operation additively and viewing $G$ as a
$\Q$-vector space as in Remark~\ref{linalg}, we can view its
$\mcF$-completion as a vector space over the completion $\Real$ of
$\Q$.
\end{rem}

\subsubsection{Ordered monoids}

 Ordered monoids are
trickier than ordered groups, since, for example, $a
>b$ in $(\Real,\cdot)$ does not imply $-a > -b,$ but rather $-a < -b.$
The easiest way around this is to require  \eqref{ogr} to hold
anyway; in other words, to declare that all elements are
non-negative; we call such a monoid \textbf{positively ordered}.

\begin{rem} If $M$ is a positively ordered monoid, then its
divisible closure $M'$ of Remark~\ref{divclo} is also positively
ordered, by putting $ \root m \of a  \ge \root n \of b$ iff $a^n
\ge b^m.$ \end{rem}

\subsection{Semirings without zero}

As with the ``standard'' supertropical structure, we work in the
 language of semirings and use~\cite{golan92} as a general
 reference. Unless explicitly stated otherwise, we use
algebraic notation, in which $\rone$ denotes the multiplicative
identity of $R$. (For our examples, we occasionally use
``logarithmic notation,'' in which $\rone$ is~$0$.)

It is more convenient to consider, slightly more generally, a
\textbf{semiring without zero}, which we notate as \semiring0, to
be a structure $(\R ,+,\cdot \, , \rone)$ such that $(\R ,\cdot \,
,\rone)$ is a monoid and $(\R ,+)$ is an Abelian semigroup, with
distributivity of multiplication over addition on both sides. (In
other words, a \semiring0\ does not necessarily have the zero
element $\rzero$, but any semiring can also be considered as a
\semiring0). Ironically, we do assume that every \semiring0\ has
the unit element $\rone$. A \textbf{\semifield0} is a \semiring0\
$R$ for which $(\R ,\cdot \, ,\rone)$ is an Abelian  group.

An \textbf{ideal} $A$ of a \semiring0\ $R$, denoted $A
\triangleleft R,$ is defined to be a sub-semigroup
 of $(R,+)$ which is also a monoid ideal of $(R,\cdot\;)$.
 An ideal  $P$   of $R$   is
 \textbf{prime} if $ab \in P$ implies $a\in P$ or $b\in P$.

\begin{defn} A \semiring0 $R$   has the \textbf{infinite element} $\infty$
if
\begin{equation}\label{inf} \infty + a = \infty = \infty\cdot a =
a\cdot \infty, \qquad \forall a \in R.
\end{equation}
\end{defn}

When $R$ does have a zero element $\zero:= \rzero$, we require
instead that
$$ \infty + \zero = \infty \quad \text{but} \quad \zero \cdot \infty =
\infty\cdot \zero = \zero.$$ (This is to enable $\zero$ to remain
the zero element.)

The following observation enables us to pass from a \semiring0\ to
a semiring.

\begin{rem}\label{ambig} $ $ \begin{enumerate} \eroman
\item
    Given a \semiring0\ $\R $, one can formally adjoin a zero  element
$\rzero$ which is multiplicatively ``absorbing'' in the sense that
\begin{equation}\label{addzero} \rzero \cdot
a = a\cdot \rzero  = \rzero, \qquad \forall a \in \R
,\end{equation} to obtain a semiring
$$ \R \,\dot\cup\, \{ \rzero \}.$$  \pSkip

 \item Alternatively, one could   formally
adjoin an element $\infty$ to obtain a \semiring0  with infinite
element
 $$ R^ \infty := R\,\dot\cup\, \{ \infty \} $$ satisfying Equation
\eqref{inf}. \pSkip

 \item Finally, one could adjoin $\infty$ and then $\zero_R$.
 Then $\infty$ satisfies \eqref{addzero} for every element  $a \in R$
 except~$\zero_R$. 
\end{enumerate}
\end{rem}

Although the notion of \semiring0\ is somewhat unusual, it fits
our needs like a glove, since the max-plus algebra $(\mathbb R,
\max , + , 0 )$, with $\rone = 0$, is a \semiring0\ before
adjoining the zero element $\rzero:= -\infty$ and much of the
theory of supertropical domains, cf.
\cite{IzhakianRowen2007SuperTropical}, is stated more concisely
when we do not have to consider special cases involving the
element $\rzero$.

(An example of how the adjoined element $\rzero = -\infty$ gets in
the way: The dual of the max-plus \semiring0\ is the min-plus
\semiring0\, where ``max'' is replaced by ``min.'' But the dual of
the max-plus semiring is not a semiring, since $-\infty$ no longer
performs the role  of the zero element!)

 A \semiring0\ \textbf{homomorphism} is
a map $\vrp: R \to R'$ of \semirings0\ satisfying $$\vrp (a+b) =
\vrp (a) + \vrp (b), \qquad \vrp (ab) = \vrp (a)\vrp (b),$$ for
all $a,b \in R,$ and $\vrp (\one_R) = \one_{R'}.$

(When working with semirings, one also requires that $\vrp(\rzero)
= \rzero'.$)

\begin{Note}  As
 with monoids, the class of homomorphisms from a semiring $R$ to another
semiring is far richer than the set of ideals of $R$. Given a
semiring  homomorphism $\vrp: R \to R'$ we can define an
equivalence relation $\equiv_\vrp$ on $R$ by $$ a  \equiv_\vrp b
\quad \text{ iff}  \quad \vrp(a) = \vrp(b).$$ If $R$ and $R'$ are
rings, then $\equiv_\vrp$ is determined by the ideal $\ker \vrp$,
since $a \equiv_\vrp b $ iff $a-b \in \ker \vrp.$ But this is no
longer true for semirings. ``Non-equivalent'' onto homomorphisms
may have the same kernel. Thus, in contrast to ring theory, the
theory of semiring homomorphisms is much richer than the theory of
ideals,  and we need to cope with this extra complication.
\end{Note}

\subsection{Ordered \semirings0}\label{ordsem}
Since the tropical structure is largely based on orderings, we
pause to consider orders on \semirings0.  In general, this is
often handled in the theory of ordered rings (for example, in
\cite{VdD}, which also treats ordered structures in model theory)
by viewing an ordered ring $R$ as $R_+ \cup R_- \cup \{ \rzero\},$
where $R_+$ is the set of \textbf{positive elements} and $R_-$ is
the set of \textbf{negative elements}, but in semirings we might
lack negative
 elements altogether. In fact, in the tropical theory our ordered
semirings almost never have negative elements.

Here we take a more general approach, which might give a better
indication of the idea of ``ghost layers.'' Given a \semiring0
$L$, we designate a sub-\semiring0 $L_+\subseteq L$ of
\textbf{positive elements}. By assumption, $1 \in L_+.$ We write
$k \ge \ell$ when $k=\ell$ or $k = \ell +p $ for $p \in L_+$.

 A priori, this relation $(\ge)$ is only a \textbf{partial pre-order}, but we
 stipulate that $\ell +p_1+p_2 = \ell $ for $p_1, p_2 \in L_+ $ implies $\ell  + p_1 =
 \ell $. Thus, $(\ge)$ is  antisymmetric, and hence a partial order.

 For $\ell \in L$, an $\ell$-\textbf{ghost sort} is an
element of the form $\ell + k$, where $k \in L_+$.
 A 1-ghost
sort of $L$ is called a \textbf{ghost sort}.


\begin{lem} Suppose $\ell \ge k$. Then:
\begin{enumerate}\eroman
 \item $\ell +m \ge k +m$ for all $m \in
L$. \pSkip
 \item $\ell p \ge kp$ for all $p \in
L_+ $.
 \end{enumerate}

\end{lem}
\begin{proof} (i): If $\ell = k +p', $ for some $p' \in L_+ $, then $\ell +m = k +m +p'. $
\pSkip (ii): If $\ell = k +p', $ for some $p' \in L_+ $, then
$\ell p = k p +p' p. $
\end{proof}

\begin{rem}
To make the exposition simpler, one usually may assume that our
partial order on $L$ is a  total order, which implies that for all
$k \ne \ell \in L$ there is $p \in L_+ $ such that either $\ell =
k+p$ or $k = p+\ell$. One situation for which only the weaker
assumption holds is in \S\ref{Fun5}, so we stay with the partial
order.
\end{rem}

\begin{defn} The \semiring0 $L$ is \textbf{non-negative}, if $L = L_+$
or if $L =  L_+ \cup  \{ \lzero \}$ when $L$   has a zero element
~$\lzero := 0_L$.

We define the following non-negative sub-\semirings0 of $L$:
   $$L_{\ge 1}: = \{ \ell \in L:
   \ell \ge \lone \},$$  a
sub-\semiring0 of $L$, 
and $$L_{> 1}: = \{ \ell \in L: \ell \text{ is a ghost sort}\},$$
a \semiring0 ideal of $L_{\ge 1}$.
\end{defn}

 Note that $L_{\ge 1}$ contains the sub-\semiring0 generated by 1.

In all of our applications except Examples~\ref{reals};.
\ref{comp0}, and~\ref{comp05}, $L = L_+$ or $L = L_+ \cup \{0\},$
In many major tropical applications we have $L := L_{+}$, and one
could assume this throughout. (The theory runs most smoothly in
this case.)  We permitted the more general situation in the
definition, in order to be able to deal with subtle issues
regarding factorization such as in Proposition~\ref{sort3} below,
which involve the more esoteric Examples~\ref{reals} and
Example~\ref{comp0}.

\begin{defn}\label{preord}  An element $\ell \in L$ is
\textbf{finite} if it satisfies the conditions:
$$\ell + m
\ne \ell, \quad  \forall m \in L_+.$$


\end{defn}\noindent

In our examples, $L$ will have at most one infinite element, often
denoted $\infty.$ For example, the \semiring0~$L$ could be the
following set of
 finite order $q$, for any natural number
$q$:
 \begin{defn}\label{trunc0} The
$q$-\textbf{truncated \semiring0}
$$L=[1,q]:=\{1,2,\dots,q\}$$   is given with  the obvious total ordering; the sum and product of two
elements $k,\ell\in L$ are taken as usual, if it does not exceed
$q-1,$ while the sum or product of $k$ and $\ell$ in $L$ is $q$
otherwise. In other words, $q$ is the infinite element and could
be denoted as $\infty.$
\end{defn}

\begin{rem}\label{trunc00} There is a natural \semiring0 homomorphism $\Net \to [1,q]$ given
by sending $m \to q$ for all ~$m\ge q$.\end{rem}

 The first examples for
infinite sets $L$ are $\bbN $ and $\bbQ _{>0},$ with the usual
addition, multiplication, and ordering. Note in these two examples
that every element is finite as well as positive. In general,
nonzero positive elements of $L$ need not be finite, and we could
have several infinite elements (as can be seen easily by means of
ordinals), but we do not deal with such issues in this paper. When
$L$ is  ordered we extend the  order to $L \cup \{\infty\}$ by
declaring that $\infty >\ell$ for all~$ \ell \in L$.

\begin{exampl}\label{trun1} The \semiring0 $[1,q]$ can be extended
to $\{ \frac im : 1 \le i \le  q \}$, which is isomorphic as a
\semiring0 to $[1,qm]$ under the map $\frac im \mapsto
i.$\end{exampl}

We introduce the following notation: $1\ell$ denotes $\ell,$ and
inductively, for any integer
 $n>1$, we define   $$n\ell = (n-1)\ell + \ell.$$


  \begin{rem}\label{preeq} If   $\ell \ge m\ell \ge \ell$ for some $m\in \Net$ and the  order $(\ge)$ on $L$ is antisymmetric, then
  $n\ell = \ell$ for all~$n$. (Indeed, $m\ell = \ell$ by
  hypothesis, and  the assertion then holds for any multiple $n$ of
  $m$. In general,  for any $u\in \Net$ with $um>n$, $$\ell \le n\ell
   \le um\ell
  =\ell,$$ so again $\ell = n\ell.$)\end{rem}

\subsection{The $\nu$-pre-order}

 \begin{defn}\label{prenu} Suppose $R$  is a  \semiring0 with a designated
 equivalence $ \nucong$ which respects addition, in the sense that
 if $a  \nucong a'$ and $b  \nucong b',$ then $a+b  \nucong a'+
 b'.$
We define a transitive binary relation $\le _\nu $ by
$$a \le _\nu b \qquad \text{ iff }  \qquad a+b \nucong b.$$
\end{defn}

\begin{lem}\label{antis} The relation $\le _\nu $ is anti-symmetric (and thus
is a partial pre-order).
\end{lem}
\begin{proof} $a  + b  = a $ and  $b  + a  = b $  implies $a=b$.
\end{proof}

 \begin{defn}\label{bipot} We say that $ b$ is $\nu$-\textbf{greater than} $a$ in a \semiring0 $R$,
written  $a <_\nu b$, if $a \le _\nu b$ but $a \not \nucong b$. A
\semiring0 $R$ is $\nu$-\textbf{bipotent} if the following two
conditions hold:

\begin{enumerate}\eroman
\item The relation $\le _\nu $ is a total pre-order on equivalence
classes, in the sense that if $a \not \nucong b,$ then either $a
<_\nu b$ or $b < _\nu a$.

\item If  $a <_\nu b$ and $a \not \nucong b,$ then $a+b = b.$
\end{enumerate}\end{defn}

We write  $a <_\nu b$  when $a \le _\nu b$ but $a\not\nucong b.$

The relation $\le _\nu $ respects the monoid structure in the
following sense:

\begin{prop}\label{mult2} If $a \le _\nu b$, then $ac \le _\nu bc$ for all $c;$ furthermore,
if $c \le _\nu d$, then   $ac \le _\nu bd.$
\end{prop}
\begin{proof} $ac + bc = (a+b)c \nucong   bc.$ Furthermore,
$ac+ bd \nucong  ac + (a+b)(c+d) = ac + ac + ad + bc + bd = ac +
a(c+d) + bc + bd \nucong   ac + ad + bc + bd = (a+b)(c+d) \nucong
bd.$
\end{proof}

\begin{cor}\label{ordm} If $R$ is as in Definition~\ref{prenu},
then $R/\nucong$ is an ordered monoid.
\end{cor}
\begin{proof} $\le _\nu$ obviously induces an order on
$R/\nucong$, which is total in view of $\nu$-bipotence. The order
is preserved under multiplication, in view of
Proposition~\ref{mult2}.
\end{proof}

\section{Layered domains without zero}\label{basic}

We get to our main algebraic notion.

\begin{rem}\label{monsem} Any ordered monoid $\tG$ can be viewed as a \semiring0\ in
which addition is given by $$a+b := \max\{a,b\},$$ with respect to
the order of $\tG$. (See
 ~\cite[Theorem~1.5]{IKR1}  for
more details.)\end{rem}

Here is the motivating example for this paper -- a mild
generalization of \cite[Proposition~5.1]{AGG}.

\begin{construction}\label{defn5} Suppose we are given a cancellative
 ordered monoid  $\tG$, viewed as a  \semiring0\ as in Remark~\ref{monsem}. For any \semiring0~$L$ we define the \semiring0
$\R(L,\tG)$ to be set-theoretically $L\times \tG$, where we denote
the element $(\ell,a)$ as $\xl{a}{\ell}$ and for $k,\ell\in L,$
$a,b\in\tG,$ we define multiplication componentwise, i.e.,
\begin{equation}\label{13}   \xl{a}{k} \xl{b}{\ell} =
\xl{(ab)}{k\ell}, \end{equation} and addition from the rules:

\begin{equation}\label{14}
 \xl{a}{k} + \xl{b}{\ell}=\begin{cases}  \xl{a}{k}& \quad\text{if}\ a >
 b,\\ \xl{b}{\ell}& \quad\text{if}\ a <  b,\\
 \xl{a}{k+\ell}& \quad\text{if}\ a= b.\end{cases}\end{equation}
\end{construction}



\begin{prop}  $R :=
\R(L,\tG)$ is a \semiring0.
\end{prop} \begin{proof}
  To prove, for example, the
distributivity law
$$x(y+z)=xy+xz,$$ write $x =  \xl{a}{k},$ $y =  \xl{b}{\ell},$ and $z
=  \xl{c}{m}.$ If $ab > ac$ then clearly $b>c$, and
$$x(y+z) = xy =  xy + xz .$$
Thus we are done unless $ab = ac,$ in which case $b=c$ since $\tG$
is cancellative, and
$$x(y+z) = \xl{a}{k}(\xl{b}{\ell}+ \xl{b}{m}) = \xl{a}{k}(\xl{b}{\ell +m})=  \xl{(ab)}{k\ell+km} =
\xl{(ab)}{k\ell} + \xl{(ab)}{km}  = xy + xz.$$

The other  verifications of the \semiring0\ axioms are
straightforward. The unit element is $\xl{\one_\tG}{1}.$
\end{proof}

Conversely, we have:

\begin{prop}\label{cancel1} If $R :=
\R(L,\tG)$ is a \semiring0, then the monoid $\tG$ is cancellative.
\end{prop} \begin{proof}
If $b>c$ with $ab = ac$ then
$$\xl{a}{1}(\xl{b}{1}+\xl{c}{1}) = \xl{(ab)}{1} $$ whereas
$\xl{a}{1}\xl{b}{1}+\xl{a}{1}\xl{c}{1} = \xl{(ab)}{2}.$\end{proof}

We define
$$\R_\ell :=\{\xl{a}{\ell} \ds |a\in\tG, \ \ell \in L\}.$$ Then $R = R(L, \tG)$ is the disjoint
union of the  $\R_\ell$, and  $R_1$ is a monoid isomorphic to~
$\tG$, which we call the \textbf{tangible} copy of $\tG.$ We have
1:1 maps $\nu_{\ell,k}: R_k\to R_\ell$ given by
$\nu_{\ell,k}(k,a)= (\ell,a)$ for any $k \leq \ell$, and  a map $
\lv: \R\to L$ given by $\lv(\ell,a)=\ell,$ for any $a\in\tG,$
$\ell \in L.$

 The sub-\semiring0 of $R$ generated by $R_1$ is just $\bigcup _{n \in
 \Net} R_{n\cdot1},$ and equals $R$ iff $L$ is generated as a
 \semiring0 by its unit element 1.

\begin{rem}\label{getzero} We could perform the same construction assuming that $L$
contains a zero element $\lzero$. Then $\R_0$ is a \semiring0 as
well as an ideal of $R(L,\tG)$.  When $\tG$ has a zero element
$\zero_\tG,$ then $R = R(L,\tG)$ is a semiring, and the
corresponding zero element of $\R_0$ is the zero element of
$\R(L,\tG)$.
\end{rem}

Since Construction~\ref{defn5} lies at the foundation of this
paper, let us axiomatize it (in slightly greater generality). The
bulk of our applications in this paper are for $L$ ordered. Note
that the construction does not require $L$ to be ordered. On the
other hand, any set $L$ has the trivial partial pre-order, defined
by $k \le \ell$ for all $k,\ell \in L.$ For this reason, we frame
our definition for partially pre-ordered \semirings0.

\begin{defn}\label{defn1} Suppose  $(\Lz, \ge)$ is a partially pre-ordered
\semiring0 (as described in \S\ref{ordsem}), with positive
elements $L_+$.
An $\Lz$-\textbf{quasi-layered \domain0} $$\R :=\ldR, \qquad
$$ is a commutative \semiring0 $\R $, together with a family $\{ R_\ell :
\ell \in L\}$ of disjoint subsets $R_\ell \subset R$,  such that
\begin{equation}\label{unionp} \R := \dot \bigcup_{\ell\in \Lz}\R
_\ell,\end{equation}  and a family
$$(\nu_{m,\ell}) :=  \{ \nu_{m,\ell}  : \ m \ge \ell \quad ( m, \ell \in
L) \}$$  of \textbf{sort transition maps}
$$\nu_{m,\ell}:\R _\ell\to \R _m$$
for each $m\ge \ell$ in $L_+$, such that $\nu_{\ell,\ell}=\id_{\R
_\ell}$ for every $\ell\in \Lz,$ and
$$\nu_{m,\ell}\circ \nu_{\ell,k}=\nu_{m,k}$$ whenever
$m\ge\ell\ge k$, satisfying the  axioms A1--A4,  and B, to be
given below.

 We define the $\nu$-relation $a \nucong b$ for $b \in R_\ell$  if
 $\nu_{m,k}(a) = \nu_{m,\ell}( b )$ in $R_m$ for some $m \ge k,\ell$. (Note
that in this case we also have  $\nu_{m',k}(a) = \nu_{m',\ell}(b)$
for every $m' > m$.)

 The axioms are as follows, where we assume $  k, \ell \in L_+$:

\boxtext{
\begin{enumerate}

\item[A1.] $\rone \in \R _{1}.$ \pSkip

 \item[A2.] If $a\in \R _k$ and  $b\in
\R _\ell,$ then $ab\in \R _{k \ell }$. \pSkip

\item[A3.] The product in $\R $ is compatible with sort transition
maps: Suppose $a\in \R _k,$ $b\in \R _{\ell},$ with $m\ge k$ and
$m'\ge \ell.$

(i) If  $ab\in \R_{k\ell},$ then
$\nu_{m,k}(a)\cdot\nu_{m',\ell}(b)= \nu_{mm',k\ell}(ab).$

(ii) If $ab\in \R_0,$   then $ \nu_{m,k}(a)\cdot\nu_{m',\ell}(b)=
ab. \qquad $ \pSkip

\item[A4.] $\nu_{\ell,k}(a) + \nu_{\ell',k}(a)
 =\nu_{\ell+\ell',k}(a)  $ for all $a \in R_k$ and all $\ell, \ell' \ge k.$ \end{enumerate}}

 \boxtext{
\begin{enumerate}

\item[B.] (Supertropicality) Suppose  $a\in \R _k,$ $b\in \R
_{\ell},$ and $a \nucong b$. Then \\ $a+b \in R_{k+\ell}$ with
$a+b \nucong a$.\\  If moreover $k\in L_+$ is infinite, then $a+b
= a.$
\end{enumerate}}
%

We say that any element $a$ of $\R _k$ has \bfem{layer}~$k$ $(k\in
\Lz)$, and  call $L$ the \textbf{sorting \semiring0} of the
quasi-layered \domain0 $R = \bigcup_{\ell \in L} R_\ell$.

 An $\Lz$-quasi-layered  \domain0 $\R := \ldR$ is  called \textbf{uniform}
 when the sorting \semiring0 $L$ is totally ordered
and the sort transition maps $\nu_{\ell,k}$ are all bijective.

An $\Lz$-\textbf{layered \domain0} is a $\nu$-bipotent
$\Lz$-quasi-layered  domain.
\end{defn}

We write $\lone := 1_L$ for the unit element of $L$.

Thus, the relation
 $(\ge)$ on $L$ satisfies
$$\ell \ge k \quad \text { when }\ \left\{
\begin{array}{l}
  \ell = k   \\
\text{or} \\
  \text{$\ell$ is a $k $-ghost sort.}
\end{array}\right.
$$

Note that infinite elements of $L$ are ``self-ghost sorts,'' in
the sense that $\ell +m = \ell$ implies that $\ell$ is an
$\ell$-ghost sort.

We
 usually require that the  relation
 $(\ge)$ on $L$ is \textbf{directed} in the sense that for any $k, \ell \in L$ there are positive elements $p_1, p_2 \in L_+ $ such
that
\begin{equation}\label{direct} k+ p_1 = \ell + p_2.\end{equation}
 By this condition, any two elements
$k$ and $ \ell$ in $\Lz$ have an
 \textbf{upper bound}; i.e., there exists $m\in \Lz$ such that $  m\ge k,\ell.$
 The existence of
``upper bounds'' enables us to define direct limits over $L$,
which is needed for Remark~\ref{directlim} below.

\begin{rem}\label{infdeg} Axiom B implies that if both $a,b \in R_k$ with $a \nucong
b$ and $k \in L_+$ is infinite, then $a = a+b = b.$
 \end{rem}

\begin{rem}\label{tang1}
 The $L$-quasi-layered \domain0 \ $R$ has
the special layer $R_1$, which is a multiplicative monoid, called
the monoid of \textbf{tangible elements}, and acts with the
obvious monoid action (given by multiplication) on each layer
$R_k$ of $R$.
 \end{rem}

\begin{defn}
 An $L$-layered \domain0 $R$ is called an
\textbf{$L$-layered $1$-\semifield0} if $(R_1,\cdot \; )$ is an
Abelian  group.
\end{defn}

In this case, the action of Remark~{tang1} is simply transitive,
in the sense that for any $a,b \in R_\ell$ there is  a unique
element $R \in R_1$ for which $ar = b.$

Note that according to this definition, an
 $L$-layered $1$-\semifield0 need not be a  \semifield0; cf.~ Proposition~\ref{gp} below.

\begin{exampl}\label{maxpl} Taking $L=\{1\}$ with $1+1 = 1$, then $R = R_1$
gives us the usual (idempotent) max-plus algebra.
\end{exampl}

\begin{exampl}\label{supertropst} $ $  \begin{enumerate} \eroman   \item Taking $L=\{1,\infty\}$ with $1<\infty,$   $ 1
\cdot 1=1,$ and all other sums and products $=\infty,$  we recover
the (standard)
 supertropical \domains0 as defined and studied in
 \cite{IzhakianRowen2008Matrices}. (The tangibles are of layer $1$, and the ghosts are comprised of layer
 $\infty$.)  \pSkip

  \item Another major example, to be considered below, is $L = \Net$.
\pSkip

  \item  Given an arbitrary sorting
\semiring0 $L$, one could define $L_{\Net}$ to be the
sub-\semiring0 generated by~$1_L$, i.e., $\Net \cdot 1_L.$  Then
one could replace an $L$-quasi-layered \domain0  $R$ by its
sub-quasi-layered \domain0 $\sum _{\ell \in L_{\Net}} R_\ell$, and
thereby assume that $L = L_{\Net}.$
\end{enumerate}
\end{exampl}

\begin{rem}\label{Krems} Several initial observations are in order.
\begin{enumerate} \eroman

\item The layered structure resembles that of a graded algebra,
with two major differences: On the one hand, the condition that
$R$ is the disjoint union of its components is considerably
stronger than the usual condition that $R$ is the direct sum of
its components; on the other hand, Axioms ~A4 and~B show that the
components are not quite closed under addition. \pSkip

\item This paper is mostly about $L$-layered \domains0, and we
require $\nu$-bipotence in many results proved here. However,
since $\nu$-bipotence does not hold for polynomials, we need the
more general $L$-quasi-layered \domains0 when studying polynomial
\semirings0. \pSkip

\item  For each $\ell\in \Lz$ we introduce the sets
$$\FR_{\ge \ell}:= \bigcup_{m\ge \ell}\R _m \qquad{and} \qquad \FR_{>\ell}:=
\bigcup_{m> \ell}\R _m .$$

 In many of  our current examples, $L = L_{\ge 1}$ and thus $R =
\FR_{\ge 1} $. When $R$ is an $L$-layered \domain0, we claim that
$\FR_{\ge 1}$ is an $L_{\ge 1}$-layered sub-\domain0\ of $R$, and
$\FR_{\ge k }$ and $\FR_{>k}$ are \semiring0 ideals of~$\FR_{\ge
1}$, for each $k \in L_{\ge 1}.$ Indeed, this is an easy
verification of the axioms, mostly  from Axiom~A2. \pSkip

 \item Given any $L$-layered \domain0 $R$ and any multiplicative submonoid $M$ of
 $R_{\ge 1}$, we want to define the $L$-layered sub-\domain0 of $R$
 generated by $M$. First we take $$M' := \{ \nu_{\ell, k}(a): k,\ell \in L, \ell \ge k, \ a \in
 M \cap R_k\},$$ which is a submonoid closed under the transition maps.
Then we take $$M'': = M' \cup \{ a+b: \ a,b \in M' \quad
\text{with} \quad
 a\nucong b \}.$$ This is closed under all the relevant operations,
so is the desired $L$-layered \domain0. Note that the second stage
is unnecessary for $a=b$ (and similarly for uniform $L$-layered
\domains0), in view of Axiom~A4. \pSkip

\item Although ubiquitous in the definition, the sort transition
maps get in the way of computations, and it is convenient to
define the elements
  \begin{equation}
\label{idemp} e_\ell:= \nu_{\ell,1}(\rone)\quad (\ell\ge 1).
\end{equation}
If $a\in \R _k,$\ $\ell\in \Lz,$ and $\ell\ge 1,$ we conclude by
Axiom A3 that
$$\nu_{ \ell\cdot k,k }(a)=\nu_{\ell\cdot k,1 \cdot k}(a\cdot \rone )
= \nu_{\ell,1}(\rone)\cdot\nu_{k,k}(a)=\nu_{\ell,1}(\rone)\cdot a
= e_\ell a.$$ Thus the sort transition map $\nu_{ \ell\cdot k,k }$
means multiplication by $ e_\ell.$

\end{enumerate}
\end{rem}

Let us introduce the \textbf{sorting map}
$$s:=s_{\R}:\R\to L,$$
which sends every element $a\in\R_\ell$ to its sort index $\ell$,
and we view the \semiring0  $\R$ as an object fibred by $s$ over
the sorting \semiring0~$\L$.

\begin{rem}  Axioms A1 and A2   yield
the conditions
\begin{equation}\label{sortval} \lv(\rone)=\lone, \qquad \lv(ab)=\lv(a)\lv(b), \qquad \forall a,b\in\R.\end{equation}
Also,  Axiom A4 yields $s(a+a) = s(a) + s(a) = 2s(a),$ thereby
motivating us to view addition of an element with itself as
doubling the layer. Applying $\nu$-bipotence to Axiom B  shows
that
$$s(a+b) \in \{ s(a), s(b) , s(a)+s(b)\}.$$
\end{rem}

 To emphasize the sorting map, as well as the order on $L$, we sometimes write
$\ldsPR$ for a given $\Lz$-layered \domain0 $\R$ with sort
transition maps $(\nu_{m,\ell}: m \ge \ell)$ and their
accompanying sorting map $s: R\to L,$  where, as usual, $L_+$ is
the set of positive elements of $L$. 

There are two main  examples.

\begin{enumerate} \ealph \item Let $R = R(L,\tG)$ (corresponding to the ``naive'' tropical geometry). By Construction~\ref{defn5}, $\nu_{m,\ell}$ are
all bijective. \pSkip

\item (Corresponding to the supertropical structure given in
~\cite{IzhakianRowen2007SuperTropical}.) Notation as in
\cite[Theorem~1]{IzhakianRowen2007SuperTropical}, we incorporate
$\mathbb K$ into the structure of $R$, by putting $R_\ell$ to be a
copy of $\mathbb K$ for $\ell \le 1$ and $R_\ell$ to be a copy of
$\tG$ for  $\ell
> 1$. We take the $\nu_{m,\ell}$ to be the valuation $v$ whenever $m > 1  \geq
\ell$, and the identity otherwise.
\end{enumerate}

We focus on the first case (after some additional general
observations), since one can reduce to it anyway via the
equivalence given below in Definition~\ref{equivrel} (which takes
us from the usual algebraic world to the tropical world).

\subsubsection{Properties of tangible elements}

We say that a monoid $(M, \cdot \; )$ is
$\nu$-\textbf{cancellative} if $ac \nucong bc$ implies $a \nucong
b$ for $a,b,c \in M$. Likewise, $M$ is
$\nu$-$\Net$-\textbf{cancellative} if $a^m \nucong b^m$ implies $a
\nucong b$ for $a,b \in M$.

\begin{prop}\label{mult21}  If $ac  \nucong bc$ for $a,b,c$ in an $L$-layered
\domain0 with $s(ac)\in L_+$ finite, then $a \nucong b.$ In
particular, $R_1$ is $\nu$-cancellative unless $1+1 = 1$ in $L$.
\end{prop}
\begin{proof} Otherwise, we may assume $a >_\nu   b$, in which
case $ac \ge _\nu  bc$ by Proposition~\ref{mult2}. If $ac  \nucong
bc$, then $s(ac) = s((a+b)c) = s(ac + bc) > s(ac),$ a
contradiction.\end{proof}

\begin{cor}\label{infinite} If $R$ is $\nu$-bipotent and $R_1$ has only finitely many elements, then either $1\in L$ is infinite
(i.e., $1+1 = 1$), or the $\nu$-relation is trivial on $R_1$ in
the sense that $a \nucong \rone$ for all $a \in R_1$.\end{cor}
\begin{proof} Otherwise, $R_1$ has an element $a$ with a different $\nu$-value from $\rzero$ or $\rone,$
and clearly  $a^i =  a^j$ for some $i>j.$ Then $a^ i + a^j$ is not
tangible. But $a^ i + a^j = a^j ( a^{i-j} + \rone),$ which is
tangible unless $\rone \nucong  a^{i-j},$ which is false by
hypothesis.\end{proof}

In brief, the only finite layered \domains0 are $\nu$-trivial, and
thus are not uniform (for $L \ne \{ 1 \})$. (One can get finite
structures by means of a 0 layer, as indicated in
Example~\ref{trun0001} below.)


 Here is an  indication of the importance of tangible
elements.

\begin{lem}\label{invelt} If $L = L_{\ge 1}$ (for
example, if $L = \Net$), then every invertible element $a$ of $R$
is tangible. \end{lem} \begin{proof} $ \lone =  s(aa ^{-1})  =
s(a)s(a^{-1}),$ implying $s(a^{-1}) = s(a)^{-1}.$ But $s(a^{-1})
\ge 1,$ so
$$\lone  \ \le \  s(a) s(a^{-1}) \ \le \ \ s(\rone) \ = \ \lone ,$$
implying $s(a)= s(a^{-1}) = \lone.$ \end{proof}

\subsubsection{The case when the multiplicative monoid of $L$ is a
 group}

On the other hand, it often is convenient to make the different
assumption that $L$ is a group. In such a situation,
Lemma~\ref{invelt} fails, and the flavor of the layered theory
differs considerably from the standard supertropical theory. But
here we can
 describe the structure of $R$ in terms of~$R_1,$
 its set of tangible elements.

\begin{prop}\label{gp} If an $\Lz$-quasi-layered \domain0  $R$ is a
\semifield0 (i.e., is a group under multiplication), then $L$ is a
multiplicative group (and thus also a \semifield0) and also $R$ is
an $\Lz$-quasi-layered $1$-\semifield0. 
\end{prop}
\begin{proof} If $R$ is a
\semifield0,  we must have $s(a^{-1}) = s(a)^{-1},$ and for $a \in
R_1$ we have $a^{-1} \in R_1.$
\end{proof}

\begin{prop}\label{removesort} If $L$ is a
multiplicative group, then one can define new sort transition maps
$\nu'_{\ell,k}$ by  $\nu'_{\ell,k}(a_k) = e_m a_k$ where $\ell = m
k$.
\end{prop}
\begin{proof} If $\ell = k$ there is nothing to prove, so we assume $\ell > k $ are non-negative, and write $\ell = k + p$ for $p \in L_+.$ Then $\ell =
k(1 + pk^{-1}),$ and $m = 1 + pk^{-1}.$ Now  $e_m a_k \in R_\ell$,
and $$\nu'_{\ell,k}(a_k) = e_m a_k \nucong a_k \nucong
\nu_{\ell,k}(a_k).$$
\end{proof}

Thus the sort transition maps can be replaced by multiplication by
the $ e_m,$ and we can remove the sort transition maps from the
definition.

\subsection{Ghost layers}

 Next, we want to extend the notion of ghosts from the standard supertropical
 situation
to the context of the layered structure.

\begin{defn}\label{ghostsurp}  An element   $b \in R$ is an
$\ell$\textbf{-ghost} (for given $\ell \in L$) if $s(b)\in L$ is
an $\ell$-ghost sort, i.e., if  $s(b) = \ell + k$ for some $k \in
 L_+$.  A \textbf{ghost element} of $R$ is a 1-ghost.
\end{defn}

In other words, for any $\ell$-ghost  $b$, either $s(b) > \ell$
for $\ell$ finite, or $s(b) = \ell$ is infinite. This explains the
difference between tangible and ghost sorts. Namely, when $b \in
R_\ell,$ $b$ is an $\ell$-ghost iff $\ell$ is infinite.

\begin{defn}\label{defn4}  
We recall $\FR_{\ge 1}$ from Remark~\ref{supertropst}(iii). In
analogy to valuation theory, we call $\FR_{\ge 1}$  the
\textbf{layered valuation \semiring0} of $R$, and we call $
\FR_{>1}$ the \textbf{ghost ideal}.\end{defn}

When $L =L_{\ge 1}$, the ghost valuation ideal replaces the ghost
ideal of the standard supertropical theory, so the different ghost
layers are to be thought of as a refinement of the ghost ideal.

On the other hand, when $ L$ contains the infinite element
$\infty$, then $R_\infty$ is an idempotent \semiring0 and there is
a \semiring0 homomorphism $\nu: R \to R_\infty$ sending $a \mapsto
\nu_{\infty,\ell    }(a)$ where $s(a) =\ell,$ and is the identity
map on $R_\infty.$ In this way, $R$
 covers the idempotent \semiring0 $R_\infty$, which can be
viewed as the ``absolute'' ghost layer, as described above in
Remark~\ref{directlim}.

Likewise, when $L =L_{\ge 1}$, there is a \semiring0 homomorphism
$\nu: R \to R_1 \cup R_\infty$ sending $a \mapsto \nu_{\infty,\ell
}(a)$ where $s(a) =\ell>1,$ and  $R$ also covers the standard
supertropical \semiring0 $R_1 \cup R_\infty$.

These roles complement each other. In general, the finite elements
$\ell \in L$ greater than $1$ can be viewed as ``sorting''
different ``layers" $R_\ell$ of ghosts.

Note that two partial pre-orders are at play -- the given partial
pre-order on the sorting set $L$, and the $\nu$-pre-order on $R$.
There is a subtle interplay between the two. It turns out that we
could develop the theory under the weaker condition that $L$ is a
partially pre-ordered multiplicative monoid, and we sketch the
appropriate changes in~Appendix~B.


\subsection{Adjoining the absolute ghost layer, and the passage to standard supertropical
\domains0}\label{adjoininf}

 Even when $L$ originally does not contain an infinite element a priori,   $L$-layered \domains0 tie in directly with the
(standard) supertropical theory, via a  ghost layer introduced at
a new element $\infty$ which we adjoin. (This works even when
$(\ge)$ is merely a partial order on $L$, although it it is easier
when $(\ge)$ is a total order.)

\begin{rem}\label{directlim} The system $\ldR$ is a directed system
with respect to the set $L$, as described in  \cite[p.~71]{J}.
Hence, by \cite[Theorem 2.8]{J}, the layers $R_k$ have a direct
limit which we denote $R_{\infty},$ and maps $$\nu_{\infty,k} :
R_k \to R_{\infty}$$ such that $\nu_{\infty,k} = \nu_{\infty,\ell}
\circ \nu_{\ell,k}$ for each $a \in R_k$ and all $k<\ell$. Since
$R = \bigcup_k R_k,$ we can piece together these maps
$\nu_{\infty,k}$ to a map $\nu: R \to R_{\infty}.$ We define
\begin{equation}\label{1} e = e_\infty :=\nu(\rone),\end{equation}
easily seen to be the unit element of  $R_{\infty}.$

We write $a^\nu$ for $\nu(a) \in R_\infty$. Thus $a^\nu  = b^\nu$
iff $a \nucong b$ in our previous notation.
\end{rem}

We call $R_\infty$ the \textbf{absolute ghost layer} and  $\nu$
the (absolute) \textbf{ghost map} of $\R.$ Note that in the
uniform case, $R_\infty$ is just another copy of $R_1,$ so
everything is much simpler.

\begin{thm}\label{Udef} Suppose $\R = \ldR$ is an $L$-layered \domain0. Then the absolute ghost layer   $R_{\infty}$ is
a bipotent \semiring0. The ghost map $\nu: R \to R_{\infty} $ is a
\semiring0\ homomorphism. Define $$U = U(R): =  R \ \dot\cup\ \
R_{\infty}.$$ Then $U$ is a \semiring0 under the given operations
of $R$ and $R_\infty,$ together with
$$
\begin{array}{rll}
a\cdot b^\nu & := & (ab)^\nu; \\[2mm]
 a +b^\nu & := & \begin{cases} a&\quad\text{if \ $ea>eb$,}\\
b^\nu &\quad\text{if \  $ea\le eb$.} \end{cases}\end{array}
 $$
 Also, extend $\nu$ to a map $\nu_U : U \to R_{\infty}$ by taking
$\nu_U$ to be the identity on $R_{\infty}.$ Then $U$ has ghost
ideal $\tG = \tG(U) := R_{\infty},$ in the sense of
\cite{IzhakianRowen2007SuperTropical}, and $\nu_U(a) = ea$ for
every $a$ in $U .$

Then $U$  can be modified to a supertropical \semiring0 $$\cR
_{1,\infty} := R_1 \ \dot\cup \ \tG,$$ retaining the given
multiplication $\cdot$ of $U,$ but with  new addition $\oplus$
given by the rules
\begin{equation}\label{6}
a\oplus b:=\begin{cases} a&\quad\text{if \ $ea>eb$,}\\
b &\quad\text{if \ $ea<eb$,}\\
ea &\quad\text{if } \  ea=eb. \end{cases}\end{equation}

\end{thm}


\begin{proof}
Axiom A3  tells us that
$$\nu_{mm',k\ell}(a \cdot b)=\nu_{m,k}(a) \cdot \nu_{m',\ell}(b)$$
for any $a\in R_k$ and $b\in\R_\ell;$ taking limits yields $$\nu(a
\cdot b)=\nu(a) \cdot \nu(b).$$ Likewise, Axiom  B tells us that
$$\nu(a+b)=\nu(a)+\nu(b).$$
 The other verifications are also easy.  By
\eqref{1} we have
\begin{equation*}\label{2}
\nu(x)=e \cdot x\quad\text{for every}\quad x\in\R.\end{equation*}
Thus $\nu\circ\nu=\nu$, and also $\nu: \R\to \tG$ is a \semiring0\
homomorphism from $\R$ onto $\tG = \tG(U).$

  We  extend the $\nu$-equivalence relation from
$R$ to $U$ by decreeing that  $a \equiv_U  b$  iff $a$ and $b$
have the same value under $\nu.$

We turn to the last assertion. Due to \eqref{6} we have
\begin{equation*}\label{7}
a\oplus b=a+b\quad \text{if}\quad  a\not \nucong b.\end{equation*}
On the other hand,
\begin{equation*}\label{8}
a\oplus b=e(a+b)\quad\text{if}\quad a\nucong b.\end{equation*}
Note that
\begin{equation*}\label{9}
a\oplus b \nucong a+b \end{equation*} in all cases. Also, $\tG(U)
: = R_\infty  = \tG(\cR _{1,\infty}).$ \end{proof}


 We may regard $\cR _{1,\infty} : =
(\cR _{1,\infty},\oplus,\cdot \, )$ as a degeneration of the
\semiring0 $U = U(R),$ where all the ghost layers have been
coalesced to $R_\infty.$  When   $L = L_{\ge 1}$, then there is a
\semiring0 homomorphism $U \to \cR _{1,\infty}$
given by $$a \mapsto \begin{cases} a \quad\quad \text{ for }\quad a \in R_1 \cup R_\infty,\\
\nu(a)  \quad \text{otherwise}. \end{cases}$$ (This map is a
special case of truncation, to be discussed in \S\ref{trunc1}.)

We are now in a position to see why Construction~\ref{defn5} of
uniform $L$-layered \domain0  is generic. We recall
\begin{equation*}\label{10} \R(L,\tG) :=\{\xl{a}{\ell} \bigm| a\in\tG,\ \ell\in L\}.\end{equation*}

\begin{thm}\label{B1new}  Suppose $R'$ is an $L$-layered \domain0 and $\tG = R'_{
\infty}$. There is a \semiring0 homomorphism $$\Phi: R' \to
\R(L,\tG)$$ given by $a\mapsto \xl{(\nu(a))}{s(a)}  $. If $R'$ is
a uniform $L$-layered \domain0, then  $\tG \cong R_1$ and $\Phi$
is an isomorphism.
\end{thm}
\begin{proof}  Clearly $\one_{\R(L,\tG)} =\xl{\one_\tG}{1},$
\begin{equation*}\label{11}
\nu_{\ell,k}(\xl{a}{k})=\xl{a}{\ell} \quad\text{if}\quad  k\le
\ell,\end{equation*} and $\xl{a}{\ell} =\xl{b}{k}$ iff $a\nucong
b$ and $k=\ell.$

We read off from the definition
that %
$\Phi$ preserves multiplication and  addition.

In case the sort transition maps are bijective,
 the direct limit construction degenerates to the identity, implying $\tG \cong R_1$. The given map $\Phi$ is clearly 1:1, since $\Phi (a) = \Phi (b)$ would imply
 $a \nucong b$ and $s(a) = s(b)$, implying $a=b$. But
  given $\xl{a}{\ell} \in R(L,\tG),$ there is $x \in R'$ with $\nu(x) = a;$
  taking $k = s(x)$ we have some $m>k,\ell$ in $L$. Write $x_\ell = {\nu_{m,\ell}}^{-1}(\nu_{m,k}(x))$. Then
$$\Phi(x_\ell) = \xl{a}{\ell}, $$ proving
$\Phi$ is onto, and thus an isomorphism.
\end{proof}


In other words, up to \semiring0\ isomorphism, $R := \R(L,\tG)$ is
the unique uniform $L$-layered \domain0~$\R$ for which
$\R_\iy=\tG$. 

\subsection{The case where $L$ is a semiring, and the role of the 0-layer}\label{adj00}

To avoid dealing with exceptional cases in the sorting set $L$, it
has been more convenient to work in the language of \semirings0.
In this subsection we deal with a zero layer, i.e., assume $0 \in
L.$ At the same time, we take the opportunity to fit the zero
element of $R$ (if it exists) into the theory. Our treatment here
is brief, with our main intent being to introduce
Example~\ref{trun0001}. We study this situation in considerably
greater detail in
 \cite{IKR4}.

New intricacies arise when $0 \in L$ and $R_0 \ne \emptyset$. One
tricky question involving $\zero_R$ (if it exists) is: Where does
it reside? In particular, is it tangible or ghost?

\begin{lem}\label{esot}
 The layer $R_0$ is also
 an ideal of~$R$.
If furthermore $R$ is a semiring,  then $\rzero \in R_0$.
\end{lem}\begin{proof} The first assertion is clear. Suppose $\rzero \in R_k.$  Then for any $a \in
R_0$ we have
$$\rzero  = \rzero \cdot a \in R_{k\cdot 0} = R_0.$$ \end{proof}

 On the
other hand,   $\zero_R  $ also acts like a ghost, since it absorbs
all elements.  Let us formalize this observation.

\begin{defn}\label{defn2} Suppose  $(\Lz, \ge)$ is a partially pre-ordered
semiring (with zero element $0$).
An $\Lz$-\textbf{quasi-layered \semiring0}   is a \semiring0 $\R
$, together with a family $\{ R_\ell : \ell \in L\}$ of disjoint
subsets $R_\ell \subset R$,  satisfying precisely the same
conditions as in Definition~\ref{defn1}, with the
%
%
following exceptions, to take into account the special role of the
0 layer:

\begin{enumerate} \ealph

\item There are sort transition maps $\nu_{0,\ell}$ for every
$\ell \in L.$ \pSkip

 \item Axiom A2 is replaced by the following axiom:

\boxtext{
\begin{enumerate}
 \item[A2$_0$] If $a\in \R _k$ and  $b\in
\R _\ell,$ then $ab\in \R _{k \ell } \cup R_0$. \pSkip
 \end{enumerate}}

 \item In Axiom~A3, the product in $\R $ is also compatible with the sort transition
maps  $\nu_{0,\ell}$: \\ If $a\in \R _k,$ $b\in \R _{\ell}$ with
$m\ge k$,   then
$\nu_{0,k\ell}(ab)=\nu_{m,k}(a)\cdot\nu_{0,\ell}(b). \qquad $
\pSkip

 \end{enumerate}
\end{defn}

\begin{rem} Since $R_1$ no longer turns out to be a monoid, we must often consider the
\textbf{fundamental submonoid} $ R_0 \cup R_1$, whereas $ R_0 \cup
R_\infty$ is an ideal of $R$. This formalizes the notion that the
``absorbing layer'' $R_0$ is both tangible and ghost.

Since $\rone \in R_1,$ every invertible element of the fundamental
submonoid must lie in $R_1$. In particular, if (excluding the
element $\zero $) the fundamental submonoid is a multiplicative
group, then it must be $R_1$. This reduces us to the layered
$1$-\semifield0 case described earlier.
\end{rem}

Here is our main example.

 \begin{exampl}\label{moveideal} Given any ideal $\mfa$ of an $L$-layered \domain0 $R$, we formally define
$R_\mfa$ to be  $R$  with  the same \semiring0 operations, and to
have  the same sort function as $R$, except that now $s(a) = 0$
for every $a \in \mfa.$  In other words,
$$(R_\mfa)_0 : = R_0 \cup \mfa; \qquad (R_\mfa)_\ell := R_\ell
\setminus (\mfa \cap R_\ell ).$$
\end{exampl}

 \begin{prop}\label{zerolay}  $R_\mfa$ is an $L$-layered \semiring0.\end{prop}  \begin{proof} We need to check associativity and distributivity. But
this is clear unless we are using elements of~$\mfa,$ and then
associativity holds because all products have layer 0. Likewise,
to see that $a(b+c)$ and $ab + ac$ have the same layer, note this
is clear if $s(a) = 0$ or if $s(b+c) \ne 0.$ Thus we may assume
that  $s(b+c) =0$, and again we are done if $s(b) = s(c) = 0,$ so
we may assume that $s(b) =0$ and $s(c)\ne  0$ with $ b >_\nu c$
but $ab \nucong ac.$ But then
$$s(a(b+c)) = s(ab) + s(ac) = s(ab),$$
so $a(b+c) = ab = ab+ac.$\end{proof}

\begin{rem}\label{levelzero} $ $ \begin{enumerate} \eroman
 \item Taking the degenerate case $\mfa = \emptyset$ and $R_0 =
 \emptyset$ reduces to the original quasi-layered \domain0
 (Definition~\ref{defn1}). \pSkip

 \item More generally, if $R$ is an $\Lz$-\textbf{quasi-layered
 \semiring0} and $R_0$ is a prime ideal of $R$, then $R\setminus
 R_0$ is an $L\setminus \{0\}$-quasi-layered \domain0.
 \end{enumerate}

 In this way, we see that Definition~\ref{defn2} encompasses
 Definition~\ref{defn1}.
\end{rem}

\begin{rem}\label{finitear} If $R \setminus \mfa$ is finite, then  $(R_\mfa)_1$ is
a finite set. Thus, we have a way of ``shrinking'' the tangible
component to a finite set.
\end{rem}

 One  instance  of arithmetic
significance is when $R  = R(L,\Net \cup \{0\})$ where $L$ is
finite, and $\mfa = \{ \xl{n}{\ell} : n
> q , \ell \in L\}$ for some $q \in \Net.$ In this case, we can ``contract'' $\mfa$ to a single
element in $R_0$.

\begin{exampl}[The layered truncated \semiring0]\label{trun0001} Take $L = \Net,$ $R = R(L,\Real),$ and fixing $q
>0$ in $L$, define $\mfa = L \times \{ n: n\ge q\} \triangleleft R$.
Then $R_\mfa$ contracts to the $L$-layered \semiring0 $$ \{
\xl{a}{k}: k\in L, \ a \in \{1, \dots, q-1\}\} \cup \{ \xl{q}{0}
\}, $$ where addition is defined as in Construction~\ref{defn5},
and multiplication $ \xl{a}{k} \xl{b}{\ell} $ is given as in
Equation~\eqref{13} except for $ab \ge q$, in which case
$\xl{a}{k} \xl{b}{\ell} =
  \xl{q}{0}$ for any $k,\ell \in L.$ Addition is given by
$$\xl{a}{k} +   \xl{q}{0} =   \xl{q}{0}.$$

The sort transition maps are as in  Construction~\ref{defn5},
except that we define $\nu_{0, k}(\xl{a}{k}) =   \xl{q}{0}$ for
all $\xl{a}{k}$. Thus, $\xl{q}{0}$ is the special infinite
element, and the sort transition map $\nu_{0, k}$ is not 1:1. In
this example, $R_1 \cup \{\xl{q}{0}\}$ should be viewed as the
fundamental submonoid.

When instead the layering \semiring0 $L$ is finite, we see that
$R_1 \cup \{\xl{q}{0}\}$ is a finite set, which merits further
study using arithmetic techniques.
\end{exampl}

\subsubsection{Adjoining the 0-layer}\label{adj0}

 Starting with an $L$-quasi-layered \domain0 $R$ with respect to a \semiring0 $L$, we can adjoin
a zero layer $R_0$ formally in several ways. The first way is
simply by adjoining a zero element to~$R$.

\begin{rem}\label{adjoin00}

For any quasi-layered \domain0 $R$  with respect to a \semiring0
${\Lz}$, the  semiring
$$\R  \ \dot\cup\  \{ \rzero\}$$ can be
 layered with respect to the semiring $$  \zL := \Lz \ \dot\cup\
\{ \lzero \},$$ 
where we take   $R_0 :=\{\rzero \} $, putting it in the zero layer
as seen by applying the argument of Proposition~\ref{zerolay}. We
take the sort transition maps $\nu _{0, \ell} (a): = \rzero$ for
all $\ell \ne 0$ and $a \in R.$
\end{rem}

However, this is not the only possibility for the zero layer, as
we saw in Remark~\ref{getzero}.

\begin{prop}\label{wholelay} If $R$ is a uniform $L$-layered \domain0, where $L$ is
a \semiring0,  then adjoining $\{ 0 \}$ formally to~ $L$ as the
unique minimal element,  we can form a uniform $\zL $-layered
\domain0 $R \cup R_0,$ where $R_0 := e_0 R_1$ is another copy of
$R_1,$ under the same rules of addition and multiplication given
by Proposition~\ref{unif1}, and the sorting maps $\nu _{0,k}$ are
the natural identifications.
\end{prop}
\begin{proof} If $a = e_0 a_1$,  $b= e_k b_1,$ and $c = e_\ell c_1  $ for $a_1, b_1, c_1 \in R_1,$ then
$$(ab)c  = e_0 e_k e_\ell (a_1b_1) c_1  = e_0 e_k e_\ell a_1(b_1 c_1) = e_0   a_1(b_1c_1)  = e_0 a_1(b_1 c_1) =
a(bc),$$ yielding associativity of multiplication. To see
distributivity, we note that $ e_k b_1+ e_\ell c_1 = e_m (b_1 +
c_1)$ where $m \in \{k, \ell, k+ \ell\},$ so
$$a(b+c) = e_0 e_m  a_1(b_1+c_1) = e_0
a_1(b_1+c_1)  = e_0 a_1b_1 + e_0 a_1c_1 = e_0 e_k a_1b_1 + e_0
e_\ell a_1c_1 = ab+ac.$$

Associativity of addition is similar. Finally, if $a = \rzero  \in
R_0$ and $b\in R_\ell$, then    $ab \in R_{0 \cdot \ell}= R_0.$
\end{proof}


To indicate the richness of this approach, let us see how we can
remove the restriction in Construction~\ref{defn5} that the monoid
$\tG$ is cancellative, using the idea of Example~\ref{moveideal}.

\begin{construction}\label{defn50} Suppose $L$ is a   semiring, and we are given
an
 ordered monoid  $\tG$, viewed as a  \semiring0\ as in Remark~\ref{monsem}.
 We say an element in $\tG$ is \textbf{non-cancellative} if it has
 the form $ab= ac$ where $b\ne c.$ The set of non-cancellative
 elements comprises a monoid ideal $\mfa$ of $\tG$, since if $ab=
 ac\in \mfa$ then, for any $d \in \tG,$ $abd= acd,$ implying $abd \in \mfa.$ Define the \semiring0
$R'$ to be set-theoretically $\((L\setminus \{0\}) \times (\tG
\setminus \mfa )\)\cup \(\{0 \}\times \mfa\)$, where again we
denote the element $(\ell,a)$ as $\xl{a}{\ell}$ and define
multiplication and addition componentwise, according to the rules
of Construction~\ref{defn5}, but with
\begin{equation}\label{compe0}  \xl{a}{k}\xl{b}{\ell} = \xl{ab}{0}  \end{equation}
whenever $ab\in \mfa.$

The same idea is applicable when we adjoin the  0-layer.
 Suppose $L$ is a   \semiring0. Adjoin 0 to $L$ to get $$\zL = L \cup \{ 0 \}.$$ Again we mimic
 Construction~\ref{defn5} and define the \semiring0
$\R(L',\tG)$ to be set-theoretically $$(L \times (\tG \setminus
\mfa ))\cup (\{0 \}\times \mfa),$$ where again we denote the
element $(\ell,a)$ as $\xl{a}{\ell}$ and define multiplication and
addition componentwise, according to the rules of the previous
paragraph, including \eqref{compe0}.

%
%
\end{construction}

%
%

\begin{prop} The layered \semiring0 of Construction~\ref{defn50} is a
\semiring0. If $\tG$ has an absorbing element~$\zero,$ then
$\xl{\zero}{0}$ is the zero element of $R$.
\end{prop}
\begin{proof} Exactly as in the proof of
Proposition~\ref{zerolay}, since we need only check the zero
layer.
Note that $\zero$ is a non-cancellative element of $\tG$, implying
$\xl{\zero}{0}\in R$, and it is clearly the zero element.
\end{proof}

 Since we  have several ways of adjoining a zero layer, the
following observation is useful.

\begin{prop}    For any semiring $R$ layered with respect to a
\semiring0 ${L} ,$   $ \R \ds{\dot\cup} \{ \rzero\}$ is an
$\zL$-layered sub-semiring of $R \ds{\dot\cup} R_0.$

More generally, for any ideal $\mfa$ of $R$, writing $\mfa_0 $ for
$\mfa \cap R_0 $, we have $(\bigcup_{\ell \ne 0}\ \R_\ell)
\ds{\dot\cup} \mfa_0$ is an $\zL$-layered sub-semiring of $R
\ds{\dot\cup} R_0.$
\end{prop}

\begin{proof} If $a   \in
\mfa$ and $b\in R_\ell$, then    $ab  \in R_{0 \cdot \ell}= R_0,$
implying $ab \in \mfa_0$.
\end{proof}

This gives rise to the question of whether we should adjoin the
entire 0-layer, or just $\rzero?$ Although one's experience from
classical algebra might lead one to adjoin only $\rzero,$ there
are situations in which one might need other elements in $R_0$ in
order to distinguish polynomials, as illustrated below in
Example~\ref{zerlay}.

\subsubsection{The 0-layer versus the $\infty$-layer}

So far we have discussed two layers that in our present context
could be extraneous, the $0$-layer and the  $\infty$-layer. These
two layers act similarly, since both $0$ and $\infty$ are
absorbing elements of $L$, except that $0$ also absorbs $\infty$
in the sense that $0 \cdot \infty = 0.$ (Of course, the difference
in the tropical theory is that $0 + a = a$ whereas $\infty +a =
\infty$, but often their multiplicative properties are more
significant.) Thus, in case $\infty \in L$ but $0 \notin L$,
$R_\infty$ is an ideal of $R$ that can often be used to replace
$R_0$ in the above discussion. One instance is given in Appendix
B.

\subsection{The case of onto sort transition maps}


Remark~\ref{Krems} leads us to a key simplification for layered
\domains0 when the sort transition maps are onto, which enables us
to reduce many results to the tangible case:
 \begin{lem}\label{Krems1} If $R$ is an $L$-layered \domain0 and $a
 \in R_\ell$ with $\nu_{\ell,1}$ onto, then $a = e_\ell a_1$ for some $a_1 \in R_1$.
\end{lem}
\begin{proof}
Taking $a_1 \in R_1$ for which $\nu_{\ell,1}(a_1) = a,$ we have $a
= \nu_{\ell,1} (a_1) =  e_\ell a_1$ by Remark~\ref{Krems}.
\end{proof}

\begin{Note}\label{uniforme} Lemma~\ref{Krems1} enables us to
simplify the theory for any layer $\ell>1$ for which
$\nu_{\ell,1}$ is onto. When $\ell < 1$ we could go in the
opposite direction, and define $e_\ell$ such that
$\nu_{1,\ell}(e_\ell) = \rone.$ This will be well-defined when
$\nu_{1,\ell}$ is 1:1 since, writing $\ell = \frac mn$ for any
$a\in R_\ell$ with $\nu_{1,\ell}(a) = \rone,$ we have
\begin{equation}\label{compe}  n e_{\ell}  = n e_{m/n}  =
 e_{m} = \nu_{m,\ell}(a) =  n a,  \end{equation} implying $a =  e_{\ell}.$

Although  our examples often have $L = L_{\ge 1},$ we use this
procedure when the occasion arises.
\end{Note}


\subsection{Uniform $L$-Layered \domains0}\label{sec:1to1}

Let us return to the most important case, that of uniform
$L$-Layered \domains0, the main example being $\R(L,\tG)$ of~
Construction~\ref{defn5}.
 Let us see how the layered theory
simplifies for uniform $L$-layered \domains0, enabling us to
eliminate the sort transition maps $\nu_{\ell,k}$ from the
picture.

\begin{rem}\label{tang100} In a uniform $\Lz$-layered \domain0, we can define $\nu_{k,\ell}$ for $k< \ell$ to be
$\nu_{\ell,k}^{-1}.$ Thus, $\nu_{k,\ell}$ is defined for all
$k,\ell \in L.$\end{rem}
\begin{lem}\label{tang11}
Any element $a \in R_\ell$ can be written uniquely as $e_\ell a_1
= \nu _{\ell,1} (a_1)$ for $a_1 \in R_1.$
\end{lem}
\begin{proof} The existence of $a_1$ follows from
Lemma~\ref{Krems1}, and uniqueness is clear since $\nu _{\ell,1}$
is presumed to be~1:1. The last assertion follows from Axiom A3.
\end{proof}

\begin{prop}\label{tang111} Axiom A2 can be replaced by the
following axiom:

\boxtext{
\begin{enumerate}
 \item[A2$'$.] If $a = e_k a_1\in \R _k$ and  $b = e_\ell b_1\in
\R _\ell,$ for $a_1,b_1 \in R_1,$ then $ab = (a_1b_1)e_{k\ell}.$
\end{enumerate}}
\end{prop}

\begin{prop}\label{equalnu}  In a uniform $\Lz$-layered \domain0, if $a
 \nucong b$ for $a \in  R_k$ and $b\in R_\ell$ then $b = \nu_{\ell,k}(a).$ In particular, if $a
 \nucong b$ for $a,b \in R_\ell$, then $a=b.$
\end{prop}
\begin{proof}  An immediate application of Lemma~\ref{tang11}.
\end{proof}

\begin{cor}\label{tang101}  In a uniform $\Lz$-layered \domain0, the transition map $\nu_{m,\ell}$ is given by $e_\ell a_1 \mapsto e_m a_1$.
\end{cor}

Now we can remove the sort transition maps from the definition,
when we write $R = \bigcup_{\ell \in L}\ e_\ell R_1.$

\begin{cor}\label{tang102} Suppose $R$ is a uniform $L$-Layered \domain0. Defining $\nu_{m,\ell}$ as in Corollary~\ref{tang101}, we see
 that Axiom A3 is equivalent
to the following axiom: \boxtext{
 \begin{enumerate}   \item[A3$'$.] $e_\ell e_k = e_{\ell k}$ for all $k, \ell \in L.$
\end{enumerate}}
 Furthermore,  Axiom A4 now is equivalent
to  Axiom B, which we can reformulate as: \boxtext{
\begin{enumerate}
    \item[B$'$.] If $a  = e_k a_1$ and $b = e_\ell a_1$ (so that $ a \nucong b$), then $a+b =
e_{k+\ell} a_1.$
\end{enumerate}}
\end{cor}
\begin{proof} The first assertion follows from the observation that $e_\ell a_1
e_k b_1 = e_{\ell k} (a_1b_1)$; when $a_1, b_1 \in R_1$ then
$a_1b_1 \in R_1.$

For the last assertion,  apply Lemma~\ref{tang100} to Proposition
\ref{equalnu}.
\end{proof}

 Note that $\nu$-bipotence and Axiom B$'$ could then be used as the definition for
addition in $R$, and we summarize our reductions:

\begin{prop}\label{unif1} A uniform $\Lz$-layered \domain0 can be described as
the \semiring0 $$R := \dot\bigcup _{\ell \in L} R_\ell,$$ where
 each $R_\ell = e_\ell R_1$, $(R_1,\cdot \; )$ is a monoid, there is a 1:1 correspondence
 $R_1 \to R_\ell$ given by $a \mapsto e_\ell a$ for each $a \in R_1,$ and
operations are given by Axioms A2$\,'$, A3$\,'$, B$\,'$, and
$\nu$-bipotence.\end{prop}

We have effectively identified any arbitrary uniform $\Lz$-layered
\domain0 with Construction~\ref{defn5}, as will be seen more
precisely in Theorem~\ref{B1new} and Proposition~\ref{111}.



\subsection{Reduction to the uniform case}

In one sense, we can reduce the general case of an $L$-layered
\domain0\ ~ $R$ to the uniform case. First we cut down on
superfluous elements. Note that if $\nu_{k,1}$ are onto for all $k
\ge 1$, then all the $\nu_{\ell,k}$ are onto for all $\ell \ge k$.
Indeed, if $a \in R_\ell$ then writing $a = \nu_{\ell,1}(a_1)$ we
have
$$a = \nu_{\ell,k}(\nu_{ k,1}(a_1)).$$

\begin{rem} Suppose $L = L _{\ge 1}.$ For any $L$-layered \domain0 $\R := \ldR,$
if we replace $R_\ell $ by $\nu_{ \ell ,1}(R_1)$ for each $\ell
\in L$,  we get an $L$-layered \domain0 for which all the
$\nu_{m,\ell}$ are onto. \end{rem}

Having reduced many situations to the case for which all the
$\nu_{m,\ell}$ are onto, we can get a uniform $L$-layered \domain0
by specifying when two elements are ``interchangeable'' in the
algebraic structure.

\begin{defn}\label{equivrel} Define the equivalence relation
$$a \equiv b \quad \text{ when } \quad  s(a) = s(b) \  \text{and} \   a\nucong b.$$\end{defn}

In view of Proposition~\ref{equalnu}, this relation is trivial in
case $R$ is a uniform $L$-layered \domain0.

\begin{prop}\label{equivlem1}  The binary relation $<_\nu$ on an $L$-layered
\domain0 $R$ induces a pre-order on the equivalence classes
$R/_\equiv$. Furthermore, if $a \equiv b$,  then  $ac \equiv bc$
and $a+c \equiv b+c$ for all $c\in R.$\end{prop}
\begin{proof} The first assertion is immediate. For the second
assertion,  $s(ac) = s(a)s(c) = s(b)s(c) = s(bc)$ and $ac\nucong
bc$, proving $ac \equiv bc$.

Next, we consider addition. If $a
>_\nu c$, then $$a+c = a \equiv b = b+c.$$ If $a <_\nu c$, then $a+c = c = b+c.$ If  $a \nucong
c$, then $$s(a+c) = s(a) + s(c) =  s(b) + s(c) = s(b+c),$$ and
$a+c\nucong a  \nucong b \nucong b+c$.
\end{proof}

\begin{rem}\label{11red} When the transition maps  $\nu_{ \ell,k}$ are onto, one can reduce to uniform $L$-layered \domains0, by means of the equivalence
relation $\equiv$ of Definition~\ref{equivrel}, since any
$\nu$-equivalent elements having the same sort are identified.
Then Proposition~\ref{equivlem1} shows that $R/\! _\equiv$ is an
$L$-layered \domain0, under the natural induced layering, and the
transition maps on $R/\! _\equiv$ clearly are bijective.
\end{rem}



\subsection{The $\ell$-surpassing relation}\label{surp1}

One of the key features of the supertropical theory is the use of
the antisymmetric ``ghost surpassing relation,'' given by
$$a \lmodg b \quad \text{when } \
\begin{cases}
  a =b  \\
  \text{or} \\
  a = b + \operatorname{ghost},
\end{cases}$$
 which replace equality and
provides analogs of identities  of classical algebra. We need to
extend that relation to
  $L$-layered \domains0, usually with respect to a certain finite positive layer $\ell$.
  (In the standard supertropical theory, $\ell = 1.$) There are
  two ways of extending ghost layers to an
  arbitrary sorting \semiring0~$L$.

\begin{lem}\label{stg}  $ $
\begin{enumerate} \eroman
    \item If $b_1,b_2$ are $\ell$-ghost, then so is $b_1+b_2$. \pSkip
 \item If $b$ is $k$-ghost   and $c \in R_\ell$ with $\ell\in L_+$, then $bc$ is $k \ell$-ghost.
\end{enumerate}
\end{lem}
\begin{proof} (i): The assertion is clear unless $b_1 \nucong b_2$
with $s(b_1)= s(b_2) = \ell,$ and with $\ell + \ell = \ell$. But
then $\ell$ is infinite and $s(b_1+b_2) = 2 \ell$. \pSkip

(ii): If $s(b)> k,$ then $s(bc)  = s(b)\ell > k \ell,$ so we are
done unless $s(b) = k$ with $k+ p = k$ for some positive $p\in
L_+$. Then $k\ell + p\ell  = k\ell,$ and $p\ell \in L_+,$ implying
$k\ell$ is infinite, so again $bc$ is  $k\ell$ -ghost.
\end{proof}




\begin{defn}\label{lsurp}
 The \textbf{$\ell$-surpassing relation}
$\lmodWl$ is given by
 \begin{equation}  a \lmodWl b \
\text{  iff either } \ \begin{cases} a=b+c &  \quad
\text{with}\quad a\quad  \ell \text{-ghost},
  \quad \\ a=b,\\  \quad\text{or} \\  a\nucong b & \quad\text{with}\quad
a \quad   \ell \text{-ghost}.\end{cases} \end{equation}
\end{defn}


%

%

\begin{prop}\label{lghost1} If $a \lmodl b$ with $s(a) \le \ell$
 and $\ell$ finite, then $a =b$.
\end{prop}
\begin{proof} We may assume  $a \ne b$, so the only
  condition of Definition~\ref{lsurp} that can hold is
  $a = b+c$ with $ c $ an $\ell$-ghost and $c \nucong a$ (since otherwise $a=b$ and we are
  done).  Thus, $s(a) \ge s(c)$ is an
$\ell$-ghost sort, implying $s(a) = \ell$ which being also an
$\ell$-ghost sort is infinite, a contradiction.
\end{proof}

But the relation $\lmodl$ is only of marginal interest; our main
focus is on the next relation.

\begin{defn}\label{ghostsurpL}   The \textbf{$L$-surpassing relation}
$\lmodWL$ is given by  $a \lmodWL b $ if $a  \lmodWl b$ where
$\ell = s(b).$
\end{defn}

The relation $\lmodWL$ generalizes equality in the following
sense:
\begin{lem}
\label{superlayer11} $ $
\begin{enumerate} \eroman
    \item  If $a \lmodWL  b $ with $s(a) = s(b)$ finite, then $a = b $.

\item If $a \lmodWL  b $ and $b \lmodWL  a ,$ then $a = b $.
\end{enumerate}
\end{lem}
\begin{proof} (i) By Proposition ~\ref{lghost1}. \pSkip

(ii) Since $s(b) \le s(a)$ and $s(a) \le s(b),$ we get  $s(a) =
s(b)$ and conclude by using (i) and Remark~\ref{infdeg}.
\end{proof}

\begin{exampl}\label{Fr11} When $L$ is non-negative, \begin{equation}(a+b)^n \lmodWL a^n +
 b^n.\end{equation}
 (Equality holds unless $a \nucong b,$ in which case we are done
 by Axiom B.)
 \end{exampl}

\begin{prop} In the standard supertropical case ($L = \{1,
\infty\}$), $a \lmodWL b $ iff $a \lmodg b.$  \end{prop}
\begin{proof} We may as well assume that $a \ne b.$ If  $a\nucong b $
with $a$  $\ell \text{-ghost},$ then clearly $a \lmodg b.$ Hence,
we may assume that  $a\not\nucong b $ but $a = b+c$ for $a$
$s(b)$-ghost. But this implies $a >_\nu b,$ so $a = c$, which
shows that $c$ is ghost and $a = c+b,$ again yielding $a \lmodg
b.$
\end{proof}

  Nevertheless, the
flavor of $\lmodWL $ differs from the standard supertropical
theory, as we see in Remark~\ref{essdif} below.

\subsection{The   layered sub-\domain0 generated
by $\rone$}

Our objective here is to provide the layered analog of~$\mathbb F
_1$ of~\cite{CC}.

\begin{lem}\label{inj} If $R$ is an $L$-layered \domain0, then
$$\varepsilon_L(R): = \{e_\ell: \lone \le \ell \in L\}$$ is an $L$-layered
sub-\domain0 of $R$, and is also has the natural $\bbN$-action
given by  $n\cdot e_\ell = e_{n\ell}.$ There is a natural
\semiring0\ isomorphism $L \to \varepsilon_L$ given by $\ell
\mapsto e_\ell.$\end{lem}
\begin{proof} $$e_k + e_\ell = \nu_{k,1}(\rone)+\nu_{\ell,1}(\rone)=\nu_{k+\ell,1}(\rone) = e_{k+\ell}$$
by Axiom A4, and likewise
$$e_k \cdot e_\ell = \nu_{k,1}(\rone)\cdot\nu_{\ell,1}(\rone)=\nu_{k \cdot \ell,1}(\rone \cdot
\rone )=\nu_{k  \ell, 1}(\rone)= e_{k\ell}$$ by Axiom A3; finally,
$\nu_{1,1}(\rone)=\rone,$ and $ \underbrace{e_\ell + \cdots +
e_\ell }_{n \text{ times}}= e_{n \ell}$ for any $n \in \bbN.$
\end{proof}

\subsubsection{Comparison with the idempotent (characteristic 1)
theory}\label{char1}

 Let us now see how this ties in with the
theory espoused in \cite{CC}, which we recall takes a monoid $\tG$
and makes it into an idempotent semiring by means of an operation
which they call $s$, and which we temporarily denote as $\tlds$ to
avoid confusion with our sorting map $s$. Intuitively, $\tlds$
means the operation $``+1''$. Thus, in the max-plus algebra one
would have $\tlds(a) \in \{ a, 1\},$ and $\tlds^2 = \tlds.$

\begin{rem}\label{char2} Suppose $R$ is $\nu$-bipotent. Then  $\tlds(a) \in \{ a, \rone\}$
whenever $a \not \nucong \rone.$ Thus, $\tlds^2(a) = \tlds(a)$
unless  $a
 \nucong \rone.$ This observation yields an interesting parallel to \cite{CC} given in Proposition~\ref{Frob3} below.
 \end{rem}

\begin{rem}\label{charsub} $\varepsilon_L$  of Lemma~\ref{inj} plays the analogous role of $\mathbb F _1$ of
\cite{CC}.\end{rem}


\section{Truncation and layered homomorphisms}

In this section, we discuss a fundamental way to cut down the
sorting \semiring0 $L$ to a more manageable one, usually finite.
(The significance of such a situation is discussed in
Example~\ref{ntrunc1}.) This is put in the context of
homomorphisms between layered \domains0, which  also could serve
as a prelude to a general discussion of morphisms, for which we
give a foretaste.  We continue our basic set-up: $\L$ is a
non-negative \semiring0 equipped with the
 order $(\ge)$, and $\R$ is an $\L$-quasi-layered \domain0.

\subsection{Truncation of the layering \semiring0}\label{trunc1}

  Let $Q$
be an \textbf{upper ideal} in $L;$ i.e., $Q$ is a non-empty ideal
of $L$ having the following property:

\begin{equation}\label{upper} \text{If } \ell\in Q, \text{ then } m\in Q  \text{ for all }
m\ge\ell.\end{equation}

\begin{rem}\label{truncq} Since $L$ is non-negative, the condition \eqref{upper} means $\ell + m \in Q$ for all $m \in L.$
Conversely, when $Q$ satisfies \eqref{upper} and $L = L_{\ge 1},$
then $Q$ is automatically an upper ideal, since $k \ell  \ge \ell$
for all $k\in L$.  There are two case of particular interest,
where we fix $q \in L$:
\begin{enumerate}\eroman
\item $Q = \{ \ell \in L: \ell \ge q \};$ \pSkip  \item $Q = \{ \ell \in
L: \ell > q \}.$
 \end{enumerate}
\end{rem}

  Our objective is to
``mod out'' $Q$ in order to make $q$ the unique largest element
which takes on the role of the infinite element $\infty $ with
respect to $L \setminus Q.$

Towards this end, we   recall the Rees quotient monoid $L/Q$: We
first define an equivalence relation $E_L(Q)$ on $L$ as follows,
writing $\sim_Q$ for $\sim_{E_L(Q)}.$ For $k,\ell \in L,$ we
decree: \begin{description}
    \item[$\text{If } k\in Q$]
    $$k\sim_Q\ell \ \Leftrightarrow\ \ell \in Q.$$
    \item[If $k \notin Q$] $$k\sim_Q\ell \ \Leftrightarrow \ k=\ell .$$
\end{description}
This is well-known to be an equivalence relation on $L$ compatible
with multiplication and  addition, since
$$k\sim_Q\ell \ \Rightarrow\ k\cdot m\sim_Q\ell \cdot m,\quad  k+m\sim_Q\ell +m,$$
for any $k,\ell ,m\in L.$ The equivalence class $[\ell]_Q$ of any
$\ell\in\L$ is $\{\ell\}$ if $\ell\notin Q,$ but is $Q$ if
$\ell\in Q,$ and the set $L/Q:=L/E_L(Q)$ is a \semiring0 under the
rules
\begin{equation}\label{15}
[\ell]_Q\cdot[m]_Q:=[\ell m]_Q, \qquad
[\ell]_Q+[m]_Q:=[\ell+m]_Q.\end{equation}

The equivalence class of $Q$ in $L/Q$ is a single element, which
we write as $q$  in the case of Remark~\ref{truncq}(i), and we
identify a class $[\ell]_Q=\{\ell\}$ with the element $\ell$ of
$L$ if $\ell\in L\setminus Q.$ The original order induces an order
on our \semiring0 $\L/Q$, where the class $q$ is larger than every
other class  (and thus plays the role of $\infty$ in $L/Q)$; the
natural map
$$\pi_Q:\L\to\L/Q,\quad \ell\mapsto[\ell]_Q$$
becomes order preserving (in the weak sense: $\ell\le m\Rightarrow
[\ell]_Q\le[m]_Q).$

 We then have
\begin{equation}\label{16}
L/Q:=\{[\ell]_Q\bigm|\ell\in L\}=(L\setminus
Q)\cup\{q\}.\end{equation} In short, $\L/Q$ arises from $\L$ by
identifying all $\ell\in Q$ with the single element $q.$

(In the case of Remark~\ref{truncq}~(ii)) the equivalence class of
$Q$ in $L/Q$ is a new element which we could denote as $q^+$.)

We also define an equivalence relation $E_{\R}(Q)$ on $\R$ as
follows, writing again $\sim_Q$ instead of $\sim_{E_{\R}(Q)}.$ Let
$x,y\in \R.$ \begin{equation}\label{17}  x\sim_Qy\ \text{ if:}\quad \begin{cases}
 \lv(x)\in
Q,\  \lv(y)\in Q  \text{ with } x \nucong y \\  \text{ or } \\
\lv(x)\notin Q ,� \text{ with }   x = y.
\end{cases}
\end{equation}

\begin{prop} $E_{\R}(Q)$ is compatible with addition and
multiplication, i.e., for all $x,y,z\in\R:$
\begin{equation}\label{19}   x\sim_Qy \
\Rightarrow  \ x+z\sim_Qy+z,\quad
xz\sim_Qyz.\end{equation}\end{prop}
\begin{proof} This is clear if $x =y$, so we may assume that $\lv(x)\in Q$ and thus $\lv(y)\in Q,$ and
we check addition case by case.

\begin{enumerate}
    \item[-] If $x >_\nu z,$ then $x+ z = x \nucong y = y+z.$
    \pSkip

    \item[-] If $x <_\nu z,$ then  $x+ z = z = y+z.$ \pSkip

    \item[-] If $x \nucong z,$ then  $x+ z \nucong z \nucong y+z,$
    implying
$$\lv(x+z)\in Q,\quad \lv(y+z)\in Q$$
\noindent by Axiom B, in view of \eqref{upper}.
\end{enumerate}
 In each case, $x+z\sim_Qy+z.$  One verifies
easily that $x\cdot z\sim_Qy\cdot z.$
\end{proof}

We write $\R/Q$ as shorthand for $\R/E_{\R}(Q)$. It follows that
$\R/Q$ carries the structure of a semiring such that the natural
map
$$\pi_Q^{\R}:\R\to\R/Q,$$ which sends every $x\in\R$ to its
equivalence class, denoted by $[x]_Q,$ is a \semiring0
homomorphism. Moreover, we have a unique map $\bar s:
\R/Q\to\L/Q,$ given by  \begin{equation*}\label{21}
  \bar \lv([x]_Q) =[\lv(x)]_Q,\end{equation*} such that the diagram
\begin{equation}\label{trunc} \xymatrix{\R
 \ar @{->>}[d]
_{\pi_Q^{\R}}\ar[r]^s & \L\ar @{->>}[d]^{\pi_Q}\\
  R/Q\ar[r]_{\bar s} & \L(Q)}\end{equation}
  commutes, and this map $\bar s$ is a \semiring0 homomorphism. Explicitly,  we have for
  $x,y\in\R$:
  \begin{equation}\label{20}
  [x]_Q+[y]_Q=[x+y]_Q,\quad [x]_Q\cdot[y]_Q=[xy]_Q.\end{equation}

  We  write more simply $$\barL:=L/Q, \qquad
  \barR:=\R/Q,$$ and
  $\bar\ell:=[\ell]_Q,$ $\bar x:= [x]_Q$ for $\ell\in L,$ $x\in\R.$
  We want to turn the \semiring0 $\barR$ into an $\barL$-layered
  \domain0. Now the surjective map $\bar s: \barR\to\barL$ already
   partitions  $\barR$ into subsets $$\barR_\lm:=\bar
  s^{-1}(\lm),$$ with $\lm$ running through~$\barL.$ It
  remains to define the sort transition maps
  $$\bar\nu_{\mu,\lm}:\barR_\lm \ds \to\barR_\mu$$ for
  $\lm,\mu\in\barL,$ $\mu\ge\lm.$

  Before we do this, we switch to a notation which we describe in detail for the case
  of Remark~\ref{truncq}(i), that better conveys
  the idea of ``truncation.'' For any
  $\ell\in L\setminus Q$ we identify $\bar\ell$ with the element
  $\ell.$
  \{Recall that $\bar \ell := [\ell]_Q=\{\ell\}.$\}

 Now $\bar\ell=q$ for any $\ell\in
  Q.$
  In this way we view $\barL$ as a subset of $\L\cup \{ q\}$:
  \begin{equation*}\label{22}
  \barL=(\L\setminus Q)\cup\{q\},\end{equation*}
  and then have $\bar\ell=\ell$ for $\ell\in(L\setminus
  Q).$ We further identify an element $a$ of
  $\R_\ell$, $\ell\in\barL$, with its image $\bar a$ in
  $\barR_{\ell} \subset \barR.$ This makes sense since $[a]_Q=\{a\}$
  if $\ell\in L\setminus Q$ while, for $\ell \in Q$, $a$ is identified with the set
  $$[a]_Q=\{b\in\R\bigm| \lv(b)\in Q,\ a \nucong b \}.$$
   In this notation
  \begin{equation*}\label{23}
  \barR=\bigcup_{\ell\in\barL}\barR_\ell.\end{equation*}
  Since now $\barL$ and $\barR$ are compared to subsets of $\L$ and $\R$,
  respectively, we need to distinguish addition and multiplication
  in these semirings from the given addition and multiplication
  $(+,\cdot \; )$ in $L$ and $\R.$

  We indicate these operations in $\barL$ and $\barR$ by the
  subscript ``$t$" (alluding to ``truncation"). Translating our
  rules \eqref{15} and \eqref{20} for addition and multiplication
  in $\bar L$ and $\bar R$ to the new notation, we obtain the
  following:

  If $k,\ell\in\L\setminus Q,$\ $a\in\R_k$ and  $b\in R_\ell,$ then
  \begin{equation*}\label{24}k+_t\ell=\begin{cases} k+\ell &\quad
  \text{if }\ k+\ell\notin Q,\\
  q&\quad \text{if }\ k+\ell\in Q,\end{cases}\qquad
  k\cdot_t\ell=\begin{cases} k\ell &\quad
  \text{if }\ k\ell\notin Q,\\
  q&\quad \text{if }\ k\ell\in Q,\end{cases}\end{equation*}
  \begin{equation*}\label{26}a+_tb=\begin{cases} a+b &\quad
  \text{if }\ k+\ell\notin Q,\\
  \overline{a+b}&\quad \text{if }\ k+\ell\in Q,\end{cases}
 \qquad
  a\cdot_tb=\begin{cases} ab &\quad
  \text{if }\ k\ell\notin Q,\\
  \overline{ab}&\quad \text{if }\ k \ell\in Q.\end{cases}\end{equation*}
   Furthermore,
 $\ell\cdot_t q= q= q \cdot_t \ell = q+_t \ell =  \ell+ _t q$ for any
 $\ell\in\barL.$

 We now decree that the sort transition maps $\bar\nu_{m,\ell}$
 are just the transition maps $\nu_{m,\ell}$ given in $\R:$
 \begin{equation*}\label{29}
 \bar\nu_{m,\ell}=\nu_{m,\ell}:R_\ell\to R_m,\end{equation*}
 if $\ell,m\in\barL,$ $m \ge\ell.$

 It is  easy   to check that the \semiring0
 $\barR,$ together with the partition
 $(\barR_\ell\bigm|\ell\in\barL)$ and the sort transition maps
 $\bar\nu_{m,\ell}$ for $\ell,m\in\barL,$ $\ell\le m,$ satisfies  the
 axioms A1--A4, and B, cf. Definition \ref{defn1}. Thus $\barR$ is
 an $\barL$-quasi-layered \domain0.

 \begin{defn}\label{defn6}
 We call this $\barL$-quasi-layered \domain0 $\barR$ the
 \textbf{truncation of} $R$ \textbf{at} $Q.$
 \end{defn}

 \begin{rem} In the special case
 that $Q=\{\ell\in L\bigm| \ell>q\}$ for some $q\in L,$ $q\ne\iy,$
 we also  say that $\barR$ is
 the \textbf{truncation of} $\R$ \textbf{at} $q.$ Furthermore, $\olR$ is $\nu$-bipotent if $R$ is $\nu$-bipotent; $\olR$ is  uniform if
 $R$ is uniform.\end{rem}

 The truncation $\R/Q$ has the
 map
 $$\nu_{\barR} : \barR\to
 \bar \tG,$$
 obtained from $\nu_{\R}$ by restriction. Also, the elements $e_\ell=e_{\ell,\barR}$
 in $\barR_{\ge 1}$ are the elements $e_\ell=e_{\ell,\R}$ in $\R_{\ge 1}$ for
 $\ell\in\L \setminus Q,$\ $\ell\ge 1,$   whereas $e_q = [e_k]_Q$ for $k \in Q$.

\begin{exampl}   In the example $\R=\R(L,\tG)$ of a uniform layered \domain0, from Example~\ref{defn5}, we obtain
 \begin{equation}\label{30} \R(L,\tG)/Q=\R(L/Q,\bar \tG).\end{equation}
\end{exampl}

  \begin{rem}\label{whereghosts} If $a \in \R(L,\tG)$ with $s(a) \in Q$, then $\bar a+\bar a = \bar a,$ since $Q$ is an
  upper ideal. Thus, the ``top'' layer of $\R(L,\tG)/Q$ is ghost. \end{rem}

\begin{exampl}\label{ntrunc1} For $L = \{ 1, \infty\}$ and $Q = \{ 1\},$
we have the truncation from the standard supertropical \semiring0
to the (idempotent) max-plus \semiring0. This shows how the
standard supertropical theory ``covers'' the tropical theory.

For $L = \Net $  and  $Q = \{ k: k \ge n \} \subset \bbN,$ the
truncation of $ \R(\bbN,\tG)$ at $Q$ is layered by $\{1,\dots,
n\}$. We can continue by taking $L = \{ k: k \le n\}$ and $Q = \{
k:
  m \le k \le n\}$ for some given $m<n$.
In this way, we get an infinite sequence of successive covers of
the max-plus algebra, each of which provides more tropical
information as $n$ increases.
\end{exampl}

\subsection{Layered homomorphisms}\label{homomorph}

Truncation can be understood in terms of
universal algebra. We assume that $L$ is non-negative. 


\begin{defn}\label{morph1} A \textbf{layered homomorphism} $$(\vrp, \rho ): \ldsR \ \to \ \ldsLRpr$$
in the category of $L$-quasi-layered \domains0 is a  \semiring0\
homomorphism
  $\rho: L \to L'$ preserving the given partial orders, i.e., satisfying the condition:
\boxtext{

\begin{enumerate}
 \item[M1.] $k \le \ell $ implies $ \rho(k) \le  \rho(\ell). $

\end{enumerate}}
 together with a \semiring0
 homomorphism $\vrp: R \to R'$ such that
\boxtext{
\begin{enumerate}
    \item[M2.] $\lv '(\vrp
 (a)) \ge \rho (\lv(a)), \quad \forall a \in R. $


\end{enumerate}}
\end{defn}

The definition becomes more complicated when $0 \in L;$ then we
need to modify Axiom~M2 to: \boxtext{
 \begin{itemize} \item[M2'.] $\lv '(\vrp
 (a))  = \ell,$ where $\ell = 0$ or $\ell \ge \rho (\lv(a)), \quad \forall a \in R. $
\end{itemize}}

  We always write $\Phi := (\vrp, \rho
).$
From now on, we assume that  the $L$-quasi-layered \domain0 $R$ is
uniform. Furthermore, we often assume $L = L'$ and $\rho = \id_L;$
we call $\Phi$ a \textbf{natural homomorphism} in this situation.
If $\Phi:  \ldsR   \to  \ldsLRpr $ is a  natural homomorphism such
that $\vrp$ is 1:1, we say that $R'$ is an \textbf{extension} of
$R$.

\begin{dig} When defining layered homomorphisms over $L$-quasi-layered
\domains0 which are not necessarily uniform, in order to preserve
all the given structure, we should also require the condition:
\begin{equation} \vrp(\nu_{\ell,k}( a))   =\nu'_{\ell', k'}(
    \vrp(a)),\\
  \qquad \forall k,\ell \in L, \quad \forall a \in R_k,\end{equation}
 with $k' =
  \lv'(\vrp(a))\le \ell' =
  \lv'(\vrp(\nu_{\ell,k}( a))).$

On the other hand, this condition is rather technical, even when
$\Phi$ is natural, which is why we focus on uniform $L$-layered
\domains0. \end{dig}

\begin{lem}\label{Phidet}  Write $e_{ \ell,R}$ for $e_\ell$ in
$R$. Then $\vrp(e_{\ell,R}) = e_{\ell,R'}$ for each $\ell$ in the
sub-\semiring0 of $L$ (resp.~$L'$) generated by $1$.
\end{lem}\begin{proof} Then  $\vrp ( e_{1,R}) = \vrp (\rone) = \one_{R'} =  e_{1,R'}.$ Thus, for each $n \in \Net,$ we have
$$\vrp(e_{n,R}) = \vrp ( e_{1,R} + \cdots + e_{1,R}) =   \vrp ( e_{1,R}) + \cdots + \vrp
(e_{1,R}) =  e_{1,R'} + \cdots + e_{1,R'} = e_{n,R'}.$$
\end{proof}

It follows at once that $\vrp$ is given  by its action on $R_1$.
\begin{prop}\label{multip} If $a = e_\ell a_1$ as in Lemma~\ref{Krems1}, then
\begin{equation}\label{M3}\vrp(a) =  \vrp(e_{\ell,R}) \vrp(a_1) =
e_{\ell,R'}\vrp(a_1). \end{equation}
\end{prop}
\begin{proof} $\vrp(a) =  \vrp(e_{\ell,R}) \vrp(a_1) =
e_{\ell,R'}\vrp(a_1).$
\end{proof}

\begin{cor} Equation~\eqref{M3} holds automatically whenever $R$ is uniform
$L$-layered.
\end{cor}
\begin{proof}  Lemma~\ref{Krems1} is applicable.
\end{proof}


%
%

\subsubsection{Examples of layered homomorphisms}

\begin{exampl}[Truncation] The commutative square \eqref{trunc} says that the map
  $$r_Q^{\R}: R \to R/Q$$ is a layered homomorphism from an $L$-layered
 \domain0 to an $L/Q$-layered
 \domain0, where $\rho$ is the natural map sending $m \to
q$ for all $m\in Q.$
  \end{exampl}

\begin{exampl}[1-localization]\label{loc1} If $R$ is an $L$-layered \domain0, then
taking any multiplicative submonoid $S $ of $ R_1$, we can form
the localization $S^{-1}R$ as a monoid, and define addition via
$$ \frac au + \frac bv = \frac {av+bu}{uv}$$ for $a,b \in R,$
$u,v \in S$. $S^{-1}R$ becomes an $L$-layered \domain0 when we
define $s(\frac au) = s(a).$ There is a natural layered
homomorphism $R \to S^{-1}R$ given by $a \mapsto \frac a \rone ,$
which is injective since $R_1$ is cancellative.

Taking $S = R_1$, we call $S^{-1}R$ the  $L$-layered
\textbf{$1$-\semifield0 of fractions} of $R$; this construction
shows that any uniform $L$-layered \domain0 can be embedded into a
uniform $L$-layered $1$-\semifield0.
  \end{exampl}

  \begin{exampl}[1-divisible closure]\label{divcl}
 We say that an $L$-layered \domain0  $R$ is $1$-\textbf{divisibly closed} if for each $a\in R_1$
and $n \in \Net$ there is $b\in R_1$ such that $b^n =a$. As in
Example~\ref{loc1}, any uniform $L$-layered $1$-\semifield0~$R$
can be
  embedded
 into a $1$-divisibly closed uniform $L$-layered $1$-\semifield0 $F$.
Namely, adjoin $\root m \of a$ to $R_1$ for each $a \in R_1$ and
$m \in \Net,$ as in Remark~\ref{divclo}, and enlarge all the other
layers accordingly. We call $F$ the $1$-\textbf{divisible closure}
of $R$.\end{exampl}

 Although Example~\ref{divcl} is all we need for the applications in this paper,
 let us put this construction in its proper context for $L$-layered \domains0.

   \begin{exampl}[Digression: $\nu$-divisible closure]\label{divcl1} We say that an $L$-layered \domain0  $R$ is $\nu$-\textbf{divisibly closed} if for each $a\in R$
and $n \in \Net$ there is $b\in R$ such that $b^n \equiv a$ under
the equivalence of Definition~\ref{equivrel}. Note that if $s(a) =
\ell$ then $s(b) = \root n \of \ell.$ This implies that $L$ must
be closed under taking $n$-th roots for each $n$. Assuming that
$L$ is a group satisfying this condition, one can construct the
$\nu$-divisible closure, sketched as follows:

\begin{description}
\item[Step 1] Given $a \in R_\ell,$ adjoin a formal element $b \in
R_{\root n \of \ell}$, and consider all formal
sums\begin{equation}\label{prim2} f(b) :=  \sum_{i}  \al_i
 b^{i}: \quad \al_i \in R, \quad \al _i^{n} a^{i} \nucong  \al _{i'}^n a^{i'}, \forall
  i,i'.\end{equation}
($  \sum_{i}  \al_i
 b^{i}$ is to be considered as the $n$-th root of  $ \sum_{i}  \al_i^{(n)}
 a^{i}.$)

  Define
$R_b$ to be the set of all elements of the form \eqref{prim2},
where any  $\a \in R $  is identified with $\al b^0$. (This could
be stated more precisely in terms of evaluations of polynomials;
 compare with Definition~\ref{primarydef} below.)  We can
define the sorting map $s: R_b \to L$ via $$s(f(b)) =\sum _{i}
s(\al_i) \root n \of \ell^i \in L.$$

 We define $\nucong$ on $R_b$ (notation as in \eqref{prim2}) by saying   $f_b \nucong f'_b: = \sum_{j=0}^{t'}
 \al'_j
 b^{i}$ if  $ \al_i^n
a^{i} \nucong  {\al'_j}^n
 a^{j}$. In particular, $f_b \nucong c$
for $c\in R$ if  $ \al_i
 b^{i} \nucong c^n$.
Likewise, we write $f_b >_\nu f'_b: = \sum_{j=0}^{t'}
 \al'_j
 b^{i}$  if $ \al_i^n
a^{i} > _\nu  {\al'_j}^n
 a^{j}$.

Now we can define addition on $R_b$ so as to be $\nu$-bipotent,
where for   $\nu$-equivalent elements we define $f(b)+g(b)$ to be
their formal sum (combining coefficients of the same powers of
$b$);
 multiplication is then defined in the obvious way, via distributivity over addition.
 Now $R_b$ is an $L$-layered \domain0, in view of
 Proposition~\ref{removesort}.

 This construction is unique up to isomorphism, since one could replace $a$ by any equivalent element
 in terms of Definition~\ref{equivrel}.
\pSkip

\item[Step 2] Using Step 1 as an inductive step, one can construct
the $\nu$-divisible closure by means of Zorn's Lemma, analogously
to the well-known construction of the algebraic closure,
cf.~\cite[Theorem~4.88]{R}. \pSkip

\item[Step 3] This construction is unique up to isomorphism,
again by the same argument known for the algebraic closure.
\end{description}

The last two steps should be viewed in terms of ``model
completeness,'' cf.~\cite[\S4.3]{M} or \cite{VdD}.
  \end{exampl}

    \begin{exampl}[Completion]\label{compl} One can construct the \textbf{completion} of any $L$-layered
    \domain0
 $R$ as follows: First,  take the completion of the ordered group $R/\nucong$ as described in Remark~\ref{linalg2}. We define $\nu$-\textbf{Cauchy sequences} in $R$ to
be those sequences $(a_i) := \{ a_1, a_2, \dots \}$ which become
Cauchy sequences modulo $\nucong,$ but which satisfy the extra
property that there exists an $m$ (depending on the sequence) for
which $s(a_i) = s(a_{i+1})$, $\forall i \ge m.$ This permits us to
define the \textbf{sort} of the $\nu$-Cauchy sequence to be
$s(a_m).$ Then we define the \textbf{null} $\nu$-\textbf{Cauchy
sequences} in $R$ to be those sequences $(a_i) := \{ a_1, a_2,
\dots \}$ which become null Cauchy sequences modulo $\nucong,$ and
the completion $\htR$ to be the factor group.

We also extend our given pre-order $\nu$ to $\nu$-Cauchy sequences
by saying that $(a_i) \nucong (b_i)$ if $(a_ib_i^{-1})$ is a null
$\nu$-Cauchy sequence, and, for $(a_i) \not \nucong (b_i)$, we say
$(a_i)
>_\nu (b_i)$ when there is $m$ such that $a_i >_\nu b_i$ for all
$i>m$. The completion $\htR$ becomes an $L$-layered \domain0 under
the natural operations, i.e., componentwise multiplication of
$\nu$-Cauchy sequences, and addition given by the usual rule that
\begin{equation}\label{144}
(a_i) + (b_i)=\begin{cases}  (a_i)& \quad\text{if}\ (a_i) >_\nu
(b_i),\\ (b_i)& \quad\text{if}\ (a_i) < _\nu (b_i),\\
 \nu_{s(a_i)+s(b_i),s(a_i+b_i)}(a_i+b_i)& \quad\text{if}\ (a_i)\nucong (b_i).\end{cases}\end{equation}

(In the last line, we arranged for the layers to be added when the
$\nu$-Cauchy sequences are $\nu$-equivalent.) It is easy to verify
$\nu$-bipotence for $\htR$.
  \end{exampl}

These constructions are universal, in the following sense:

\begin{prop}\label{univ} Suppose there is an embedding $\vrp: R \to F'$ of a uniform $L$-layered
\domain0 $R$  into a  $1$-divisibly closed, uniform $L$-layered
$1$-\semifield0 $F',$ and let $F$ be the $1$-divisible closure of
the $1$-\semifield0 of fractions of $R$. Then $F'$ is an extension
of $F$. If $F'$ is complete with respect to the $\nu$-pre-order,
then we can take $F'$ to be an extension of the completion of $F$.
\end{prop}
\begin{proof} This is standard, so we just outline the argument.
First we embed the $L$-layered $1$-\semifield0 of fractions of $R$
into $F'$, by sending $\frac {b} {a_1} \to \frac
{\vrp(b)}{\vrp(a_1)}.$ This map is 1:1, since if $\frac {b} {a_1}
= \frac {d} {c_1},$ then $c_1 b = a_1d,$ implying $ \vrp(c_1 b) =
\vrp(a_1d),$ and thus $\frac {\vrp(b)}{\vrp(a_1)} = \frac
{\vrp(d)}{\vrp(c_1)}.$ Thus, we may assume that $R$ is an
$L$-layered $1$-\semifield0. Now we define the map $F \to F'$ by
sending $\root m \of a  \to \root m \of \vrp(a)$, for each $a \in
F_1.$ This is easily checked to be a well-defined, 1:1 layered
homomorphism.

In case $F'$ is complete, then we can  embed  the completion of
$F$  into $F'$. (The completion of a $1$-divisibly closed
$1$-\semifield0 is $1$-divisibly closed, since taking roots of a
$\nu$-Cauchy sequence in $F_1$ yields a $\nu$-Cauchy sequence.)
\end{proof}




\begin{thm}\label{tandet} Suppose  $\Phi = (\vrp, \rho):\ldsR \to \ldsLRpr$ is a  layered homomorphism  of
 layered \domains0, where $L$ is totally ordered.  
Then the restriction of $\Phi$ to $L$ is determined by the action
of $\vrp$ on $R_1$.
\end{thm}  \begin{proof} By Lemma~\ref{Krems1}, for any $a \in
R_k$ we have $a = e_{\ell,R}a_1 ,$  for some $a_1 \in R_1$, and
thus
$$\vrp(a) =  \vrp(e_{\ell,R}) \vrp(a_1)=   e_{\ell,R'} \vrp(a_1),$$
\end{proof}


%

Layered homomorphisms can also be used to understand
Theorem~\ref{B1new}.
  \begin{prop}\label{111} Given any \semiring0\ $L$ and $L$-layered \domain0\ $R =
  \ldsR,$ take $\tG $ to be the direct limit of the $R_\ell$, as described in
  Remark~\ref{directlim}.
  Then we have the layered homomorphism  $\Phi = (\vrp, \id_L)$ where  $\vrp: R \to R(L,\tG)$ is given by $$a \mapsto
  (s(a),a^\nu).$$

  Moreover, if $\ldsLRpr$ is a uniform
$L'$ -layered \domain0, then $\Phi$ is universal, in the sense
that
 any layered homomorphism
$$(\psi, \id_L): R \to R'$$ factors through $\vrp,$ via a
layered homomorphism
$$(\pi, \id_L): R(L,\tG) \to R'$$ such that $\psi = \pi \circ
\vrp.$

Finally, any uniform $L$-layered \domain0 $R$ is isomorphic to
$R(L,\tG),$ where $\tG = R_1$ (viewed as a monoid).
  \end{prop}
  \begin{proof} $\vrp(ab) = (s(ab),ab^\nu) =
(s(a),a^\nu)(s(b),b^\nu)$. Addition is trickier, since we have
  to handle the case of $a+b$ where $a \nucong b$. But here $$\vrp (a+b) = (s(a+b),(a+b)^\nu) =
  (s(a)+s(b),a^\nu)
  =\vrp(a)+\vrp(b).$$

  To prove the next assertion, define  $\pi: R(L,\tG) \to R'$ by
  $\pi((s(a),a^\nu)) = \psi(a),$ and note that $\pi$ is a
  homomorphism in view of Corollary~\ref{tang102}, and is
  well-defined because the sort transition maps in $R'$ are bijective.

  The last assertion is seen by considering the natural
  homomorphism $R(L,\tG) \to R$, where the restriction to $R_1$ is
  the identity. The $\ell$ component $e_\ell R_1$ then can be
  identified with $R_\ell,$ in view of Theorem~\ref{tandet}.
  \end{proof}

A  useful layered isomorphism is given in Remark~\ref{switch}.



\subsection{Layered supervaluations and the layered analytification}

In case the layered \domain0 is not uniform, we need a more
general notion of morphism, treated in \cite{IKR4}. To understand
what is going on, we need to generalize the notion of
``valuation.'' Valuations are important in algebraic geometry, and
play a key role in tropical theory largely because of the
following example.

\begin{exampl}\label{analyt0}
Suppose $K$ is the field of Puiseux series $\{ f: = \sum _{u \in
\Real} \a _u \la ^u :$ $f$ has well-ordered support$ \}$ over a
given field $F$. Then we have the $m$-valuation $v: K \to \tG$
taking any Puiseux series $f$ to the lowest real number $u$ in its
support.
\end{exampl}

A word about notation: Given a valuation $v: K \to \tG$, one can
replace $v$ by $-v$ and reverse the customary inequality to get
$$v(a+b) \le \max\{v(a), v(b) \},$$
which is more compatible with the max-plus set-up. In what
follows, we define an $m$-\textbf{valuation} to be a valuation
whose target is a monoid,
 cf.~\cite[Definition~2.1]{IKR1}. In other words, we weaken the assumption that $\tG$ be an ordered group to $\tG$ merely being
 an ordered monoid, whose operation we write from now on in
 multiplicative notation. (In other words, $v(ab) = v(a)v(b).$) This fits in better with our algebraic
 notation for \semirings0. Thus, any valuation $v: K \to \tG$ is an
 $m$-valuation, where we just disregard addition in $K$.

Payne \cite{Pay2} has developed an algebraic version of
Berkovich's theory of analytification, which can be viewed as the
limit of tropicalizations. In his theory, a \textbf{multiplicative
seminorm} $|\phantom{t}|: W \to \Real$ on a ring $W$ is a
multiplicative map satisfying the triangle inequality $$|a+b| \le
|a| + |b|.$$ 
The underlying space in Payne \cite{Pay2} is the set of
multiplicative seminorms from $K[\lm_1, \dots, \lm_n]$ to
$\Real_{>0}$ extending $v$, for a given $m$-valuation $v: K \to
\Real_{>0}$. We generalize this definition by taking an arbitrary
ordered \semiring0 instead of $\Real_{>0}.$

The supertropical version, the \textbf{strong supervaluation}, is
defined in \cite[Proposition~4.1 and Definition~9.9]{IKR1}  as a
monoid homomorphism $\vrp$ satisfying $\vrp(a)+\vrp(b) \lmodg
\vrp(a+b)$, where $\lmodg$ is the ghost surpassing relation of
\cite[Definition~9.1]{IKR1}.
 In  this way, strong supervaluations generalize seminorms.

Here is the layered analog. 

\begin{defn}\label{layeredv} A   \textbf{layered supervaluation} on a ring $W$ is a map $\vrp: W\to R$
  from $W$ to an $L$-layered semiring~$R$ with the following
  properties:
 \begin{alignat*}{2}
&LV1:\ &&\vrp(1)=\rone,\\
&LV2:\ &&\forall a,b\in R: \vrp(ab)=\vrp(a)\vrp(b),\\
&LV3:\ &&\forall a,b\in R: \vrp(a+b) \le _\nu  \vrp(a)+\vrp(b) ,\\
&LV4:\ &&\vrp(0)=\rzero.
\end{alignat*}

A   $\BB$-\textbf{layered supervaluation} on a ring $W$ is a
layered supervaluation $\Phi: W ^ \times \to R$, where $W ^
\times:= W\setminus \{ 0 \}$, such that $\Phi(W) \subseteq R_0
\cup R_1.$

Specifically, in the special case where $0 \notin L$, a
\textbf{tangible layered supervaluation$^\dagger$} on an integral
domain $W$ is a map $\Phi: W ^ \times \to R$
  from $W\setminus \{ 0 \}$ to an $L$-layered \domain0 $R$ with the following
  properties.
 \begin{alignat*}{2}
&LV1^\dagger :\ &&\Phi(1)=\rone,\\
&LV2^\dagger :\ &&\forall a,b\in R: \Phi(ab)=\Phi(a)\Phi(b),\\
&LV3^\dagger :\ &&\forall a,b\in R: \Phi(a)+\Phi(b) \le _\nu
\Phi(a+b).\end{alignat*}

A   \textbf{tangible layered supervaluation$^\dagger$} on an
integral domain $W$ is a layered supervaluation$^\dagger$   such
that $\Phi(W) \subseteq R_1.$

\end{defn}


\begin{prop} Suppose that  $R = R(L, \tG)$ an  $L$-layered \domain0. If $\Phi: W \to \tG$
is a $\BB$-layered supervaluation of a ring $W$, then $\Phi(a)$ is
tangible for every invertible element $w$ of $W$. (In particular,
if $W$ is a field, then $\Phi(W ^ \times)$ is tangible.)
\end{prop}
\begin{proof} $\Phi(w)\Phi(w^{-1}) = \Phi(1) = \rone,$ so
$\Phi(w)$ is tangible by Lemma~\ref{invelt}.
\end{proof}

In this situation, the tangible layer determines the layered
supervaluation.

\begin{rem} Under the assumptions of the proposition, it follows that the
transmissions treated in \cite[Proposition~6.41]{IKR4} arise from
layered homomorphisms.\end{rem}

 The morphisms in the layered category should then
be those maps which transfer one layered supervaluation to
another. In the standard supertropical situation, these are the
transmissions of \cite{IKR3}, which are given in the layered
setting in~\cite{IKR4}.
 This paves the way for the following concept, with,
notation as in~Example~\ref{analyt0}:

\begin{defn}\label{analyt}  Let $R = R(L,\tG)$, and view $v$ as the
composite map of monoids $$K \overset v \to \tG \cong R_1
\subseteq R.$$ Then for any affine algebraic variety $X$ over $K$,
we define $K^{\operatorname{layered-an}}$ to be the set of
$\BB$-{layered valuations} from $K[\lm_1, \dots, \lm_n]$ to $R$
that extend $v$.
\end{defn}

The space $K^{\operatorname{layered-an}}$ extends
$K^{\operatorname{an}}$ of \cite{Pay2}, and its theory invites
further study.

\section{Layered functions and their roots, with
multiplicities}\label{multipl}

As usual, we assume throughout that $R = \ldsR$ is an $L$-layered
\domain0. We recall the multiplicative sort function $\lv: R \to
L$, and write $\xl{a}{\ell}$ when we need to emphasize that
$\lv(a) = \ell.$


\subsection{The layered function \semiring0}\label{Fun5}

Before constructing  the polynomial \semiring0   and the Laurent
polynomial \semiring0, we need an umbrella structure in which to
develop the theory.

\begin{defn}\label{lotsoffun}  For any set $\tSS$, and any \semiring0
$R$, $\Fun (\tSS,R) $ denotes the set of functions  $ f: \tSS\to
R$, which are $\nu$-\textbf{compatible}, in the sense that
 if $\bfa \cong_\nu \bfa '$, then $f(\bfa) \nucong f(\bfa ')$.
\end{defn}

  $\Fun (\tSS,R) $ also is  a \semiring0, whose operations are given
pointwise:
$$ (fg)(\bfa) = f(\bfa)g(\bfa), \qquad (f+g)(\bfa) = f(\bfa) + g(\bfa),$$ for all $\bfa \in \tSS.$
The unit element  of $\Fun (\tSS,R) $ is the constant function
always taking on the value $\rone$.

\begin{rem}\label{rmk:LsurFun}
We write $f \lmodL g $, for $f,g \in \Fun (\tSS,R)$, when $f(\bfa)
\lmodL g(\bfa) $,  $\forall \bfa \in S.$ Likewise, we write  $f \nucong g $  when $f(\bfa)
\nucong g(\bfa) $,  $\forall \bfa \in S.$
Now we also have
the \textbf{Frobenius-type properties}:
$$ (f + g) ^k \lmodL f^k + g^k, \qquad \forall f,g \in \Fun
(\tSS,R),$$ and
$$ (f + g) ^k  \nucong  f^k + g^k, \qquad \forall f,g \in \Fun
(\tSS,R).$$
\end{rem}

\begin{rem}\label{homim} If $\vrp: R \to R'$ is a \semiring0 homomorphism,
then there is a \semiring0 homomorphism $\Fun (\tSS,R) \to \Fun
(\tSS,R')$ given by $f \mapsto \vrp \circ f,$ where
$$\vrp \circ f : \bfa \mapsto \vrp(f(\bfa)), \qquad \forall \bfa \in S.$$ \end{rem}

\subsubsection{Polynomials}

%
%

 Given a
 \semiring0 $R$, we have the polynomial \semiring0 $R[\Lambda]$ in the
 commuting indeterminates $$\Lambda := \{ \la_1, \dots, \la _n
 \}.$$

 By convention,  $\xl{\la }{\ell }$ denotes
$\xl{\rone}{\ell}\lm.$ Thus, any monomial can be written in the
form $ \al_\bfi \lm_1^{i_1}\cdots \lm_n^{i_n}$ where $\bfi =
(i_1,\dots, i_n)$, which has layer $s(\al_\bfi).$   For any
 polynomial $f = \sum_\bfi \al_\bfi \lm_1^{i_1}\cdots
 \lm_n^{i_n}$,
 we may apply  the sort transition maps~$\nu_{\ell,k}$   to its
 coefficients. We say a polynomial $f$ is \textbf{tangible} if
 each of its coefficients is tangible.

\begin{lem}\label{comp0}  Any polynomial   is $\nu$-compatible. \end{lem}
\begin{proof}  Any monomial obviously is $\nu$-compatible, and a polynomial is a sum of
monomials.\end{proof}

Just as in
\cite{IzhakianRowen2007SuperTropical}, we view polynomials in
$R[\Lambda]$ as functions, but perhaps taking values in some
extension~$R' $ of~ $R$. More precisely, for any subset $\tSS
\subseteq R^{(n)},$ there is a natural homomorphism
\begin{equation}\label{eq:polymap1} \psi: R[\Lambda] \to \Fun
(\tSS, R),
\end{equation}
 obtained by viewing a polynomial as a function on $\tSS$. In
 fact, for  any subset $\tSS
\subseteq R^{(n)}$ and  \semiring0 homomorphism $\vrp: R \to R' $,
 there is a natural homomorphism \begin{equation}\label{eq:polymap2} \psi_\vrp:
R[\Lambda] \to \Fun (\tSS, R'),
\end{equation}  which is the composite of the natural map
$\widetilde {\vrp}: R[\Lambda] \to R' [\Lambda],$ with the natural
homomorphism $$\psi' :  R' [\Lambda] \to \Fun (\tSS,R' ) .$$

 When $R$ is a $1$-\semifield0, the same
 analysis is applicable to Laurent polynomials $R[\Lambda,
 \Lambda^{-1}]$, since the homomorphism $\la_i \mapsto a_i$ then sends $\la_i^{-1} \mapsto a_i^{-1}$.
 Likewise, we can also define the \textbf{\semiring0 of rational polynomials}
  $R[\Lambda]_{\rat}$, where the exponents of the indeterminates $\la_i$ are taken to
  be arbitrary rational numbers. Then in view of
  Remark~\ref{linalg}, when $R$ is $1$-divisibly closed and $\Net$-cancellative,
  the homomorphism $\la_i \mapsto a_i$ then sends $\la_i^{m/n} \mapsto a_i^{m/n}$.

  In general, we
work in some natural sub-\semiring0 of $\Fun (\tSS,R' ),$ which we
denote as
 $\mcR$, which might  be identified with the \semiring0 of polynomials, Laurent polynomials,  or rational
 polynomials.

\subsection{Decompositions of functions}

 We want to decompose functions as sums of ``nice'' functions,
 the way that polynomials are sums of monomials. This can be done
 axiomatically, in a rather natural way.
We generalize a notion from \cite{IzhakianRowen2007SuperTropical},
defined there  on $R[ \lm_1, \dots, \lm_n]$. Essential summands of
a function of $\Fun (\tSS,R)$ can be described more easily in the
layered setting than in the standard supertropical setting,
because we can utilize the different layers. Throughout, we fix a
sub-\semiring0 $\mcR$ of $\Fun (\tSS, R'),$ where $R'$ is a
suitable $1$-divisibly closed,  $L$-layered $1$-\semifield0
containing $R$. (One would expect $R'$ to be obtained via
Examples~\ref{loc1} and~\ref{divcl}, but this is a nontrivial
issue that needs separate consideration.)


\begin{defn}\label{decomp}  A function $f\in \FuncalSR $ \textbf{dominates} $g\in \FuncalSR$ \textbf{at} $\bfa \in S$ if
$f(\bfa ) \nuge g(\bfa)$; $f$~\textbf{strictly dominates} $g$ at
$\bfa $ if $f(\bfa ) \gnu g(\bfa )$.

Write $f = \sum_i h_i$, where $h_\bfi  \in \mcR.$ The summand $h_i
$ is \textbf{essential} at $\bfa \in \tSS$ if $f(\bfa) =
h_i(\bfa)$; $h_i $ is
 \textbf{inessential} at $\bfa \in \tSS$ if $f(\bfa) = \sum _{j \ne i} h_{j}(\bfa)  . $ Also, $h_i $ is \textbf{quasi-essential}
at $\bfa$ if $h_i $  is neither essential nor inessential at
~$\bfa$; it follows that $f(\bfa) \nucong h_i (\bfa)$ but $f(\bfa)
\ne h_i(\bfa)$.   We say $h_i $ is \textbf{essential} in $f$ if
$h_i $ is essential at some $\bfa \in \tSS$; $h_i $ is
\textbf{inessential} in $f$ if $h_i $ is inessential at every
$\bfa \in \tSS$. Finally,
 $h_i $ is
\textbf{quasi-essential} if $h_i$ is neither essential nor
inessential; in other words, $h_i$ is quasi-essential at some
points of $\tSS$, but not essential at any point.

%
%
%
A \textbf{decomposition} of $f$ is a sum $f = \sum_i  h_i$ where
each $h_i$ is essential or quasi-essential. The \textbf{shell} of
the decomposition is the sum of those $h_i$ that are essential.
The  \textbf{support} of the decomposition at a point $\bfa$ is
the sum of those $h_i$ that are either essential or
quasi-essential at $\bfa$.
%

%
\end{defn}

\subsection{The layering map}

  $\Fun (\tSS,R) $ plays an extremely important role in our research,
  so we look for a layered framework with respect to an appropriate
  sorting \semiring0, which turns out to be
$\Fun (\tSS,L) $.  We assume throughout this discussion that
$L$ is non-negative, since we do not know how to interpret
negative layers. (The zero layer itself is problematic enough,
since it does not add to the value of a polynomial; for example,
if $f = \la +  \xl{1}{0},$ then  $f(\xl{1}{\ell}) = \xl{1}{\ell},$
for any $\ell \in L$.)

\begin{rem}\label{Funsort} When $L$ is a partially pre-ordered \semiring0, $\Fun (\tSS,L) $ is
partially pre-ordered by the relation $p \le q$ if $p(\bfa) \le
q(\bfa)$ for all $\bfa \in \tSS.$ This partial pre-order is
directed, since $p(\bfa) ,q(\bfa) \in \Fun (\tSS,L)$ are bounded
by $ p(\bfa)+ q(\bfa).$
\end{rem}

\begin{defn}\label{laymap} The \textbf{layering map} of a function $f\in \Fun (\tSS,R)$ is
the map $\vmap_f: \tSS\to L$ given by $$\vmap_f(\bfa) :=
s(f(\bfa)), \qquad \forall  \bfa \in \tSS.$$ \end{defn}

Thus, $\vmap_f \in \Fun (\tSS,L) $. If $R$ is  $L$-layered, then $
\FuncalSR$ inherits a layered structure from $R$ pointwise with
respect to
  $\Fun (\tSS,L) $, in the following sense:
We can define a sorting map $\frak s:\Fun (\tSS,R) \to \Fun
(\tSS,L) $ by sending $f\mapsto \vmap_f$.

\begin{rem}\label{Funsort1}
The sorting map $\frak s $ plays the analogous role, with respect
to functions, as the original sorting map $s: R \to L$.
 The \semiring0 $\Fun (\tSS,R) $ satisfies Axioms
A1--A3 and B with respect to the sorting \semiring0 $\Fun (\tSS,L)
$, all
 verified pointwise, but $\Fun (\tSS,R)$ is not $\nu$-bipotent,
since some of the evaluations  of $f+g$ might come from $f$ and
others from $g$.

\end{rem}

 The
layering map of a function is the key to our notion of variety.
The geometry is contained in the information it provides, as
indicated in Remark~\ref{lg}. In the standard supertropical
theory, $\vmap_f (\bfa) = \infty$ iff $\bfa$ is a root of $f$, and
$\vmap_f ^{-1}({1})$ is the complement set of the root set of $f$.
We return to this idea in \S\ref{Zar0}.

Since the layering map $\vmap_f$ is so important, one is led to
ask how far $\vmap_f$ can be from a constant. In one sense, this
is easy. View $f $ in $\Fun ( R_1^{(n)},R')$ if a polynomial $f$
is a sum of  $m$ monomials with tangible coefficients, then
clearly for any $\bfa \in R_1^{(n)},$ $s(f(\bfa))\le m$, so
$\vmap_f$ is bounded by $m$. Such considerations lead to a
connection between $\vmap_f$ and simplicial theory.

\begin{figure} \hskip 3cm
\begin{minipage}{0.8\textwidth}
\begin{picture}(40,230)(0,0)
\includegraphics[width=4.2 in]{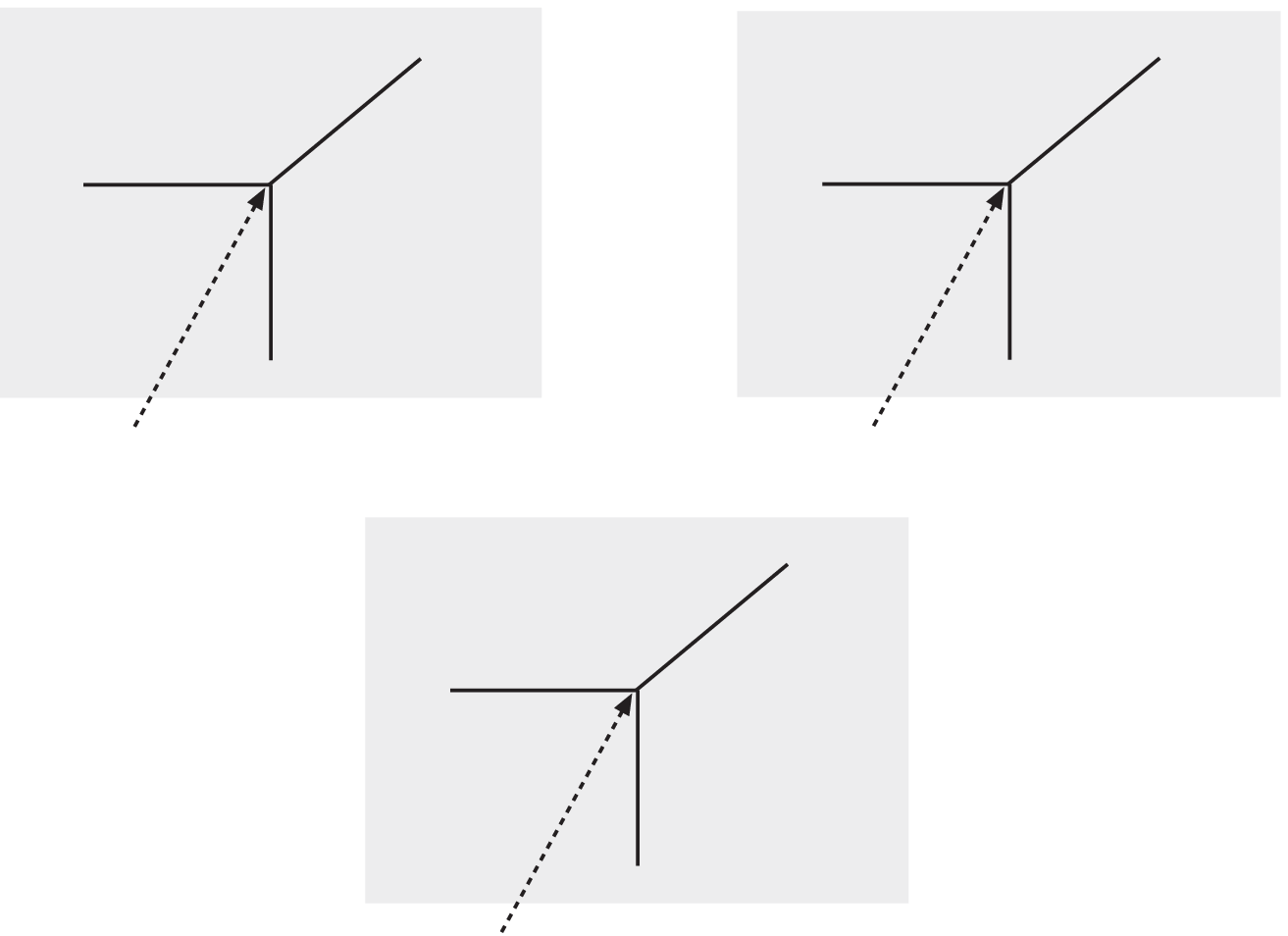}
 \put(-5,220){${R_\ell \times R_\ell}$}
\put(-180,220){${R_1 \times R_1}$} \put(-90,95){${R_k \times
R_\ell}$}


 \put(-285,115){$_{\vmap_f =3}$}  \put(-110,115){$_{\vmap_f =2\ell + 1}$}

 \put(-305,177){$_{\vmap_f =2}$}  \put(-138,177){$_{\vmap_f =\ell+1}$}

 \put(-290,215){$_{\vmap_f =1}$}  \put(-120,215){$_{\vmap_f =\ell}$}

 \put(-215,213){$_{\vmap_f =2}$}  \put(-40,213){$_{\vmap_f =2 \ell}$}

 \put(-300,135){$_{\vmap_f =1}$}  \put(-125,135){$_{\vmap_f =1}$}

 \put(-205,155){$_{\vmap_f =1}$}  \put(-25,155){$_{\vmap_f =\ell}$}

 \put(-250,132){$_{\vmap_f =2}$}  \put(-75,132){$_{\vmap_f =\ell+1}$}

 \put(-205,95){$_{\vmap_f =\ell}$}   \put(-130,92){$_{\vmap_f =k+\ell}$}  \put(-110,35){$_{\vmap_f =k}$}
\put(-213,15){$_{\vmap_f =1}$} \put(-165,14){$_{\vmap_f =k+1}$}
\put(-200,-5){$_{\vmap_f =k+\ell+1}$} \put(-227,58){$_{\vmap_f
=\ell+1}$}
\end{picture}
\end{minipage}

\caption{\label{fig:line} The values of the layering map $\vmap_f$
of the generic tangible linear polynomial $f = \al \lm_1 +\bt
\lm_2 + \gm$ for points on $R_1 \times R_1$ (on the top left) and
on $R_\ell \times R_\ell$ (on the top right). The bottom
illustration shows the value of $\vmap_f$  on points having
coordinates of different layers.}
\end{figure}

\begin{exampl} Take
$R = R( \Net,\Real)$, written in logarithmic notation.
\begin{enumerate} \eroman
    \item Figure \ref{fig:line} shows the values of the layering map of the
     (tangible) linear polynomial $f = \al \lm_1  + \bt \lm_2 +  \gm,$ where  $\al,\bt,\gm \in
    R_1.$ \pSkip
    \item The values of the layering map of the (non-tangible) quadratic polynomial  $f = \lm_1 \lm_2 + \xl{1}{k} \lm_1 + 1 \lm_2 +
0$ are  presented in  Figure \ref{fig:curve}.

\end{enumerate}

\end{exampl}

\begin{figure}[h] \hskip 3cm
\begin{minipage}{0.8\textwidth}
\begin{picture}(40,200)(0,0)
\includegraphics[width=3.8 in]{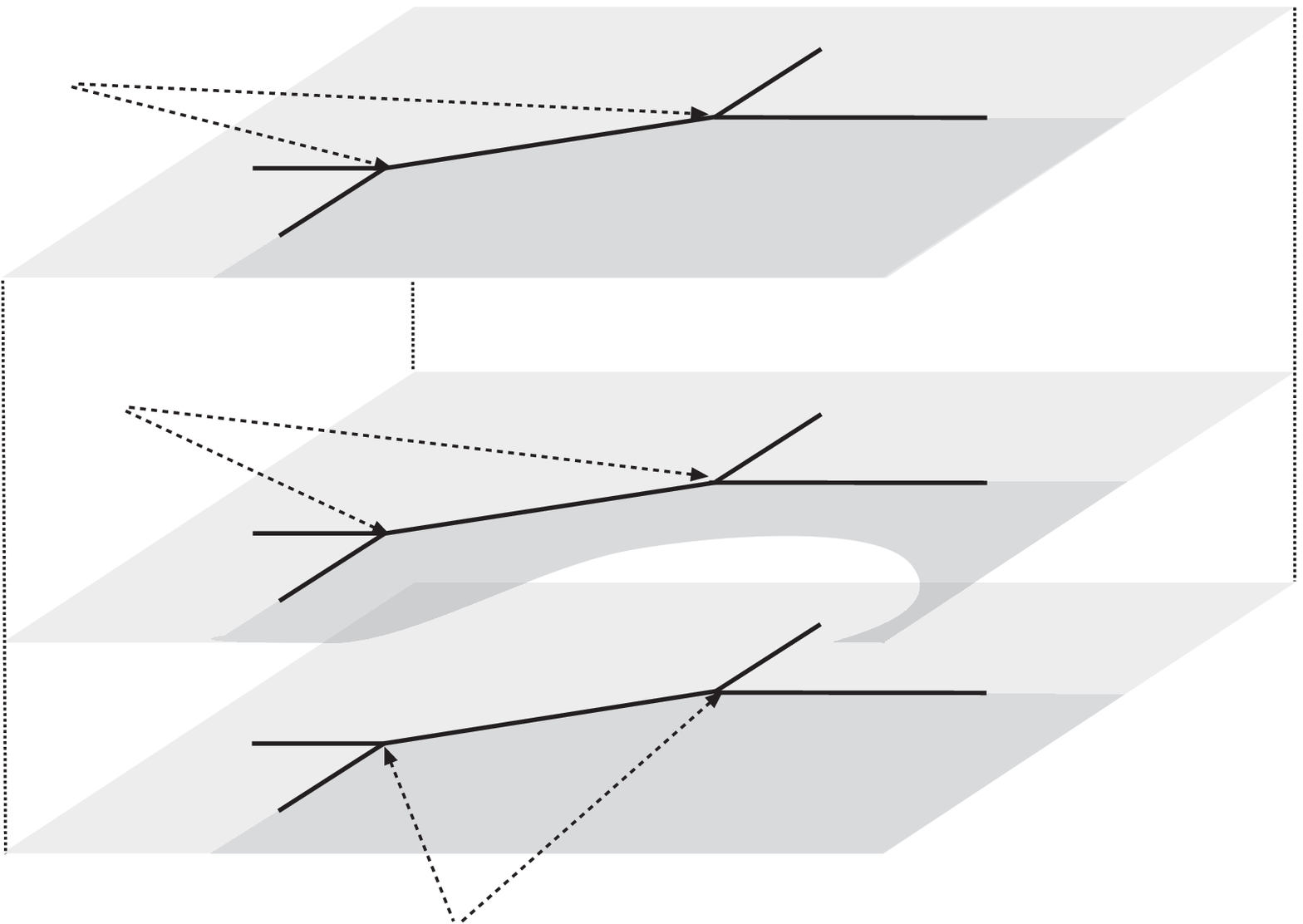}
\put(0,193){${R_\ell \times R_\ell}$} \put(0,116){${R_2 \times
R_2}$} \put(0,72){${R_1 \times R_1}$}


 \put(-200,-5){$_{\vmap_f = k
+2}$} \put(-170,20){$_{\vmap_f = k }$} \put(-150,65){$_{\vmap_f =
1 }$}  \put(-40,65){$_{\vmap_f = 1 }$}  \put(-260,20){$_{\vmap_f =
1 }$} \put(-65,48){$_{\vmap_f = k+1 }$} \put(-180,46){$_{\vmap_f =
k+1 }$} \put(-242,37){$_{\vmap_f = 2 }$}

 \put(-260,64){$_{\vmap_f = 1 }$}
\put(-40,110){$_{\vmap_f = 2 }$} \put(-190,110){$_{\vmap_f = 2 }$}
\put(-65,92){$_{\vmap_f = 2k+2 }$} \put(-280,111){$_{\vmap_f = 2k
+ 4  }$} \put(-242,82){$_{\vmap_f = 3 }$}
 \put(-260,142){$_{\vmap_f = 1 }$}
\put(-40,187){$_{\vmap_f = \ell }$} \put(-190,187){$_{\vmap_f =
\ell }$} \put(-65,169){$_{\vmap_f = \ell(k+1) }$}
\put(-170,142){$_{\vmap_f = \ell k }$} \put(-285,182){$_{\vmap_f =
\ell(k + 2)  }$} \put(-252,158){$_{\vmap_f = \ell+1 }$}
\end{picture}
\end{minipage}

\caption{\label{fig:curve} The values of the layering map
$\vmap_f$ of $f = f = \lm_1 \lm_2 + \xl{1}{k} \lm_1 + 1 \lm_2 +
0$, a non-tangible polynomial,  are presented for the different
layers, ``stacked'' one  above the other.}
\end{figure}

\begin{rem}\label{evalu}   For any $\bfa \in \tSS,$ the substitution map $$\Phi
_\bfa:\FuncalSR \to R,$$ given by $f \mapsto f(\bfa)$ defines a
homomorphism of semirings$^\dagger$. Thus, it makes sense to study
evaluations of functions in this general situation.

More generally, if $\tSS \supseteq \tSS'$ there is a natural
homomorphism $\FuncalSR \to \FuncalSpR$ given by restricting the
domain of a function from $\tSS$ to $\tSS'.$ (We get the previous
paragraph by taking $\tSS' = \{ \bfa \}.$)

\end{rem}

\begin{rem}\label{valusemiring2} Taking $E_L$ as in
Remark~\ref{charsub}, we have the \semiring0 $\Fun(\tSS,E_L)$
which plays the parallel role to \cite[Example~3.24]{CC} in
enabling us analogously to define representable functors and
schemes.
\end{rem}

\begin{defn}\label{laysurp}
Given two functions $f,g \in \Fun (\tSS,R)$, define $f \lmodWL g $
when $f(\bfa) \lmodWL g(\bfa)$ for all $\bfa\in \tSS$. We write $f
\nucong g$ when $f(\bfa) \nucong g(\bfa)$ for each $\bfa \in
\tSS.$
\end{defn}

\begin{rem}\label{essdif} Definition~\ref{laysurp}  differs considerably in the general
layered setting from the standard supertropical setting.  In the
standard supertropical theory, when $f,g$ are rational polynomials
in essential form,  $f \lmodWL g $  means that each  coefficient
of~$f$ surpasses the corresponding   coefficient of $g$. For
example, $\la + 3^\nu \lmodg \la +2.$ But in the general layered
setting, this is no longer true: $f = \la + \xl{3}{2}$ does not
surpass $g =  \la + \xl{2}{1},$ as seen by specializing $\la$ to
$a = \xl{2}{2},$ since now $f(a) = \xl{3}{2}$ whereas $g(a) =
\xl{2}{3}$. This phenomenon affects  roots and varieties.
\end{rem}

 This rather general framework encompasses some very useful concepts, when we
 work with some other construction such as polynomials, to be considered shortly.

We may take $\tSS \subseteq R_1^{(n)}$ when we want to restrict
our attention to functions evaluated on tangible elements. To
obtain a richer but more complicated structure, we could take
$\tSS = R^{(n)}$ or $\tSS =  (R \cup \{ \rzero\})^{(n)}$, and work
with $\Fun (S, R' )$ for a suitable extension $R'$ of $R$.

 Later, when we
define layered varieties, we can take $\tSS$ to be a given layered
variety.

\subsubsection{Roots of polynomials}

The notion of root is crucial in geometry. In the standard
supertropical theory recall that an element $\bfa \in \tSS$ is a
(ghost) root of a function $f \in \mcR$ iff $f(\bfa)$ is ghost,
which can occur in two ways, either by an essential monomial (at
$\bfa$) with ghost coefficient (cluster root), or by a pair (or
more) of quasi-essential monomials (corner root). The situation in
the general layered theory is analogous, although we need to deal
with different ghost layers.

\begin{defn}\label{rootlev}  An element $\bfa\in \tSS$ is an $ \ell$-\textbf{root}
of $f\in \mcR$ if $f(\bfa)$ is an $ \ell$-ghost.
\end{defn}

 \begin{rem}\label{lg}
 $\vmap_f ^{-1}(L_{>1})$ is
just the set of 1-roots of $f$ in $S$. \end{rem}

More important for us is the notion of ``corner root.'' We write a
(rational) polynomial~$f = \sum  h_i,$ where the $h_i$ are
(rational) monomials, and call the $h_i$ the \textbf{monomials of}
$f$.

 \begin{defn}\label{rootlev10}  An element $\bfa\in \mcS $   is a \textbf{corner
root} of a (rational) polynomial~$f$ if $f(\bfa) \ne h(\bfa)$ for
each (rational) monomial $h$ of $f.$
\end{defn}

 \begin{exampl} Take the uniform $\Net$-layered \domain0 $R = R( \Net,\Real)$, written in logarithmic notation, and let $f = \xl{\la}{j}  + \xl{2}{1}$ and $g =\xl{\la}{k}  + \xl{3}{1}.$
 For any $  a    \in R_\ell,$
 we have the following values of $\{ s(f(a)), s(g(a))\}$:

\begin{equation}\label{prob}
\begin{cases} \{j\ell,k\ell\}  & \quad \text{for} \quad  a >
 3;\\
 \{j\ell,k\ell\!  +\! 1\}  & \quad \text{for} \quad  a =
 3;\\
\{ j\ell,1\} &  \quad \text{for}  \quad  2 <  a  <
 3 ;\\
 \{j\ell\! +1\! ,1\}  & \quad \text{for} \quad   a =2 ;\\
  \{1\}  & \quad \text{for} \quad   a  < 2  .
\end{cases}
\end{equation}

 \end{exampl}
In this way, we see that the layering map of a polynomial $f$
contains much information about the layers of the coefficients of
$f$, and we can distinguish sets of polynomials via their roots,
much more effectively than in the non-layered case. But of course
we might have to go to some extension of $R$ to find the roots of
$f$; for example if we take $f = \la ^2 + 3$ with $R =
R(\Net,\Int)$.

 \begin{lem}\label{corn1} The shell of any (rational) polynomial $f$ has at least
two monomials. (In fact, those monomials in the support having
maximal and minimal degree under the lexicographic order are
essential.)\end{lem}
\begin{proof} Increasing (resp.~decreasing) the indeterminate of
highest degree by a small amount yields a point at which the
monomial is essential.\end{proof}

\subsubsection{Properties of monomials}

Our usage of the word  ``monomial'' is in the appropriate context,
either for $R[\Lambda]$, $R[\Lambda,
 \Lambda^{-1}]$, or  $R[\Lambda]_{\rat}$. (Rational) monomials have several nice
properties. 

\begin{rem} We  write $\al \Lm^\bfi$, with $\bfi = (i_1, \dots, i_n)$, for the
monomial $\al \lm^{i_1}_1 \cdots \lm^{i_n}_n$.

\begin{enumerate}\eroman
%
%
%

\item  A rational polynomial is a monomial iff  its shell has no
more than one (essential) summand.  (Indeed, a monomial $h$ has
  a decomposition $h=h_1+h_2$ iff $h_1 = \a\Lambda^\bfi$
  and $h_2 = \bt\Lambda^\bfi$ for suitable $\a, \bt,$
  implying $h=  (\a+\bt)\Lambda^\bfi,$ so   $h_1$ and $h_2$ are quasi-essential.)
  \pSkip

 \item The only rational polynomials which are tangible are the monomials
 with tangible coefficients.

 \end{enumerate}
\end{rem}

\begin{rem}\label{path1} For $R$ is $1$-divisibly closed, define
 the \textbf{path} $\mcP(\bfa,\bfb)$ joining elements $\bfa$ and $\bfb$ of $R^{(n)}$ to be $ \{ \bfa^ {t  } \bfb^ { 1-t}  : t \in \Q, \ 0 < t <
1 \}$.  Paths
contain elements ``close'' to $\bfa.$ Monomials preserve paths, in
the sense that for every monomial $h(\la_1, \dots, \la_n)$, all
$\bfa, \bfb \in R^{(n)},$ and all $t \in \Q$,
\begin{equation}\label{mult1} h(\bfa ^t {\bfb} ^{1-t}) =h(\bfa)^t h(
\bfb)^{1-t}   , \end{equation} so $ h(\mcP(\bfa,\bfb))
=\mcP(h(\bfa),h(\bfb)). $
\end{rem}

\begin{rem}\label{multmon}
 As in
\cite[Lemma~5.20]{IzhakianRowen2007SuperTropical}, one sees the
following, for any monomials $h_1$ and $h_2$ and all  $\bfc \ne
\bfa, \bfb$ in the path $ \mcP(\bfa,\bfb)$:

 \begin{enumerate} \eroman
    \item If $h_1(\bfa ) \nuge h_2(\bfa )$ and
$h_1(\bfb) >_\nu h_2(\bfb),$ then $h_1(\bfc )> _\nu h_2(\bfc ) $;
 \pSkip

 \item If $h_1(\bfa ) > _\nu h_2(\bfa )$ and
$h_1(\bfb)\ge_{\nu} h_2(\bfb),$ then $h_1(\bfc) >_\nu  h_2(\bfc )
$;
 \pSkip

    \item If $h_1(\bfa ) \ge_{\nu} h_2(\bfa )$ and
$h_1(\bfb)\ge_{\nu} h_2(\bfb),$ then $h_1(\bfc )\ge_{\nu} h_2(\bfc
) $.
\end{enumerate}
\end{rem}

\subsection{The test algebra}

It could be that two polynomials $f$ and $g$ agree on all values
on $R^{(n)}$ but not on  all values on ${R'} ^{(n)}$, for an
extension $R' $ of $R$. In the extreme case, if $R = \{ \rone \}$
and $L = \{ 1 \},$ then all polynomials agree on $R$, but not
necessarily on extensions of $R$.

\begin{exampl}\label{trun2} The monomials $\la^q$ and   $\la^{q+1}$ agree on the  \semiring0 $[1,2^q]$
but not on its  extension  $[1,2^{q+1}]$, cf.~Example \ref{trun1}.
\end{exampl}

We want to choose a specific extension ${\tldR}$ that will check
this for all extensions. For example, in classical algebraic
geometry, for any integral domain $R$ one would take $\tldR$ to be
the algebraic closure of $R$.

 \begin{rem}   For the purposes of this remark only,
 we define an equivalence $\equiv_{R'}$ of (rational) polynomials by saying two
 (rational) polynomials $f,g$ satisfy $f \equiv_{R'} g$ iff $\psi_{\vrp}(f) =
 \psi_{\vrp}(g)$ (cf. \eqref{eq:polymap2})
 for every possible homomorphism $\vrp: R \to  R'.$
 The question arises as to how we choose $R' $.
 Presumably, we could have $f \equiv_{R' _1} g$ but not $f \equiv_{R' _2}
 g$ for extensions $R_1'$ and $R_2'$ of $R$. Let us call a layered extension $ \tldR$ a \textbf{test algebra} for a class
 of (rational) polynomials $\mcC$ over $R$
 if, whenever $f \equiv _{\tldR} g$,
 then $f (\bfa) =  g(\bfa)$ for all $\bfa \in R'$, for any extension~$R'$ of $R$. \end{rem}

\begin{lem}\label{test01}  Suppose $R$ is an $L$-layered \domain0, where
$R_1$ is a cancellative monoid with $R \ne \ve_L(R)$
(cf.~Lemma~\ref{inj}). Let $F$ be the $1$-divisible closure of the
layered $1$-\semifield0 of fractions of $R$. Then:

\begin{enumerate} \eroman
    \item
 $R$ itself is a test algebra for monomials. \pSkip

\item   $F$  is a test algebra for (rational) polynomials in one
indeterminate.
\end{enumerate}
\end{lem}
\begin{proof} (i): Suppose we have two monomials $f \ne g.$ Then
$fg^{-1}$ is a monomial $\ne \rone,$ and thus takes on some value
other than $\rone$ on $\bfa \in R^{(n)},$  $\bfa = (a_1, \dots,
a_n)$ since $R$ has elements $a_j \not \nucong \rone$, by the
argument given in the proof of Corollary~\ref{infinite}. \pSkip

(ii): Suppose that $f$ and $g$ agree on $R$. In view of (i), we
may assume that $f$ and $g$ agree on $F$. If a monomial $h$ is
essential for $f$ at some point $\bfa$ of $R^{(n)}$, then $h$ also
appears in $g$ and actually is essential for $g$ at  $\bfa$, in
view of (i) (noting that $F^{(n)}$ has points  ``close'' to $\bfa$
since $F$ is $1$-divisibly closed). We need to show that any
monomial $h$ of~$f$ which is inessential on $F$ remains
inessential on an extension. But the corner roots of $f$ all lie
in $F$ (since we only have one indeterminate and $F$ is divisibly
closed), and any corner root has infinitely many other points
``closer'' to it than the next corner root, seen by taking the
path joining them, so the assertion is obvious.
\end{proof}

The hypothesis that $R_1$ is cancellative is required, in view of
Example~\ref{trun2}.

\begin{rem}\label{cand} For $R$ a uniform $L$-layered \domain0, we take
$\tldR$ to be the completion  (Example~\ref{compl}) of the
$1$-divisible closure (Example~\ref{divcl}) of the $L$-layered
$1$-\semifield0 of fractions $F$ of $R$. If there is a test
algebra, there is one that contains $\tldR.$ Indeed, we need to
show that any (rational) polynomials $f,g$ that agree on an extension $R'$
already agree on $R$. Replacing $R'$ by the completion of its
$1$-divisible closure, we may assume that $R'$ is a complete,
$1$-divisibly closed, $L$-layered $1$-\semifield0. But then, in
view of Proposition~\ref{univ}, there is a layered embedding of
$F$ into~ $R'$, so we may replace $R$ by $F$ and assume that $R =
F$. But likewise, we can then embed the $1$-divisible closure, and
then its completion, into $R'$.\end{rem}


\begin{thm}\label{test1} $\tldR$ of Remark~\ref{cand} is a test algebra for
all (rational) polynomials in $R[\Lambda]$.
\end{thm}
\begin{proof} Although the easy argument of Lemma~\ref{test01}(ii)
is not immediately applicable, due to the fact that corner roots
are not necessarily defined over $F$, there is a direct argument
which comes from a standard theorem in inequalities. Starting with
the same argument as in Lemma~\ref{test01}(ii), we may assume that
$f$ and~$g$ have the same essential monomials with respect to
evaluations in $\tldR^{(n)},$ i.e., have the same shells (cf.~
Definition~\ref{decomp}), and need to show that if their
quasi-essential monomials define the same function on
$\tldR^{(n)}$, then they define the same function on ${R'}^{(n)}$,
for an arbitrary extension $R'$ of $R.$ In view of
Remark~\ref{cand}, we may assume that $R = \tldR.$

First we show that no inessential monomial $h$ of $f$ with respect
to~$\tldR^{(n)}$ becomes essential with respect to~${R'^{(n)}}.$
Since essentiality is determined according to $\nu$-values, we may
restrict evaluations to $R_1$, so this part of the proof is really
a statement about the max-plus algebra. Multiplying through by
$h^{-1},$ we may assume that $h$ is the constant $\rone$. Using
Remark~\ref{linalg}, we can view $\tldR^{(n)}$ as a vector space
over $\mathbb Q,$ which becomes a vector space over $\Real$ by
means of Remark~\ref{linalg2} Note that we have switched from
multiplicative notation on the monoid $R_1$ to additive notation,
and the constant $\rone$ becomes the ``zero'' vector, so the
inessentiality of $h$ translates to the lack of a solution vector
$\bold x = (x_1, \dots, x_n) $ with
\begin{equation}\label{fark} A \bold x < -\bold b ,\end{equation} where $\bold b =(\a _1, \dots,
\a_m)\in R_1^{(m)}$ and $A = (\a_{j,k})$ is the $m \times n$
matrix whose $(j,k)$-entry is the power $(j_k)$ of the monomial~
$h_\bfj = \a _j \la_1^{j_1}\cdots \la_n^{j_n}$ of $f$, i.e.,
obtained by taking the powers of the indeterminates.

Viewing $R_1$ as a complete ordered Abelian group, we can define a
metric on $R_1,$ which in turn provides a sup metric on
$R_1^{(n)}$. There is a famous theorem from the theory of real
inequalities, often called Farkas' Lemma (cf.~\cite{DSS},
\cite{GKT}, \cite{BV}), which has already been used in tropical
geometry in \cite{DS}. Pick some  element $c \in R_1$ for which,
in our original notation, $c < \rone$. (In the vector space
notation, $c<0$.) Farkas' Lemma implies that there is no solution
in $ R_1^{(n)}$ to Equation~\eqref{fark} iff there is a solution
in $ R_1^{(n)}$ for $ \bold y \in R_{\ge0}^{(m)}$ to the system
$$ \bold y^ {\tran} A = (0), \qquad \bold y^ {\tran}\bold b = -  c,$$
(which in the original notation means $c^{-1}$), where $^ {\tran}$
indicates the transpose.

The proof of Farkas' Lemma is topological, cf.~\cite{DSS},
and thus works for the metric space $R_1^{(n)}$ instead of~
$\Real^{(n)}$.  Farkas' Lemma is actually stated for $ c = 0$, but
one gets this modification by applying Farkas' Lemma to the matrix
inequality
\begin{equation}\label{fark1} \bigg(A \ds{\bigg|}(-c)\bold b \bigg) \binom {\bold x}{d} < (0),\end{equation}
noting that a solution $(\bold x, d )$, for $d>0$ would yield a
solution $\bold {x'} = (x_1', \dots , x_n')$ to
Equation~\eqref{fark} for $x'_i = \frac {x_i}{cd}$. But then there
cannot be a solution $ \bold x$ to Equation~\eqref{fark} in
$\tldR^{(n)}$, for otherwise
$$ 0 =  \bold y^ {\tran} A \bold x < -  \bold y^ {\tran} \bold
b = c.$$

Having disposed of inessential monomials, we only need to concern
ourselves with quasi-essential  monomials $h$. By definition, $h$
is quasi-essential at some corner root $\bfa = (a_1, \dots, a_n)$
of~ ${R'^{(n)}}$, and by definition, $\csupp(f)$  has at least two
distinct (quasi-essential) monomials at $\bfa$, $$\bt \la_1^{i_1}
\cdots \la_n ^{ i_n}, \qquad \gamma \la_1^{j_1} \cdots \la_n ^{
j_n},$$ where $(i_1 , \dots,  i_n) \geq (j_1 , \dots, j_n)$ in the
lexicographic order.  We claim that $f(\bfa) = g(\bfa)$. This
would yield the theorem, since we already know that the shells of
$f$ and $g$ are the same.

To prove the claim, we take the smallest $k$ such that $ i_k >
j_k,$ we have $\bt a_k^{i_k} \cdots a_n ^{ i_n} = \gamma
a_{k}^{j_k} \dots a_n ^{ j_n},$ and thus $$a_k ^{i_k-j_k} =
\frac{\gamma}{\bt} a_{k+1}^{j_{k+1}-i_{k+1}}\cdots
a_{n}^{j_{n}-i_{n}}.$$ Replacing $\la_k ^{i_k-j_k} $  in $f$ and
$g$ by $\frac{\gamma}{\bt}\la_{k+1}^{j_{k+1}-i_{k+1}}\cdots
\la_{n}^{j_{n}-i_{n}}$ throughout, reduces the number of
indeterminates by one, and thus by induction our claim holds for
these new rational polynomials at $\bfa$, and thus for $f$ and $g$
at~$\bfa.$ Since this is true for each possible corner root of
${R'^{(n)}},$ but the supports are defined over $\tldR$, we can
apply this argument in turn to all quasi-essential monomials.
\end{proof}

\begin{cor}\label{test2} The $1$-divisible closure of the $1$-\semifield0 of fractions of a uniform $L$-layered \domain0 $R$ is a test algebra for
all (rational) polynomials in $R[\Lambda]$.
\end{cor}
\begin{proof} If two polynomials $f$ and $g$ differ on $\bfa =
(a_1^{\gamma_1}, \dots,  a_n^{\gamma_n}),$ then say $f$ has a
monomial that dominates~$g$ at~$\bfa$ and thus at $\bfa ' =
(a_1^{\gamma'_1}, \dots,  a_n^{\gamma'_n}),$ where the $\gamma'_j$
are rational numbers close enough to the $\gamma_j.$
\end{proof}

 A more elegant way of obtaining these results, treated in \cite{IKR4}, is to show that $\tlR$ is model complete, since the theory of
ordered divisible Abelian groups is model complete,
\cite[Corollary~3.17]{M}. See \cite[\S4.3]{M} or \cite{VdD} for a
general model-theoretic approach, which is being developed in the
layered situation by T.~Perri.

Many rational polynomials that would be equivalent in the standard
 supertropical situation  now
are not equivalent, since quasi-essential monomials could force a
root to a different sort; see Example \ref{multroots}(iii), (iv),
and (v) below. In this setting, we may still discard all monomials
strictly dominated by $f$.

\section{Layered topologies}\label{Zar00}

We now get towards one of the main issues -- how does one define
layered varieties in such a way as to relate to tropical geometry?
There are two approaches -- the first is related more to a version
of the Nullstellensatz, whereas the second is linked more to the
simplicial structure of the Newton polytope. To avoid technical
complications, we assume throughout this section that the sorting
set $L$ is
    totally ordered and non-negative.

\subsection{The layered component topology}\label{Zar0}
First we look at a relevant topology, assuming $S \subseteq
R^{(n)}$, which reflects further on some of the concepts involving
the
    Nullstellensatz in \cite{IzhakianRowen2007SuperTropical}.

\begin{defn}\label{Zar1} Write $f = \sum_\bfi f_\bfi $, a sum of monomials in $R[\Lm]$,
where $f_\bfi = \a _\bfi \Lambda
^\bfi$  for $\bfi = (i_1,\dots, i_n)$.
 Define the \textbf{components} $D_{f,\bfi} \subseteq S$
of $f$  to be
$$D_{f,\bfi} := \{ \bfa  \in \tSS : f(\bfa) = f_\bfi (\bfa)
\}.$$ For $k_1, \dots, k_n \in
L$, the $(k_1, \dots, k_n)$-\textbf{layer} 
of the component $D_{f,\bfi}$ is
$$\xl{D_{f,\bfi}}{(k_1, \dots, k_n)} := \{  \bfa \in D_{f,\bfi}\ds : \
s(a_j) = k_j, \ 1\le j \le n, \quad \text{where}  \quad \bfa =
(a_1, \dots, a_n)\}.$$
\end{defn}

\begin{prop}\label{Zar3} Suppose $f = \sum_\bfi  f_\bfi$ and $g = \sum_\bfj  g_\bfj$. Then
$$D_{f,\bfi}\cap D_{g,\bfj} =  D_{fg,\bfi+\bfj}.$$\end{prop}
\begin{proof}
$f_\bfi g_\bfj$ is one of the monomials $h_{\bfi + \bfj}$ of the
product $fg,$ so  $D_{f,\bfi}\cap D_{g,\bfj} \subseteq D_{fg,
\bfi+ \bfj}$, and $fg(\bfa) = f_\bfi (\bfa)g_\bfj (\bfa)$ on
$D_{f,\bfi}\cap D_{g,\bfj}$. Conversely, for any $\bfa \in \tSS$
for which $fg(\bfa) = f_\bfi (\bfa)g_\bfj (\bfa)$, we must have
$$fg(\bfa) \lmodWL f_\bfi (\bfa)g_\bfj (\bfa) = fg(\bfa),$$
implying $f(\bfa) = f_\bfi(\bfa)$ and $g(\bfa) = g_\bfj(\bfa)$,
and thus  $D_{fg, \bfi+ \bfj} \subseteq D_{f,\bfi}\cap
D_{g,\bfj}.$
\end{proof}

\begin{cor} The set of components of
polynomials comprises a base for a topology on $\tSS.$
\end{cor}

\begin{defn}\label{Zar4} We call this topology the  \textbf{layered component topology}.\end{defn}

 Given $\bfa = (a_1, \dots, a_n)$ and $\bfa' = (a'_1, \dots,
a'_n)$, we write $\bfa \nucong \bfa'$ if $a_i \nucong a'_i$ for
each $i = 1,\dots,n$.

\begin{lem}\label{lem:z1}  Suppose $R$ is  $L$-layered \domain0. If  $\bfa \in D_{f,\bfi},$ then $\bfa'
\in D_{f,\bfi}$ for all $\bfa' \nucong \bfa$ in $ \tSS$.\end{lem}
\begin{proof} Assume $\bfa' \notin
D_{f,\bfi}$, then $f(\bfa') \nucong f_{\bfj} (\bfa')$  for some
$f_\bfj \neq f_\bfi$. Since $\bfa \nucong \bfa'$, we also have
$$f_{\bfi} (\bfa) \nucong f_{\bfi} (\bfa') \nucong  f_{\bfj}
(\bfa')  \nucong f_{\bfj} (\bfa),$$  a contradiction for $\bfa$
being in $D_{f,\bfi}.$
\end{proof}

In other words, when $R$ is $L$-layered with each $\nu_{k,1}$
onto, any component is determined by each $(k_1, \dots,
k_n)$-layer, for arbitrary $k_1, \dots, k_n \in L$. In particular,
when determining the components, we may focus on the $(1,  \dots,
1)$-layer, i.e., tangible vectors $\bfa \in R_1 ^{(n)}.$

\begin{lem}\label{lem:z2} If $f = \sum_\bfi f_\bfi$ and $f' = \sum_\bfi f'_\bfi$
 with   $f_\bfi \nucong f'_\bfi$ for each $\bfi$, then $f$ and $f'$  have the same sets of components.
\end{lem}
\begin{proof} Dominance of the monomial $f_\bfi$ is determined by the
$\nu$-value, not by the layer.  Explicitly, for any  $\bfa \in
D_{f,\bfi}$ we have  $f(\bfa) = f_\bfi(\bfa) \nucong f'_\bfi
(\bfa)$. If $f'(\bfa) = f'_\bfj(\bfa)$, for some $\bfj \neq \bfi$,
then,  since $f_\bfj \nucong f'_\bfj$,  we get $f(\bfa) \nucong
f_\bfi(\bfa) \nucong f_\bfj(\bfa)$ -- a contradiction.
\end{proof}

\begin{lem}\label{lem:z3} $g = (\sum_\bfi f_\bfi)^k$ and $h = \sum_\bfi
f_\bfi^k$ have the same sets of components.
\end{lem}
\begin{proof} Derived directly from the Frobenius-type  property, $(\sum_\bfi f_\bfi )^k \nucong  \sum_\bfi
f_\bfi^k$, c.f. Remark \ref{rmk:LsurFun}.
\end{proof}

%
%
%

We~ can understand the layered component topology in terms of the
layering maps. First, we recall a well-known fact from algebra,
stated in the context of \semirings0.

\begin{lem}\label{Van} Suppose $L \ne \{1\}$ is a cancellative as well as $\Net$-cancellative monoid. Then for any
finite set $\{f_1, \dots, f_m\} \subset L [\Lm]$ of monomials with
coefficients in $L$, there exist $k _1, \dots, k_n \in L$ such
that all the $f_j(k_1, \dots, k_n)$ are distinct (since there are
finitely many of these).
 \end{lem}
 \begin{proof} A standard induction argument on $n$. For $n=1,$ we
 note that $a k_1^ i = b k _1^j$ for $i \le j$ iff $a = b
 k_1^{j-i},$ which has a unique solution for $k_1.$ But $L$ is
 infinite, by Lemma~\ref{inf1}, so almost all elements of $L$ will
 satisfy the conclusion of the lemma.

 In general, write $$f_i(\la_1, \dots, \la_{n}) =  \overline{f_i}(\la_1, \dots,
 \la_{n-1})\la _n^{i_n},$$ and choose $a_1,
 \dots, a_{n-1}$ such that the $ \overline{f_i}( a_1, \dots,
 a_{n-1})$ are distinct whenever the $ \overline{f_i}(\la_1, \dots,
 \la_{n-1})$ are distinct.
 Then by the previous paragraph almost all choices of $a_n$ will
 yield the conclusion of the lemma.
 \end{proof}

\begin{thm} Suppose $L$ is a cancellative monoid as well as $\Net$-cancellative monoid, and
$R$ is uniformly $L$-layered. Then the layering map of a
polynomial $f\in R[\Lambda]$ determines its components. More
precisely, if $D_{f,\bfi}$ are the components of $f = \sum_\bfi
f_\bfi$, there exist $k_1, \dots, k_n \in L $ such that letting
$\ell_\bfi = \vmap_{f_\bfi}(\bfa) $ where $\bfa = (a_1, \dots,
a_n)$ is in the $(k_1, \dots, k_n)$-layer of the component
$D_{f,\bfi}$, we have
$$ \xl{D_{f,\bfi}}{(k_1, \dots, k_n)} = \{\bfa \in S\cap (R_{k_1} \! \times \! \cdots \! \times\!
R_{k_n}): \vmap_{f}(\bfa) = \ell_\bfi\}.$$
\end{thm}
\begin{proof} Write $f = \sum_\bfi \a_\bfi \la_1^{i_1}\cdots
\la_n^{i_n},$ and $\bfa =  (\xl{ a_1 }{k_1}, \dots,  \xl{ a_n
}{k_n} );$ then $$\vmap_{f_\bfi}(\bfa ) = s(\al_\bfi)\prod
_{j=1}^n k_j^{i_j}.$$

In view of Lemma~\ref{Van}, there are $k_1, \dots, k_n \in L $
such that, letting  $\ell_\bfi = \vmap_{f_\bfi}(\bfa ) $, the sums
of the $\ell_\bfi$ are distinct. But by definition $\ell_\bfi$ is
the sort corresponding to $f$ evaluated on the elements of $
\xl{D_{f,\bfi}}{(k_1, \dots, k_n)},$ as desired. In other words,
the layering
 map distinguishes among the various components of $f = \sum_\bfi
 f_\bfi.$
 \end{proof}

\subsubsection{The layered Nullstellensatz}

The layered component topology is rich enough environment to
 formulate the Nullstellensatz of
 \cite{IzhakianRowen2007SuperTropical}. It is convenient
 to assume that $R$ is a 1-divisibly
closed,  $L$-layered 1-\semifield0, where $L$ is a cancellative
monoid with $L = L_{\ge 1}.$ We also assume all the sort
transition maps of $R$ are onto.

In analogy to \cite{IzhakianRowen2007SuperTropical}, we write $f
\preceq _{D_{f,\bfi}} g$
 if some component $D_{g, \bfj}$ of
$g$ contains ${D_{f,\bfi}}$; we write $f \epsc A $ for $A
\subseteq R[\lm_1, \dots, \lm_n],$ if for every essential monomial
$f_\bfi$ of $f$ there is some $g\in A $ (depending on
$D_{f,\bfi}$) with $f \preceq _{D_{f,\bfi}} g$. Restricting this
definition by bringing  in the sort map $\vmap$, we get the
following:

 \begin{defn}\label{epst1}
We write $f \preceq_{D_{f,\bfi}}^\vmap g$ if $f
\preceq_{D_{f,\bfi}} g$ and $\vmap_f(\bfa) \ge \vmap_g (\bfa)$ for
all
  $\bfa\in D_{f,\bfi}, $ and define $f \epsc^\vmap A $ for $A  \subseteq R[\lm_1, \dots,
\lm_n],$ if for every essential monomial $f_\bfi$ of $f$ there is
some $g \in A $ with $f \preceq _{D_{f,\bfi}}^\vmap g$.
\end{defn}

We want to check this property for the tangible part of
components.
 We say that set $S \subset R^{(n)}$ is
\textbf{tangibly compatible} (with respect to $R$) if
 whenever $\bfa \in \tSS$  we also have $\bfa' \in \tSS$ for all $\bfa' \in
R_1^{(n)}$ such that $\bfa \nucong \bfa'.$

\begin{rem}\label{cor:z2} If the sort transition maps $\nu_{\ell,1}: R_1 \to R_\ell $
of $R$ are   onto for each $\ell$ in $L$ and $\tSS \subseteq R^{(n)}$ is
tangibly compatible, then for each $\bfa \in D_{f,\bfi}$  there exists
$\htbfa \nucong \bfa,$ where  $\htbfa \in \tSS \cap R_{1}^{(n)}$.

In this case, the components are determined by the 1-layer.
\end{rem}

\begin{lem}\label{lem:z4} Suppose  $\tSS \subseteq R^{(n)}$ is
tangibly compatible, and $f \preceq_{D_{f,\bfi}}^\vmap g$, with $
f_\bfi = \al _ \bfi \lm_1^{i_1} \cdots \lm_n^{i_n}$. Then there is  some
 monomial $g_\bfj = \bt _ \bfj \lm_1^{j_1} \cdots \lm_n^{j_n}$ of
 $g,$ for which $s(\al_\bfi) \geq s(\bt_\bfj)$ and $i_k \geq j_k$  for every $k = 1,\dots, n.$
In fact, $g_\bfj $ can be taken to be the dominant monomial of $g$ at $\bfa,$ for any $\bfa \in D_{f, \bfi} \cap R_1^{(n)}$. (Note that $ D_{f, \bfi} \cap R_1^{(n)}\neq \emptyset$,  by
Corollary \ref{cor:z2}.)
\end{lem}
\begin{proof}
Pick $\bfa \in D_{f, \bfi} \cap R_1^{(n)}$, which exists by
Corollary \ref{cor:z2}, then $s(\al_\bfi) = \vmap_f(\bfa) \geq
\vmap_g(\bfa) = s(\bt_\bfj)$. Assume that $j_k > i_k$ and pick
$\bfa = (a_1, \dots, a_n)\in R^{(n)}$, with $a_k \in R_\ell$ and
$\ell = s(\al_\bfi)$ then
$$ \vmap_f(\bfa) =  s(\al_\bfi) \ell^{i_k} =  \ell^{i_k + 1}  \geq  s(\bt_\bfj) \ell^{j_k} = \vmap_g(\bfa)$$
and hence, since $j_k > i_k$, $0 > s(\bt_\bfj) \ell^{j_k - i_k -1}
\geq s(\bt_\bfj)$ -- a contradiction to $L = L _{\geq 1}$.
\end{proof}


We say that an ordered monoid is called  \textbf{archimedean} when
for any $a,b
>1$ there is some $k\in \Net $ such that $a^k > b.$
For $A \subseteq R[\Lambda],$ define
\begin{equation}\label{rad1} \root\Lay\of{A } = \{ f \in R[\Lambda] : f^k
\lmodWL g \text{ with  } f^k \nucong g \text{ for some } g \in
A\}.\end{equation}

\begin{thm}\label{Null2} \textbf{(Layered Nullstellensatz)}
Suppose  $L = L_{\ge 1}$ is archimedean, and $R$ is a 1-divisibly
closed,  $L$-layered 1-\semifield0 whose  sort transition maps are
all onto, such that $R_1$ is archimedean. Suppose $A \triangleleft
A[\Lambda],$ and $f \in R[\Lambda].$ Then
$$f \epsc^\vmap A   \quad \text{ iff } \quad  f \in \root\Lay\of{A }  .$$
\end{thm}
\begin{proof} $(\Rightarrow)$ Let $f = \sum_\bfi f_\bfi$, and
 take $\htf = \sum_\bfi \htf_\bfi \in R_1[\Lm]$, which exists
since the sort transition maps assumed  to be onto. Similarly,
define  $\htg = \sum_\bfi \htg_\bfi \in R_1[\Lm]$, for every $g
\in A$, and let $\htA := \{ \htg \ds | g\in A\} \subset R_1[\Lm].$

In  view of
 Corollary \ref{cor:z2},  restricting the components of $\htf$ and all $\htg \in \htA$ to $S_1 := S \cap
R_1^{(n)}$, each $S_1 \cap D_{\htf,\bfi}$ is contained in $S_1
\cap D_{\htg,\bfj}$ for some $\htg \in \htA$. Then, since $\htf$
is tangible, by
\cite[Theorem~7.17]{IzhakianRowen2007SuperTropical}, for some $m
\in \Net$
\begin{equation}\label{eq:null1} (\htf)^m = \bigg(\sum_\bfi \htf_\bfi\bigg)^m= \sum_{\htg \in
\htA'} \hth \htg + \text{ghost} , \qquad \htA' \subseteq \htA, \
\end{equation}
where 
$\hth =  \sum_\bfk \hth_\bfk = \htgm_\bfk \Lm^\bfk$
 are
polynomials in $R_1[\Lm].$ Let $\htPhi := \sum_{\htg \in \htA'}
\hth \htg.$

Note that by steps 2 and 3 in the proof of
\cite[Theorem~7.17]{IzhakianRowen2007SuperTropical}, we know that
for every $\htf_\bfi^m$ in the expansion of $\htf^m$ we have
$\htf_\bfi^m = \hth_{\bfk} \htg_\bfj$ for distinct $\htg \in
\htA'$. In particular
\begin{equation}\label{eq:null2} \htf^m(\bfa)
= \htf_\bfi^m (\bfa) = (\hth_{\bfk} \htg_\bfj) (\bfa) = (\hth
\htg) (\bfa), \qquad \text{for every } \bfa \in D_{f,\bfi},
\end{equation}  and $\htf^m(\bfa') \geq_\nu (\hth_{\bfk}
\htg_\bfj) (\bfa')$ for any $\bfa' \in S \setminus D_{f,\bfi}$.

We next need to coordinate the layering of the different monomials
$f_\bfi^k$ in \eqref{eq:null1}. Since $\htf_\bfi^m = \hth_{\bfk}
\htg_\bfj$ we have \begin{equation}\label{eq:null3} f_\bfi^m =
\al_\bfi^m  \Lm ^{m\bfi} \nucong \hth_{\bfk} g_\bfj = \htgm_\bfk
\Lm ^\bfk \bt_\bfj \Lm^\bfj = \htgm_\bfk  \bt_\bfj \Lm^{\bfk +
\bfj}
\end{equation}
with $s(\htgm_\bfk) =  1$. By hypothesis, $\vmap_{f_\bfi}(\bfa)
\geq \vmap_{g_\bfj}(\bfa)$ for each $\bfa \in D_{f,\bfi}$, and
thus $\vmap_{f^m}(\bfa) = \vmap_{f_\bfi^m}(\bfa) \geq
\vmap_{g_\bfj}(\bfa)$. We may also assume that
\begin{equation}\label{eq:null4} \vmap_{f^m}(\bfa)
 \geq \vmap_{\Phi}(\bfa), \qquad \text{for every }  \bfa \in S,
\end{equation} since otherwise
 we can take $m' := m + \ell $, for $\ell$ large enough
and replace  each $\hth $ respectively by $\hth \htf^\ell_\bfi$,
preserving~\eqref{eq:null2}. (Note that
 $f$ has finitely many
  components.)

Then, $s(\al_\bfi^m) =  s(\al_\bfi)^m \geq s(\bt_\bfj)$, by Lemma
\ref{lem:z4}. Let $\ell \in L $  be such that  $ \ell s(\bt_\bfj)
=   s(\al_\bfi)^m$.  Take $\gm_\bfk \nucong \htgm_\bfk$, where
$\gm_\bfk \in R_\ell$, and define $h =  \sum_\bfk h_\bfk$ to be
$\hth$ with each $\hth_\bfk$ replaced by $h_\bfk = \gm_\bfk
\Lm^\bfk $ to get $f_\bfi^m = h_{\bfk} g_\bfj$.
 In conjunction with \eqref{eq:null4} we
get $ f^m \lmodL \sum_{g \in A'} h g$.

 $(\Leftarrow)$ Taking $g$ as in \eqref{rad1}, we see that
 $$ D_{f,\bfi} = D_{f^k,\bfi} = D_{g,\bfi}  .$$
\end{proof}

\subsubsection{Layering maps}

Recall that $\tR$ is a sub-semiring of $\FunSR$.

\begin{defn} The \textbf{layering map} of a set ${\tI }\subset
\mcR$ is the map $\vmap_{\tI }: \tSS\to L$ given by
$$\vmap_{\tI }(\bfa) := \min \{ \vmap_f(\bfa): f \in {\tI } \} ,$$ where
$\vmap_f$ is given in Definition~\ref{laymap}.\end{defn}

This definition carries the implicit assumption that $\min \{
\vmap_f(\bfa): f \in {\tI } \}\in L$ for every $\bfa \in \tSS$.
There are several ways to attain this:
\begin{enumerate}
\item  ${\tI }$ is finite. \pSkip

\item $L$  satisfies the descending condition (such as  $L =
\Net$). \pSkip

 \item  $L$ is complete and bounded from below (such as
$\Real_{\ge 1}$).  \pSkip

\end{enumerate}

%
%
%


\begin{exampl} We take the uniform $\Net$-layered \domain0 $R = \R(\Net,(\Real,+))$.
\begin{enumerate}

\item $f_k =  \la_1^k + \la_2 + 0$ for $k \in \Net  .$  Then for
$\bfa = (a_1,a_2)$ tangible we have
$$\vmap_{f_k}(\bfa) = \begin{cases}   3 & \text{ for } a_1 = a_2 = 0;
\\ 2 & \text{ for } a_1 = 0 >  a_2
\quad \text{ or }\quad  a_2 = 0 >  a_1 \quad \text{ or }\quad  a_1
^k =  a_2 > 0; \\ 1 & \text{ otherwise. }
\end {cases}$$

\item ${\tI } =\{ f_k : k \in \Net \}$,  $\bfa = (a_1,a_2)$
is tangible.
$$\vmap_{{\tI }}(\bfa) = \begin{cases}   3 & \text{ for } a_1 = a_2 = 0; \\ 2 & \text{ for } a_1 = 0 >  a_2
\quad \text{ or }\quad  a_2 = 0 >  a_1;  \\ 1 & \text{ otherwise.
}\end {cases}$$
\end{enumerate}
\end{exampl}

\begin{lem} Suppose $L = L_{\ge 1}.$ If $\tI  = \sum _{j\in J} \mcR  f_j,$ then
$$\vmap_{\tI }(\bfa) = \inf _j \{ \vmap _{f_j}(\bfa) : j \in J \}.$$ \end{lem}
\begin{proof} $(\le)$ is clear. But for any $g = \sum_j g_j f_j\in \tI ,$ $g_j \in \tR$, we
have $$\vmap_{\tI}(\bfa) \geq s(g(\bfa)) \ge \min_{j \in J } \{
s((g_j f_j)(\bfa)) \} =  \min_{j \in J} \{s(g_j(\bfa))
s(f_j(\bfa)) \} \ge \min_{j \in J} \{ s(f_j(\bfa)) \} = \min_{j
\in J}  \{ \vmap_{f_j}(\bfa) \}.$$
\end{proof}

\begin{prop} For $\tI _j \subset  \Fun (\tSS,R'),$
$$  \vmap _{\sum _j \tI _j} = \vmap _{\cup _j \tI _j} = \inf _j \{ \vmap _{\tI _j} \} ;\qquad   \vmap _{\tI _1\tI _2}=   \vmap _{\tI _1 }  \vmap _{\tI _2} .$$
\end{prop}
\begin{proof} The first assertion is immediate, and the second is clear since the sort map is
multiplicative. If $s(f_i(\bfa)) = \ell_i$ for $i = 1,2$, then
$s(f_1(\bfa)) s(f_2(\bfa)) = \ell _1 \ell _2 = s(
(f_1f_2)(\bfa)).$
\end{proof}

In the other direction, we can describe ideals of functions in
terms of layering maps.

\begin{defn}\label{geomid}  Given a sub-semiring $\mcR $ of
$\FunSR$,and $Z\subseteq S,$ define $\mathcal R_Z = Z \cap
\mathcal R.$  Given any map $\vmap : Z\to L$ where $Z \subseteq
\tSS$, define
  $\tI _{\vmap}(Z)$ to be $$ \tI _{\vmap}(Z):= \text{$\{ f \in
\mcR  : f(\bfa) $ is  $\vmap(\bfa)$-ghost, $\forall \bfa \in Z
\}.$}$$

A  \textbf{geometric layered ideal} of $\mcR $ is an ideal of the
form $\tI _{\vmap}(Z)$ for a suitable
 map $\vmap : Z\to L$. When $Z$ is understood, we write $\mcI_{\vmap}$ for $\tI _{\vmap}(Z)$.
\end{defn}

Strictly speaking, the notation for $Z$ is redundant, since we can
choose $\tSS$ as we please.  But often we start with $\tSS =
R^{(n)}$, and then take $Z$ to be a closed subset of $S$ with
respect to the layered component topology, so we utilize the
symbol $Z$ for clarification.

\begin{prop}\label{Zarcorresp} Suppose $L = L_{\ge 1}.$ Then
   $\tI _{\vmap}(Z)\triangleleft \tR_Z$, and there   are 1:1 order-reversing correspondences between
the layering maps of $\mcR $ and the geometric layered ideals of
$\mcR $, given by $\vmap \mapsto \tI _{\vmap}(Z)$ and $I \mapsto
\vmap_I$. \end{prop}
\begin{proof} Clearly $\vmap_I $ is closed under addition, and $\tI _{\vmap}(Z)$ is an ideal when $L = L_{\ge 1},$ since
then the  layering map increases.  For the second assertion, one
just follows the standard arguments in the Zariski correspondence.
Namely, we need to show that for any layering map $\vmap$,
defining the  geometric layered ideal $I = I_\vmap$, that $\vmap
_I = \vmap$ and $I_{\vmap_I} =  I$.

Clearly $I \supseteq I_{\vmap_I}$. But if $f\in I$ then by
definition $f \in \vmap_I.$ Hence $\vmap _I = \vmap$, so
$I_{\vmap_I} = I_\vmap = I$.
 \end{proof}


The hypothesis that   $L = L_{\ge 1}$ is crucial, since otherwise
we could multiply by $e_\ell$ for $\ell<0$ and the definition of
$\vmap _I$ would become meaningless.

These results indicate that tropical geometry can be understood
through a careful study of the algebraic structure of the layering
maps, as translated to $\mcR $.

\begin{defn}\label{layirr} The layering map $\vmap_{{\tI }}$ is
\textbf{irreducible} if $\vmap_{{\tI }}$ cannot be written as the
product $\vmap_{{\tI _1}}\vmap_{{\tI _2}}$ of two layering
maps.\end{defn}

\subsection{Layered varieties}

We take the standard approach of algebraic geometry, but need to
modify it because we do not have negation.  Due to space
limitations, we give only a rough outline, leaving details for a
separate paper. We need a concise algebraic definition of variety,
at least in the affine setting.

\subsubsection{The corner locus}

To introduce layered varieties, we make Definition~\ref{rootlev10}
more explicit.

\begin{defn}\label{rootlev1}   Given a rational polynomial $f = \sum_\bfi h_\bfi$
written as a sum of rational monomials $h_\bfi$, for $\bfi =
(i_1,\dots, i_n)$,
 define the \textbf{corner support}   at $\bfa,$ denoted
 $\csupp _\bfa (f)$, to be the set of those $h_\bfi$ for which $s(h_\bfi (\bfa)) > 0$ and $f(\bfa)
 \nucong h_\bfi (\bfa).$\end{defn}

\begin{defn}\label{rootlev2} An element $\bfa\in \tSS$   is a \textbf{corner root} of~$f$ iff $|\csupp_\bfa (f)|\ge
2.$ We define $\mcZ_\corn(f)$ to be the set of corner roots of
$f.$
\end{defn}

Thus, for any corner root $\bfa$ of $f$, there are at least two
$h_\bfi $ which are  quasi-essential in~$f$ at~$\bfa$, for which
$s(h_\bfi(\bfa)) \in L_+$ are positive, and $f(\bfa)$ is
$s(h_\bfi(\bfa))$-ghost for these $h_\bfi$.

 \begin{rem}\label{lg}
 $\vmap_f ^{-1}(L_{>1})$ is
just the set of 1-roots of $f$ in $S$. \end{rem}

In classical algebraic geometry, given a polynomial $f$, one takes
its zero locus. Our
 layered analogy is to take its set of corner roots,
which we call the \textbf{corner locus} of $f$. Note that the
complement set of the corner locus is the union of the components
of $f$. This motivates the next definition.

In order to hone in on corner roots of (rational) polynomials,  we
modify Definition~\ref{geomid}. For convenience, we take $\mcR
\subseteq R[\Lambda]_{\rat}.$

 \begin{dig}\label{Zar2} We could mimic Definition~\ref{Zar1} by defining   $f_\bfa :=
 \sum_{\bfi} \{h_\bfi :   h_\bfi \in \csupp _\bfa (f)\}$; the same argument
 as in Proposition~\ref{Zar3} shows that the sets $ D_{f,\bfa} =
 \{\bfb \in S: f(\bfb ) = f_\bfa ( \bfb)\}$ are a base for a topology
 that refines the layered component topology of
 Definition~\ref{Zar4}. This topology better reflects the
 simplicial nature of tropical geometry.
\end{dig}

 \begin{defn}\label{geomid2}  Given $Z \subset S$,  define   $$ \tI _{\corn}(Z):= \{ f \in
\mcR  : |\csupp _\bfa (f)|\ge
 2, \ \forall \bfa \in Z\}.$$ A  \textbf{corner layered ideal} of $\mcR $ is
an ideal of the form $\mcI_{\corn}(Z)$.

The \textbf{corner locus} $\mcZ_{\corn}({I })$ of a subset ${I
}\subset \mcR $ is the intersection of the corner loci
$\mcZ_{\corn}({f })$ of the functions  $f$ in ${I }$. Any such
corner locus will also be called an (affine) \textbf{layered
variety}.
\end{defn}

\begin{rem}   We lose the specific layers used in computing
$\tZ_{\corn}({\tI }).$ Furthermore, this process leads to
unexpected varieties, often arising as degenerate intersections of
usual tropical hypersurfaces. For example, if
$${\tI }_1 = \{ \la _1 + \la _2 +0,\ \la _1 + \la _2 +(-2)\},\qquad {\tI }_2
= \{ \la _1 + \la _2 +0,\ \la _1^2 + \la _2 +0 \},$$ then $Z_1 :=
\mcZ_{\corn}({\tI _1}) = \{ (a,a): a \ge_\nu 0 \}$, a ray which is
the intersection of two tropical lines not in general position,
which is not a customary tropical variety. Furthermore $Z_2 :=
\tZ_{\corn}({\tI }_2)$ is the union of the other two rays in the
tropical line $\tZ_{\corn}( f ),$ where $ f = \la _1 + \la _2 +0$.

In this way, it might seem that we could reduce the tropical line
as the union of two layered varieties. This is not desirable,
since one would want the tropical line to be an irreducible
variety, and its layering map is irreducible, in terms of
Definition~\ref{layirr}.
 Note that $\vmap_{{\tI }_1}(\bzero) = 2$ and
$\vmap_{{\tI }_2}(\bzero) = 3$ for $\bzero = (0,0)$, so
$$\vmap_{{\tI }_1}(\bzero)\vmap_{{\tI }_2}(\bzero) =6  > \vmap_{f}(\bzero) = 3.$$
\end{rem}

\begin{lem}\label{geomid31} For monomials $h_i$ of $f_i$, we have
$h_i \in \csupp_{\bfa}(f_i)$ for $i = 1,2$ iff  $h_1 h_2 \in
\csupp_{\bfa}(f_1f_2)$.
\end{lem}
\begin{proof} Each monomial dominates at $\bfa$ iff the product
dominates at $\bfa$. For if $h_1' h_2' (\bfa)
> _\nu h_1 h_2 (\bfa)$ for some other monomial $h_1'h_2',$ then  $h_1'   (\bfa)
> _\nu h_1   (\bfa)$ or $h_2'   (\bfa)
> _\nu h_2   (\bfa)$.\end{proof}

 \begin{lem}\label{geomid3} If $|\csupp _{\bfa} (f)|\ge
 2,$ then $|\csupp _\bfa (fg)|\ge
 2$ for all $g \in R[\Lambda]_{\rat}.$\end{lem}
\begin{proof} Suppose $h_1 \ne h_2$ are the rational monomials in $\csupp _\bfa (f)$  of respective
lowest and highest degree (under the lexicographic order), and
$h'_1$ and $h'_2$ are the rational monomials in $\csupp _\bfa (g)$
of  lowest and highest degree. Then $h_1 h_1'$ and $h_2 h_2'$
differ and are the rational monomials in $\csupp _\bfa (fg)$  of
respective lowest and highest degree, implying $|\csupp _\bfa
(fg)|\ge
 2$.
\end{proof}

 \begin{lem}\label{geomid4} $|\csupp _\bfa (f+g)|\ge \max\{ |\csupp _\bfa (f)|,|\csupp _\bfa (g)| \}
 .$\end{lem}
\begin{proof} All monomials remain in the support of the sum.
\end{proof}

\begin{prop}\label{geomid5}
$ \tI _{\corn} (Z) \triangleleft \mcR ,$ for every $Z \subseteq
\tSS.$\end{prop}
\begin{proof} Combine Lemmas~\ref{geomid3} and \ref{geomid4}.
\end{proof}

\begin{prop}\label{Zarcorresp2} There   are 1:1 order-reversing correspondences between
the corner loci and the corner layered ideals, given by $Z \mapsto
\mcI_{\corn}(Z)$, $Z \subset S$,  and $\mcI \mapsto
\tZ_{\corn}({\tI }).$
\end{prop}
\begin{proof} Clearly $\mcI_{\corn}(Z)$ is closed under addition.  For the second assertion, one
just follows the standard arguments in the Zariski correspondence,
paralleling the argument given in the proof of
Proposition~\ref{Zarcorresp}.
 \end{proof}

\subsubsection{The layered Zariski topology}

\begin{thm}\label{Zar5} The   sets of the form  $\tZ_{\corn}({\tI })$   for $\tI
\triangleleft \mcR $ comprise the closed sets of a coarser
topology than the layered component topology of
Definition~\ref{Zar4}, in which every open set is dense.
 \end{thm}
\begin{proof} If $\mcI $ is generated by a single element, then $\tZ_{\corn}({\tI }) = \tZ_{\corn}(\{
f\}),$ since any corner root of $f$ is a corner root of a multiple
of $f$. Thus, we can reduce to generators of ideals, and $$\bigcap
_j \tZ_{\corn}(\tI _j) = \tZ_{\corn}\bigg( \bigcup _j \tI _j\bigg)
= \tZ_{\corn}\bigg( \sum _j \tI _j\bigg).$$ Moreover,
$$\tZ_{\corn} (\tI _1) \cup Z_{\corn} (\tI _2) = \tZ_{\corn}(\tI _1
\tI _2),$$ in view of Lemma~\ref{geomid31}. Thus these sets
$\tZ_{\corn}({\tI })$ comprise a topology, and all of them are
closed in the topology of Definition~\ref{Zar4}.

In view of Proposition~\ref{Zar3}, the intersection of any two
non-empty open sets contains some non-empty component and thus is
non-empty. Hence every open set is dense.
\end{proof}

 \begin{defn} We call the topology of Theorem~\ref{Zar5} the \textbf{layered Zariski
 topology}. \end{defn}

 Propositions~\ref{Zarcorresp} and \ref{Zarcorresp2}  enable us to transfer some
 tropical geometry to the algebraic theory of
 $\mcR .$ The corner loci are the natural candidates for tropical varieties;
the ones corresponding to prime ideals of $\mcR $ (and thus
 the irreducible layering maps).

 (One should note that this correspondence is somewhat weaker than the usual
 Zariski correspondence in algebraic geometry, since we have not
 determined which ideals of $\mcR $ have the form $\tI _{\corn} (Z)$
 or $\tI _{\vmap}(Z)$.) Indeed, we should consider all of the ideals obtained from the
 $\ell$-varieties, but this is outside of the scope of the present
 paper.

The layered Zariski
 topology plays a key role in studying tropical dimension.

 \begin{rem}\label{translat} The translation to the algebraic structure of $\mcR $ should be helpful in many basic tasks, such as defining dimension.
In ring theory there are three basic definitions of dimension for
affine algebras:
\begin{enumerate} \ealph
    \item The classical Krull dimension (measured by maximal lengths of
prime ideals);
    \item The more general module-theoretic version of Krull
dimension, considered by Gabriel and studied in depth by Gordon
and Robson in \cite{GR};
    \item  The Gelfand-Kirillov
dimension~\cite{KL}.
\end{enumerate}
 These all coincide for commutative affine
algebras, and are defined for semirings as well as rings, so could
be used to define the dimension of the coordinate \semiring0 $\mcR
\cap \Fun(Z,R)$ and thus of the (tropical) layered
variety.\end{rem}

 We   present one interesting example which
 indicates the direction one might take.

 \begin{exampl}\label{zerlay}  Take $R = R(\Net, \Real)$, and consider the polynomials $f_1 = 2\la_1   + 3 \la _2
 + 5$ and $f_2 =3\la_1   +2\la _2
 + 5$. Write $\bfa = (\xl{a_1}{k}, \xl{a_2}{\ell}) \in R^{(2)}$ in logarithmic
 notation. The layering maps
 are
 $$\vmap_{f_1}(\bfa ) = \begin{cases}  k+\ell+1 & \text{ for }
a_1 = 3 ,\ a_2 = 2;\\  k+\ell  & \text{ for }
  a_1 =  1 + a_2 > 3;
 \\  k +1 & \text{ for }
  a_1 =   3 > 1 + a_2 ; \\   \ell+1 & \text{ for }
 a_2 =   2 >  a_1 -1  ; \\   1 & \text{ otherwise. }
\end {cases}$$
$$\vmap_{f_2}(\bfa ) = \begin{cases}   k+\ell+1  & \text{ for }
a_1 =2 ,\ a_2 = 3;\\  k+\ell   & \text{ for }
  a_2 =  1 + a_1 > 3;
 \\    \ell+1  & \text{ for }
  a_2 =   3 > 1 + a_1 ; \\   k +1 &  \text{ for }
 a_1 =   2 >  a_2 -1  ; \\   1 & \text{ otherwise. }
\end {cases}$$
Thus, $$\vmap_{f_1,f_2}(\bfa) = \begin{cases}   \min \{ k, \ell\}
+1 & \text{ for } a_1 = a_2 = 2;\\
    1 & \text{ otherwise. }
\end {cases}$$ 
When $L = L_{\ge 1}$, this is dominated by $\vmap_{g}$ where $g =
\la _1 + \la _2$, since
$$\vmap_{g}(\bfa) = \begin{cases}  k +\ell & \text{ for } a_1
= a_2 ;\\
    1 & \text{ otherwise. }\end {cases}$$
    But $g$ is not generated by $f_1, f_2$, which would run
    counter to our intuition. This example shows that the layering map
    does not suffice to determine the geometry in the standard supertropical theory,
    which is what led us to a different version of the
    Nullstellensatz in \cite{IzhakianRowen2007SuperTropical}.

    On the other hand, if we permit $L$
    to have a zero element, then taking $\bfa = (\xl{2}{0},
    \xl{2}{0}),$ we have $\vmap_{f_1,f_2}(\bfa) = 1 >
    \vmap_{g}(\bfa).$ Thus, the zero layer could be useful
    in the theory.

     Likewise, we can resolve the difficulty for
    $L = \Q_{>0},$ since then we could take $k = \ell = \frac 13.$

\end{exampl}%


%



\begin{exampl} $ $ Take $R = R(\Net, \Real)$, $\tI  = \langle f_1, f_2\rangle$, where $f_1 =  \la_1  +  1  $ and  $f_2 =  \la_1  +  2
$.
\begin{enumerate} \item For $\tSS = R_1^{(n)},$
$\bfa = (a_1, \dots, a_n)$ we have
$$\vmap_{f_1}(\bfa ) = \begin{cases}  2 & \text{ for } a_1 = 1; \\   1 & \text{ otherwise. }
\end {cases}$$ $$\vmap_{f_2}(\bfa ) = \begin{cases}  2 & \text{ for } a_1 = 2; \\   1 & \text{ otherwise. }
\end {cases}$$ Thus, $\vmap_{\tI }$ is 1 identically on $\tSS$, which
is the same as the  layering map of a constant, and thus yields
the same layered variety (namely all of $\tSS $). \pSkip

\item For $n=1$, $\tSS =R,$ and $L = \Q_{> 0},$ we can recover the
tangible part of components of $\tI $ by means of the layering
map. Indeed, suppose $a \in R_\ell$.
$$\vmap_{f_1}(a ) = \begin{cases}  \ell & \text{ for } a >_{\nu} 1;
\\ \ell +1 & \text{ for } a \nucong 1; \\  1 &  \text{ for } a <_{\nu}
1.
\end {cases}$$
$$\vmap_{f_2}(a ) = \begin{cases}  \ell & \text{ for } a >_{\nu} 2; \\
 \ell +1  & \text{ for } a \nucong 2; \\   1 & \text{ for } a <_{\nu}
 2.
\end {cases}$$
Thus, if $\ell >1,$ then
$$\vmap_{\tI }(a ) =
\begin{cases} \ell & \text{ for } a \ge_\nu 2; \\   1 & \text{ otherwise, }
\end {cases}$$
thereby yielding one component  of $\la +2$ and the closure of the
other component, and for $\ell < 1,$
$$\vmap_{\tI }(a ) =
\begin{cases} 1 & \text{ for } a \le_\nu 1; \\   \ell & \text{ otherwise. }
\end {cases}$$

In particular, the use of these extra layers enables us to
distinguish between $\vmap_{\tI }$ and the layering map of a
constant.
\end{enumerate}
\end{exampl}


Alternatively, using $\ell$-roots, one could define the
$\ell$-\textbf{variety of} ${\tI }$ to be
$$
\begin{array}{lrl}
Z_\ell({\tI })  & :=  & \{ \bfa \in \tSS: f(\bfa)  \text{ is an
$\ell$-root}, \forall f \in \tI \} \\[1mm] &  = &  \{ \bfa \in \tSS:
\vmap_{\tI }(\bfa) \   \text{is an $\ell$-ghost sort}
\}.\end{array}
$$

This notion of variety also takes into account the multiplicity of
the root, by applying $\vmap_f ^{-1}$ to the set of $\ell$-ghost
sorts. Of particular interest is the case $\ell = 1,$ since this
provides the set of elements whose values (under functions
in~${\tI }$) are ghosts.

\section{Polynomials in one indeterminate}\label{polyone}

In this section, we consider the case of polynomials in one
indeterminate, quoting \cite{Erez1} extensively, in order to
understand factorization and multiple corner roots. This is
followed up in the next two sections, where we study resultants
and then derivatives (and anti-derivatives). In order to obtain
decisive results, we assume that $R$ is a uniform $L$-layered
$1$-\semifield0.


 Factorization
  of polynomials behaves much better than in the standard
supertropical theory.

We say that a polynomial $f = \sum_{i=0}^t \xl{\al _ i}{\ell_ i}
\lm ^{i }$ is \textbf{monic} if $\xl{\al _ t}{\ell_ t} \nucong
\rone.$
\begin{rem} When studying the $\ell$-roots of a polynomial $f = \sum_{i=u}^t \xl{\al
_ i}{\ell_ i} \lm ^{i },$  we may divide out by $\la^u$ and may
assume
 that $u =0$ without affecting the roots (other than $\rzero)$.
Thus, we may assume throughout that our polynomials
 are not divisible by $\la.$

Furthermore, since the uniform $L$-layered \domain0 $R$ is a
1-\semifield0, we can always replace $f$ by $\al _ t^{-1}f$  and
thereby assume that $f$ is monic. We often assume that $\ell_t =
1,$ since this does not affect the $\nu$-values of the roots.
\end{rem}

We write each polynomial as a decomposition into a sum of
essential monomials; i.e.,   deleting any monomial would change
the polynomial as a function, and call this the \textbf{essential
form} of $f$.
 In other words, if
 \begin{equation}\label{esscon0} f = \sum_j \a_j \la
^{i_j},\end{equation} with each $\a_j \ne \rzero$, then for each
$u$ there is some $a$ such that  $\a _u \la^{i_u}$ dominates $f$
at $a$. This means that the graph of $f$ is concave up. Let us
call $i_{j+1}$ the \textbf{essential exponent following} $i_j$.

\begin{rem} \label{monicval} Suppose $a \in R_\ell.$ If $f   = \la ^t + \sum_{i=0}^{t-1} \xl{\al _ i}{\ell_ i} \lm
^{i }$ is in essential form, then $f(a) = a^t$ when $a^\nu$ is
``large enough,'' so in this case $s(f(a)) = \ell ^t.$ Similarly,
when $\xl{\al _ 0}{\ell_ 0}\ne \rzero,$  $f(a) = \xl{\al _
0}{\ell_ 0}$ when $a^\nu$ is ``small enough,'' so in this case
$s(f(a)) = \ell_ 0.$
\end{rem}

\subsection{Homogeneous parts of a polynomial}

Let us give a more explicit version of these considerations.

\begin{lem}\label{uselem} If $f = \sum_i \a_i \la^i$ is in  essential
form,  then
 $\a_i\a_j >_{\nu} \a_{i-1}\a_{j+1}$
for all $i>j$.
\end{lem}
\begin{proof} It is well-known that the slope  $\frac{{\a_i}}{\a_{i-1}}$ must be
$\nu$-greater than the slope  $\frac{\a_{j+1}}{\a_{j}}$ in order
for the monomials $\a_j \la^j$ and $\a_i \la^i$ to be
essential.\end{proof}

For any tangible $a$,  we call those monomials in $\csupp _a (f)$
  the \textbf{dominant monomials} of $f$ at $\bfa$.
%

\begin{rem}\label{essconv1} Any   polynomial  $f = \sum_j  \a_{i_j} \la^{i_j}$ in essential
form satisfies the following convexity condition:

\begin{equation}\label{esscon}  \left( \frac {\a_{i_{j+1}}}{\a_{i_j}}\right)^{i-i_{j}}
\ge _{\nu}  \left( \frac {\a_{i}}{\a_{i_{j}}}\right)
^{i_{j+1}-i_{j}}. \end{equation} for each $i_j \le i \le i_{j+1}.$

Now, for any $i$ between $i_j$ and $i_{j+1},$ taking $$d_j =
i_{j+1}-i_{j},$$ we formally replace $\a_i$ by
$\xl{\a_{i_j}\left({ \frac
{\a_{i_{j+1}}}{\a_{i_j}}}\right)^{(i-i_j)/d_j}}{0}.$ (This is
 the $\nu$-largest coefficient that can be attached to $\la^i$ without affecting
the $\nu$-values of $f$ as a function, since it is in the
0-layer.) We call this new polynomial the \textbf{full form} of
$f$. Thus, the essential form of $f$ is the ``minimal'' polynomial
equal to $f$ as a function, whereas the full form of $f$ is the
``maximal'' polynomial equal to $f$ as a function.

Having adjoined these new 0-layer monomials, we can now write $f =
\sum_{i=0}^t \al_i \lm ^{i }$, define the slopes $$\frak m_i =
\frac {\a _{i+1}}{\a _{i}},$$  and note that $\frak m_i \ge \frak
m_{i-1}$ for each $i$. (These can also be identified with the
slopes for the essential form.) When
$$\frak m_{i'-1} < \frak m_{i'} = \frak m_{i'+1}= \cdots = \frak
m_{i''-1} < \frak m_{i''}$$ for suitable $i'$ and $i'',$ we call $
\sum_{i=i'}^{i''} \al_i \lm ^{i }$ the $\frak
m_{i'}$-\textbf{homogeneous part} of $f$. (Note that $i''$ is the
essential exponent following $i'$.) We write $\intt(\frak
m_{i'},f) := \{ i', \dots, i''\}.$

We also define homogeneous parts for the extreme cases -- When $f$
has degree $t$, the \textbf{top part} of $f$ is that homogeneous
part such that $i'' = t+1$, and the \textbf{bottom part} of $f$ is
the $\frak m_{0}$-homogeneous part, i.e., with $i'= 0$.
\end{rem}

We also have a \semiring0 homomorphism paralleling the
homomorphism $\Psi_1$ of
\cite[Remark~3.3]{IzhakianRowen2008Completion}.

\begin{rem}\label{switch} Suppose that $\tSS$ is a multiplicative group. Define the map $\Fun(\tSS, R) \to\FuncalSR$
given by $f\mapsto \bar f$, where $\bar f(\bfa) = f(\bfa^{-1}).$
This is clearly an isomorphism of order 2, since
$$\overline{f+g}(\bfa) = (f+g)(\bfa^{-1}) =   f (\bfa^{-1})+ g
(\bfa^{-1}) = (\bar f + \bar g)(\bfa),$$ and likewise
$\overline{fg}(\bfa) =  \bar   f(\bfa) \bar   g(\bfa).$

If $f = \sum _i h_i$, then $\bar f = \sum \overline{h_i}.$ But
$h_i$ is essential at $\bfa$ iff $\overline{h_i}$ is essential at
$\bfa^{-1}.$

The map $f\mapsto \bar f$   reverses the order of the corner
roots, and thus switches the top part of~$f$ with the bottom part.
We make use of this duality to shorten some proofs.
\end{rem}

For example, in the uniform case, if $f = \la^4 + \xl{2}{3}\la^3 +
\xl{4}{2}\la^2 + \xl{5}{5}\la +\xl{6}{1},$ then the 2-homogeneous
part of $f$ is $\la^4 + \xl{2}{3}\la^3  + \xl{4}{2}\la^2$ and the
1-homogeneous part is $\xl{4}{2}\la^2 + \xl{5}{5}\la +\xl{6}{1}.$
Intuitively, the slope is constant on the $\frak
m_{i'}$-homogeneous parts of the polynomial $f$. In general,
$\frak m_{i'}$ is the unique corner root of the $\frak
m_{i'}$-homogeneous part of $f$. By definition, $\intt (\frak
m_{i'},f)= \csupp_{\frak m_{i'}}(f).$

\subsection{Separable polynomials}

\begin{defn} A corner root $a$ of a polynomial $f(\la)$ is \textbf{simple} if $f = (\la +a)g$, where $a$ is not a corner root
of $g.$ A polynomial $f$ is \textbf{separable} if $f$ is the
product of a constant together with linear factors having
$\nu$-inequivalent corner roots.
\end{defn}

\begin{rem} A polynomial is  separable  iff each corner root is simple.
\end{rem}

 The following
observation is due to Sheiner~\cite[Lemma~3.10]{Erez1}.

\begin{prop}\label{multr1} If $R$ is an $L$-layered \semifield0, and if a polynomial $f = \sum_{i=0}^t
\al_ i\lm ^{i }$ is in  essential form, where $s(\al_i) = \ell_i,$
then
$$ f = \xl{\al_ t}{\ell_ t} \prod _{i=0}^{t-1} \left( \la + \xl{\bt_i}{k_i}\right),$$
where $k_i =\frac{\ell_ {i}}{\ell_{i+1}};$ and $\bt_i = \frac{\al_
i}{\al_ {i+1}};$ i.e., $f$ is separable.
 \end{prop}
\begin{proof} Dividing out by $\xl{\al_ t}{\ell_ t},$ we may assume that $ \xl{\al_ t}{\ell_ t}  = \xl{\rone}{1}.$ In analogy to
\cite[Lemma~8.28]{IzhakianRowen2007SuperTropical} (seen by
repeated applications of Lemma~\ref{uselem}),
$$f
=
 \left(\lm+ \xl{\al_ {t-1}}{\ell_ {t-1}}\right) \bigg(\la^{t-1} +\sum_{i=0}^{t-2} \frac{\xl{\al_ i}{\ell_ i}}{\xl{\al_ {t-1}}{\ell_ {t-1}}} \lm ^{i
}\bigg),$$ and we continue by induction.
\end{proof}

\subsection{Primary polynomials}\label{prim1}

Unfortunately,  not every polynomial in $R[ \lm ]$ is separable.
Sheiner~\cite[Lemma~3.10]{Erez1}  handles the general situation by
treating uniform layered domains with $\lzero \in L$, but this
theory is considerably more technical, and factorization loses
uniqueness. In order to treat the general situation, one needs a
more technical approach. Our next definition is in opposition to
separability.

\begin{defn}\label{primarydef} A monic  polynomial $f$ of degree $t$ is called
$a$-\textbf{primary} if  $$f = \la^t + \sum_{j=0}^{t-1}  \al_j
  \lm ^{i_j}$$ where
 $ \al_j   \nucong a^{t-i_j}$ for all $i$.
\end{defn}

\begin{lem}\label{aprim} When $R$ is $\nu$-cancellative and  $\nu$-$\Net$-cancellative, the monic
polynomial $f$  is $a$-primary iff every  corner root of $f$ is
$\nu$-equivalent to $a$.\end{lem}\begin{proof} Any corner root $b$
satisfies $$  a^{t-i_j}b^{i_j}  \nucong \al_j
  \lm ^{i_j}b^{i_j}  \nucong \al_{j'}
  b^{i_{j'}} \nucong  a^{t-i_{j'}}  b^{i_{j'}}$$ for some $j,j'$,
  which implies by  $\nu$-cancellation and
  $\nu$-$\Net$-cancellation that $a \nucong b.$ The reverse implication is
  obtained by reversing this argument.
\end{proof}

From this point of view, the  primary polynomials are the ones
with the simplest root locus, namely the $\nu$-equivalence class
of a single point.

\begin{prop}\label{factortan}
 If $f = \sum _{i=1}^m
h_i $ is in essential form, where $\deg h_i = i,$ then
\begin{equation}\label{outf} h_j f = \bigg( \sum_{i=j}^m h_i \bigg) \bigg( \sum_{k=1}^j
h_k \bigg).
\end{equation}
\end{prop}
\begin{proof} For $i>k$, we have $h_ih_k(a) \le _\nu h_{i-1}h_{k+1}(a)$ for
all $a$, with strict inequality unless these monomials are all in
the same homogeneous component. It follows that  each term in the
left side also appears in the right side, and the other terms
$h_ih_k$ on the right side (for $i>k$) are dominated by
$h_{i-1}h_{k+1}$ and, by induction descending to $j$, are
dominated by a term on the left side, with the domination strict
at some step.
\end{proof}

\begin{prop}\label{factortan1}
 If $f = \sum _{i=0}^m
\a_i \la^i$ has bottom part $ \sum _{i=0}^j \a_i \la^i$, which  is
$\a _j$ times some $a$-primary polynomial $f_a$, then $f = \tilde
f f_a,$ where
$$\tilde f = \a_j^{-1}\bigg(\sum _{i=j}^m \a_i \la^i\bigg).$$
\end{prop}
\begin{proof} By definition of bottom part, the slopes change at $\la^j$, so we have the same convexity argument as in Proposition~\ref{factortan}.
\end{proof}

Thus, we can factor polynomials at their bottom part.
 By duality, one could also factor out the top part first. But at
 any rate, iterating this procedure, we can write any monic
 polynomial $f$ as a product of  primary polynomials, i.e.,
 $$f = \prod _a f_a,$$
 where each $f_a$
 is $a$-primary. We call this the \textbf{primary decomposition}
 of the polynomial $f$. This motivates us to study primary
 polynomials. In the customary theory of polynomials over a field,
 the only primary polynomials would be powers of linear
 polynomials. The situation here is considerably more subtle.

\begin{defn}  $\Prim_a \subset R[\lm],$ $a \in R$, denotes the set of
 $a$-primary polynomials.
\end{defn}

\begin{exampl}\label{exaa} If $R$ is any $L$-layered \domain0, and $a\in R,$ we
define
\begin{equation}\label{eq:aa}
\langle a \rangle _\nu :=   \{b \in R: b\nucong \rone \text{ or }
b\nucong a^j \text{ for some } j \in \Net \},
\end{equation}
and in particular, taking $a = \rone$,
$$\langle \rone \rangle _\nu =
\{b \in R: b\nucong \rone \}.$$ These   clearly are $L$-layered
sub-\semirings0 of $R$.
\end{exampl} Although its structure is rather trivial, $\langle a
\rangle _\nu$ plays a key role in the factorization theory.

\begin{prop} $\Prim_a$ is a sub-monoid of $\langle a
\rangle _\nu[\la]$, and is also closed under addition of
polynomials of the same degree.
\end{prop}
\begin{proof} By definition, each coefficient of an $a$-primary
polynomial belongs to $\langle a \rangle _\nu$. Also,
$$\sum_{i=0}^t \xl{\al_i }{k_ i} \lm ^{ i }+
\sum_{i=0}^t \xl{\bt_i }{\ell_ i} \lm ^{ i } = \sum_{i=0}^t
\xl{(\al_i+\bt_i) }{k_ i+\ell_ i} \lm ^{ i },$$ whereas $\xl{\al_i
}{k_ i}+\xl{\bt_i }{\ell_ i} \nucong  a^{t-i}+a^{t-i} \nucong a
^{t-i}.$ Thus it remains to show that the product of $a$-primary
polynomials is $a$-primary.
$$\bigg(\sum_{i=0}^t \xl{\al_i }{k_ i} \lm ^{i
}\bigg)\bigg( \sum_{j=0}^{t'} \xl{\bt_j }{\ell_ j} \lm ^{j
}\bigg)$$ is a sum of monomials $$\xl{\al_i }{k_ i} \lm ^{i }
\xl{\bt_j }{\ell_ j} \lm ^{j} = \xl{(\al_i \bt_j)}{k_ i\ell_ j}
\lm ^{i+j},$$ and $\xl{\al_i }{k_ i}\! \xl{\bt_i }{\ell_ i}
\nucong a^{t-i} a^{t'-j} = a^{t+t'-i-j},$ so the product is indeed
$a$-primary.
\end{proof}

\begin{prop}\label{primval} If $f
= \sum_{i=0}^m \xl{\al_i }{\ell_ i} \lm ^{i }$ is $a$-primary,
then for $b \in R_k,$ $$f(b) = \begin{cases} \xl{b^m}{k^m} &
\quad\text{if}\quad b>_{\nu} a;\\   \xl{a^m}{\sum \ell_ i k^{i} }
&  \quad\text{if}\quad b \nucong a;\\ \xl{\al_0 }{\ell_ 0} &
\quad\text{if}\quad b<_{\nu} a.\end{cases}$$\end{prop}
\begin{proof} Clearly the monomial $\lm ^{m }$ strictly
dominates $f$ at $b$ when $ b>_{\nu} a$, and $\xl{\al_0 } {\ell
_0}$ strictly dominates when $ b<_{\nu} a$, whereas when $b=a$ all
the terms have the same $\nu$-value, and one just combines the
ghost layers.
\end{proof}

\begin{cor}\label{comp1} Suppose $f
= \sum_{i=0}^t \xl{\al_i }{\ell_ i} \lm ^{i }$ is $a$-primary, and
$b \in R_k$. Then $$s(f(b)) = \begin{cases} k ^t &
\quad\text{if}\quad b >_{\nu} a;\\ \sum _{i=0}^t   \ell_i k^{i} &
\quad\text{if}\quad b  \nucong a;\\ \ell_0 & \quad\text{if}\quad b
<_{\nu} a.
\end{cases}$$\end{cor}

We can generalize Proposition~\ref{primval}.

\begin{prop}  Notation as above, in any  homogeneous part  $\sum _{i=i'}^{i''} \xl{\al_{i} }{\ell_i} \lm ^{i }$
of a polynomial $f$, we have $f(b) = \xl{\al_{i''} }{ {k^{i''}
\ell_{i''}} }b^{i'' }$ whenever $\frak m_{i'} <_\nu b <_\nu \frak
m_{i''},$ where $i''$ is the essential exponent following $i'$ and
$\ell = s(b)$.
\end{prop}
\begin{proof} This is the strictly dominant term in the $\frak m_{i'}$-homogeneous part,
which we claim dominates all other terms. Indeed, $$\xl{\al_{i }
}{ {k^{i } \ell_{i }} }b^{i  } \nucong \a_{i''} a^{i''-i}b^i =
\a_{i''} \frak m_{i'}^{i''-i}b^i,$$ which is greatest when $i =
i''.$
\end{proof}

For the bottom part, $f(b) = \xl{\al_{0}}{\ell_0}$ when $b <_\nu
\frak m_0.$
\begin{prop}\label{primmult0} For any
$a$-primary polynomial $f$ in $R[\lm]$, we have $f(b) \nucong
(a+b)^{\deg f}$ for all $b \in R.$
\end{prop}
\begin{proof} One checks each of the three cases in Proposition~\ref{primval}.\end{proof}
\begin{cor}\label{primmult} For any  product $f = \prod f_a$ of $a$-primary polynomials $f_a$ and any $b$-primary polynomial $g_b$,
we have $$\prod _a f_a(b) ^{\deg g_b} \nucong \prod_a (a+b)^{\deg
f_a \deg g_b} \nucong \prod _a g_b(a) ^{\deg f_a}.$$
\end{cor}
\begin{proof} Apply Proposition~\ref{primmult0} twice, for $f_a$ and for $g_b$.\end{proof}
\begin{cor}\label{primmult2} For any  products   $f = \prod _a f_a$
and $g = \prod _b g_b$  of $a$-primary polynomials   and
$b$-primary polynomials respectively,  we have $$\prod _{a,b}
f_a(b) ^{\deg g_b} \nucong \prod_{a,b} (a+b)^{\deg f_a \deg g_b}
\nucong \prod _{a,b} g_b(a) ^{\deg f_a}.$$
\end{cor}

\begin{rem}\label{sortpsi} Recall from  \cite[Lemma~3.13]{Erez1} that there
 is a monoid homomorphism $\psi_a : \Prim_a \to L[\lm]$ given by
\begin{equation}\psi_a \bigg(\sum_{i=0}^t \xl{\al_i }{\ell_ i} \lm ^{i
}\bigg)= \sum_{i = 0}^t \ell _i \lm ^{i },
\end{equation}
which also is additive on $a$-primary polynomials of the same
degree.  (This is an easy consequence of Axioms~A3 and~B.
Furthermore, $\psi_a$ is an isomorphism in the important case that
$R =\R(L,\tG)$ of Construction~\ref{defn5}. The 0-layer terms
obtained in multiplying together $a$-primary polynomials are
inessential, and thus can be excluded. For example,  $$ (\la +
\xl{ a }{-\ell})(\la + \xl{ a }{\ell})= \la^2 + \xl{ a}{0}\la+
\xl{ a^2 }{-\ell^2}= \la^2 + \xl{ a^2 }{-\ell^2}$$ since the
monomial $\xl{ a}{0}\la$ is inessential. (Otherwise, the
homomorphism would break down, which happens in other situations,
as pointed out by Sheiner \cite{Erez1}.)
\end{rem}

Sheiner~\cite{Erez1} obtained the following uniqueness result:

\begin{thm}\label{Erez1}\cite[Lemma~3.10, Theorem~3.12]{Erez1} Any polynomial $f\in
R[\la]$ over a uniform $L$-layered \semifield0 $R$ can be factored
in the form
\begin{equation}\label{fact1} f = \a f_{a_1}\cdots f_{a_d}\end{equation}
where $\a \in R,$
 $a_1 > _\nu a_2 >  _\nu \cdots > _\nu a_d$ are the corner roots
of $f$, and each $f_{a_j}$ is $a_j$-primary. This factorization is
unique with respect to the $a_j^\nu$.\end{thm}

Existence is given in Proposition~\ref{factortan}. Uniqueness is
obtained by noting that for any $b \in R$,
$$s(f(b)) = s(\a) s(f_{a_1}(b))\cdots s(f_{a_d}(b)),$$ so any
other factorization
 must have the same corner roots, in view of
 Remark~\ref{monicval} and Corollary~\ref{comp1}, which show that
 different corner roots in the factorizations would give different
sorts. One concludes by observing that any two
 distinct $a$-primary polynomials are different as functions in view of Remark~\ref{sortpsi} (as is seen by substituting $ \xl{a
}{\ell}$ for different values of $\ell$.)

This factorization into primary polynomials is sufficient for many
applications, such as:

\begin{cor}\label{fact2} Suppose $ f = \a f_{a_1}\cdots f_{a_d}$ as in
\eqref{fact1}, where $a_1 > _\nu a_2 >  _\nu \cdots > _\nu a_d,$
and suppose $b \in R_\ell$.   Also, write
$$f_{a_j} = \lm ^{t_j}+\sum_{i=1}^{t_j} \xl{\al_{i,j} }{\ell_
{i,j}} \lm ^{t_j-i }$$ for   $1 \le j \le d.$  If $b\nucong a_j $
for some $j$, then
$$s(f(b)) = s(\a) \bigg(   \ell^ {t_ j} +  \sum _{i=1}^{t_j}  \ell_
{i,j} \ell^{{t_j}-i}\bigg) \prod _{i=j+1}^d \ell_{t_i,i}\prod
_{i=1}^{j-1} \ell_{0,i}.$$

If $a_j <_\nu  b < _\nu a_{j+1},$ then
$$s(f(b)) =  \prod _{i=j+1}^d \ell_{t_i,i}\prod _{i=1}^{j} \ell_{0,i}$$
\end{cor}

This enables us to compute in some sense how much the ghost layer
is raised by evaluating at the corner root $a$.
 Unfortunately, a certain ambiguity remains --
sometimes an $a$-primary polynomial could be factored into
$a$-primary polynomials of smaller degree, and this need not be
unique, as seen for example in
 \cite[Corollary~3.14]{Erez1}. Ironically, this failure can be viewed in a
 positive light, by a connection to ``classical'' algebra.

 \begin{rem}\label{polyinfo} Let $\mcL$ denote
 the ``classical'' polynomial \semiring0 $L[\la]$. In view of
 Remark~\ref{sortpsi}, any factorization of an $a$-primary
 polynomial can be transferred to a classical factorization in $\mcL$. Thus, we obtain information about factorization of $a$-primary
polynomials in terms of factorizations in $\mcL$. But when we take
$L$ to be positive,   the classical factorization is modified
somewhat, which leads to various difficulties. For instance,
taking $L = \Q _{>0}$ leads us to classical factorization of
polynomials into polynomials having {\it positive} coefficients,
which is not necessarily unique although factorization of
polynomials over $\Q$ is unique.

These considerations lead one towards considering a larger sorting
\semiring0\ $L$ which is a field, and in particular would have
negative elements, so that we could factor polynomials into linear
binomials. But once~$0$ is adjoined to $L$, the hypothesis of
Theorem~\ref{Erez1} is no longer valid.  As mentioned in
Remark~\ref{sortpsi},  \cite{Erez1} explains how the $0$ layer can
ruin factorization into primary polynomials. Thus, we must tread a
narrow path. First we take the  primary decomposition, and then
study each
 $a$-primary polynomial in turn by means of the map $\psi_a$.
 \end{rem}

\subsection{Multiple roots}

The layered theory enables us to study multiplicities of corner
roots.
%

 We
have rather precise information, but at the cost of taking $L =
\Q,$ i.e., considering negative ghost layers.

\begin{prop}\label{sort3}  Suppose $f$ is $a$-primary of degree $t$, with $s(a) = \ell.$ Write $f_\psi$ for $\psi_a(f)$
from Remark~\ref{sortpsi}.
\begin{enumerate} \eroman
    \item $f(a) = \xl{a ^t}{f_\psi(\ell)};$ in particular,
 $s(f(a))   = f_\psi(\ell).$ \pSkip

\item  $(\la + a)$ divides $f$ iff $f_\psi(-\ell)= 0,$ iff $f(
\xl{a }{-\ell}) \in R_0$. \pSkip

\item   $(\la + a)^m$ divides $f$, iff $-\ell$ is a root of
 $f_\psi$ of multiplicity at least $ m$. \pSkip
\end{enumerate}

\end{prop}

\begin{proof} (i) Write $f
= \sum_{i=0}^m \xl{\al_i }{\ell_ i} \lm ^{i }$. Since $f(a)$ is a
sum of terms each $\nu$-equivalent to $a^t,$ $f = \sum_{i=0}^m
\xl{\al_i }{\ell_ i} \lm ^{i }$, we have
 $s(f(a)) = \sum
\ell_ i \ell ^{i} = f_\psi(\ell).$ \pSkip

(ii) Write $f_\psi  = (\la + \ell)g_\psi + \ell ',$ according to
the classical Euclidean algorithm, and take $g$ such that $g_\psi
= \psi_a(g)$. Then $\ell' = 0$ iff $f_\psi (-\ell) = 0,$ iff $f(
\xl{a }{-\ell}) \in R_0$ in view of (i). Note that if $f_\psi  =
(\la + \ell)g_\psi$, then $f = (\la + a) g $ in view of
Remark~\ref{sortpsi}. \pSkip

(iii) Apply induction to (ii).
\end{proof}


We may like to express the multiplicity of a corner root $a$ (cf.~
Definition~\ref{rootlev10}) directly in terms of the ghost layer
$s(f(a))$, especially in the case where $f$ is $a$-primary. The
first guess might be that $a\in R_\ell$ has multiplicity $\ge m$
 if $s(f(a))\ge 2^m s(f(b))$,
  for suitable $b \in R_\ell$ which is not an $\ell$-root of $f$. There
 are several difficulties
  with this approach: The constant term could have a very large ghost layer in comparison
  with the intermediate terms, which
  distorts the factorization. Also, non-roots could
 yield different ghost layers. Nevertheless, here are some
 examples to aid intuition.

\begin{exampl}\label{multroots}
Suppose $f \in \mcR$, and $a\in R_1$.

\begin{enumerate} \eroman
\item  If $f(a) \in R_1,$ then $s(f(a)) = 1 = 2^0.$   \pSkip

    \item The tangible corner root $ a$ of $f= \lm + a$ satisfies $f(a) =
    a + a  \in R_2,$ so  $s(f(a)) = 2 = 2^1.$\pSkip

    \item The tangible corner root $ a$ of $f=(\lm + a)^m$ satisfies
    $f(a) = (a + a)^m \in R_{2^m},$ so  $s(f(a)) =  2^m.$ \pSkip

\item The tangible corner root $ a$ of $f=\lm^2 + a ^2$ satisfies
    $f(a) = a^2 + a^2 \in R_{2},$  so  $s(f(a)) = 2 = 2^1.$ \pSkip

\item The tangible corner root $ a$ of $f=\lm^2 + a \lm  + a ^2$
satisfies
    $f(a) = a^2 + a^2+ a^2\in R_{3},$ so $s(f(a)) = 3.$ \pSkip

\item The tangible corner root $ a$ of $f=\lm^2 + \xl{a}{2}\lm  +
a ^2$ satisfies $s(f(a)) = 4.$ \pSkip

\item The tangible  corner root  $a \in R_\ell$ of  $f =
(\xl{\la}\ell +a)^m$
 satisfies $$s(f(a))
 = \sum \binom{m}{j} \ell^j \ell^{m-j}= \ell^{m} \sum \binom{m}{j} = 2^m\ell^m,$$
 whereas for $b \not \nucong a, $ $s(f(b)) = \ell^ m.$

\end{enumerate}

\end{exampl}

\section{Layered resultants}

Resultants are an attractive tool since they provide a link
between linear algebra and geometry, and also provide a criterion
for when polynomials are relatively prime. The supertropical
resultant was studied in the standard supertropical case in
\cite{IzhakianRowen2008Resultants}, and the same definition works
more generally in the layered theory.

\subsection{The layered resultant}

 We assume throughout   this section
that $f,g \in R[\lambda]$ have respective degrees $m,n$ over the
$L$-layered \domain0 $R$ and, for convenience, we write $f =
\sum_{i=0}^{m} \xl{\al_i }{\ell_ i} \lm ^{m-i }$ in full form,
where the inessential coefficients have ghost layer $ 0$, as in
Remark~\ref{essconv1}. We also adjoin $ \xl{\rzero}{\ell} $
formally to $R$ for each layer $\ell$, in order to be able to deal
more easily with matrices. (For example, the off-diagonal entries
of the identity matrix are $\rzero : =  \xl{\rzero}{0}$.)

The \textbf{layered permanent} $\Det{A}$ of an $n \times n$ matrix
$A = (a_{ij})$, with $a_{ij} \in R$, is defined to be
\begin{equation}\label{eq:per} \Det{A}:= \sum_{\sig \in S_n}
a_{1, \sig(1)} \cdots a_{n,\sig(n)};
\end{equation}

Surprisingly, the theory of the layered permanent parallels the
classical theory of the    determinant, as is seen
in~\cite{IzhakianRowen2008Matrices}, which is the reason that we
use the same notation as is customarily used for the determinant.

\begin{defn} For any semiring $R$, suppose $f =
\sum _{i=0}^m \al  _i \lm ^i \in R[\lm]$, and let $A_n(f) $ denote
the $n\times (m+n)$ matrix
$$\left(
\begin{array}{ccccccccccc}
 \al _0 & \al _1 & \al _2 & \al _3 & \dots & \al _m &    &
  & &   \dots &
\\
   & \al _0 & \al _1 & \al _2 & \dots & \al _{m-1} & \al _m &     &   &    &    \\
   &    &  \al _0 & \al _1 & \dots & \al _{m-2} & \al _{m-1} & \al _m &  &    &     \\
 \vdots &  &   & \ddots & \ddots & \ddots & \ddots
 & \ddots &   &  & \vdots \\
   &    &   &  & \al _0 & \dots  & \dots  & \dots  & \dots   & \al _m  &     \\
   &   \dots  &   &  &  & \al _0  & \al _1  & \dots  & \dots   & \al _{m-1} & \al _m \end{array} \right),$$
where the empty places stand for  $\rzero$.

For a polynomial $g = \sum_{j=0}^n \bt_j \lm^j$ of degree $n\ge
1$, the \textbf{Sylvester matrix} (also called the
\textbf{resultant matrix}) $\res(f,g)$ is the $(m+n) \times (m+n)$
square matrix
$$ \res(f,g) = \left(
\begin{array}{c}
   A_n(f) \\
  A_m(g) \\
\end{array}
   \right).$$ We define the \textbf{(layered) resultant} of $f$ and
$g$ to be the layered permanent $\Det{ \res (f,g) }$ when $m,n \ge
1,$ and to be $\al_0^n$ when $f$ is a constant (i.e., $m=0$), and
analogously to be $\bt_0^m$ when $g$ is a constant.
\end{defn}

There is a fine point which we must address: Conceivably, one
could alter the resultant when replacing a polynomial by its full
form. By the end of this discussion, we will see that this process
does not change the resultant, but we bear this difficulty in
mind.

\begin{rem}\label{permcom0} Multiplying $f$ through by $\al_m$ multiplies the resultant by $\al_m^n$ which does not affect any of our assertions. Thus, for
convenience, we will always assume that $f$ (and also $g$) are
monic.\end{rem}

\begin{rem}\label{permcom} When $g = \sum _{j=0}^n  \beta_j
 \la^{j},$ the resultant $\Det{\res(f,g)}$ is easily seen to be the sum of
terms of the form
 $$\a_{i_1}\cdots \a_{i_n} \bt_{j_1}\cdots  \bt_{j_m}$$
where $(i_1+1, i_2+2, \dots, i_n +n,  j_1+1,\dots,j_m+m)$ is a
permutation of $(1, \dots, m+n)$. In particular, $i_{u + t} \ne
i_u +t$, $j_{u + t} \ne j_u +t$, and $i_{u + t} \ne j_u +t$ for
all $u$ and $t$.

It follows that
\begin{equation} \label{ind1}
\begin{array}{ll}
   i_1 + \dots + i_n +j_1 + \dots + j_m  & = \\[1mm]
    (1+\cdots + m+n) - (1+\cdots + m) - (1+\cdots + n) & =\\[1mm] \binom
{m+n+1}2 - \binom {m+1}2 - \binom {n+1}2 & = mn.  \end{array}
\end{equation}
\end{rem}

To simplify notation, in the following examples of matrices, we
write $\one$   for $\rone$; empty spaces stand for $\rzero$.

\begin{exampl}\label{res1} Throughout this example, we take
 $\al  := \xl{ \al }{k   } \in R_{k}$, $\al_i := \xl{ \al }{k_i  } \in R_{k_i}$,
 $\bt := \xl{ \bt}{\ell  }$, and $\bt_j := \xl{ b }{\ell_j } \in
R_{\ell_j}$. We specify the layer only when it differs from this
notation.

\begin{enumerate}\eroman \item For $f = \la +\al ,$ $g = \la +\bt,$
$$\Det{\res(f,g)}= \al + \bt  =
\begin{cases}
\al  &  \text{for} \quad  \al >_\nu \bt,\\
\xl{ \al }{k+\ell }  & \text{for} \quad  \al \nucong \bt,\\
\bt &  \text{for} \quad  \al <_\nu \bt.
\end{cases}$$

\item For $f = \sum _{i=0}^m \al_i \la^{i}$   and  $g=
\la + \bt$,
$$
\begin{array}{lll}
|\res(f,g)| & =  & \left|\begin{array}{ccccccccccc}
\al _0  &  \al _1  & \dots  &  \al _{m-1} & \al _m\\
 \bt & \one  &  & & \\
   & \bt & \one  &  &  \vdots \\
\vdots &  &   \ddots  &  \ddots & \\
 & \cdots &  & \bt & \one   \end{array}\right| =  \al_0 + \al_1 \bt + \cdots +  \al_m \bt^m = f(\bt)
\end{array}
 $$
which is $$\begin{cases}
\al_0  & \text{for}\quad  \al _1 \bt <_\nu \al _0,
\text{ (in other words, $\bt$ is smaller than
all roots of $f$),}\\
\xl{\a_i \bt^i}{\sum_{i \in \intt (\bt,f)}k_ i \ell^{i}}
& \text{for}\quad \bt \ \text{a root of } f,\\
\xl{\a_i \bt^{i}}{k_ i  \ell^{i}} & \text{for}\quad \al _i
\bt>_\nu \al _{i-1} \text{(in other words, $\bt$ is between roots
of $f$)}.
\end{cases}$$ In every case this equals $f(\bt).$
Likewise, $$|\res(f,\la)|= \left|\begin{array}{ccccccccccc} \al _0
&  \al _1  & \dots  &  \al _{m-1} & \al _m\\
 & \one  &  & & \\   & & \one  &  &  \vdots  \\
\vdots  &  &     &  \ddots & \\
 &  \cdots &  &  & \one
  \end{array}\right|= \al _0.$$

\item Suppose $f = (\la +\a_1) (\la +\a_2) = \la^2+(\a_1+\a_2)\la
+ \a_1 \a_2 $ and $ g = \la + \bt$, with $\a_1 <_\nu \a_2$. Then
$$|\res(f,g)|= \left|\begin{array}{ccccccccccc}
 \a _1 \a_2  &  \a _1+  \a_2 & \one
\\
  \bt & \one   &  \\
& \bt & \one \end{array}\right| =  \a_1 \a_2 + ( \a_1 + \a_2 )\bt
+ \bt^2 = f(\bt),$$ which is
$$\begin{cases} \a_1 \a_2  & \text{if}\quad \bt < _\nu \a_1,\\
\a_2 \xl{\a_1}{k_ 1  + \ell}   & \text{if}\quad \bt \nucong \a_1,\\
  \bt \a_2& \text{if}\quad \al_1 <_\nu \bt  <_\nu \a_2,\\
\xl{\a_2^2}{\ell(k_2  + \ell)}   & \text{if}\quad \bt \nucong \a_2,\\
 \bt^2 & \text{if}\quad \al_2 < _\nu \bt.\end{cases} $$

\item When $f $ is $a$-primary of degree 2, say $f = \la^2 + \a_1
\la +  \a_0 ,$ $g = \sum _{j=0}^n  \beta_j \la^{j},$

$$|\res(f,g)|= \left|\begin{matrix}
  \a_0  & \a_1  & \one  &  & \cdots &
\\
   &  \a_0  &\ \a_1 & \one    &  & \vdots  \\
\vdots & &   &  \ddots &  \\
& &  & \a_0  & \a_1  & \one \\
    \beta_0
  &  \beta_1    & \dots  &  \beta_{n-1}
  &  \beta_n
  \\
    &       \beta_0
 &  \beta_1   & \dots  &  \beta_{n-1}
  &  \beta_n
  \end{matrix}\right| ,$$
which, in view of Remark~\ref{permcom}, is the sum of terms of the
form $$\a_{i_1}\cdots \a_{i_n} \bt_{j_1}  \bt_{j_2},$$ where
$(i_1+1, i_2+2, \dots, i_n +n,  j_1+1,j_2+2)$ is a permutation of
$(1, \dots, n+2)$. Equation \eqref{ind1} implies $$i_1 + \dots +
i_n = 2n -j_1 -j_2.$$ Thus each term involving $\bt_{j_1}
\bt_{j_2}$ is $\nu$-equivalent to
\begin{equation}\label{termnu}
a^t \bt_{j_1}  \bt_{j_2}\end{equation} where $t = \sum_{u=1}^n
(n-i_u) = n^2 -2n + j_1 +j_2.$

Since these $\nu$-values are all the same, one sees that the layer
of the sum of terms involving
 $\bt_{j_1}  \bt_{j_2}$ is precisely the permanent of the layer matrix of $\res(f,g),$ with the $n+1,n+2$ rows and
$j_1, j_2$ columns erased.  But \eqref{termnu} equals
\begin{equation}\label{discdom}
 a^{n^2-2n}  {\bt_{j_1}}{a^{j_1}} {\bt_{j_2}}{a^{j_2}}.\end{equation}
 If $a$ is not a corner root of $g$, then $g(a) =
 {\bt_{j}}{a^{j}}$ for some $j$ (which strictly dominates the other terms
 of this form), and picking $j_1 = j_2 = j$ yields $$|\res (f,g)| \nucong a^{n^2-2n} g(a)^2 . $$
If $a$ is a corner root of $g$, then there are $j_1 \ne j_2$ for
which ${\bt_{j_1}}{a^{j_1}} = {\bt_{j_2}}{a^{j_2}}$ dominates
\eqref{discdom}, so we have at least two summands of the form
\eqref{termnu}.

In particular, $|\res (f,g)| \nucong a^{n^2 -2n} \bt_{0}^2$    if
the
 smallest root of $g$  dominates $a$. If
the
 smallest root of $g$ strictly dominates $a$, then we get a unique dominant term from $\bt_0^2,$
  and $|\res (f,g)| =  a^{n^2 -2n} \bt_0^2.  $ Likewise, if $a$  strictly dominates all the
  roots
of $g$, then    we get the dominant terms in $|\res (f,g)|$ by
choosing $\bt_n$ along the lower part of the main diagonal, which
means the remaining part in computing the permanent must be
$\al_0$ along the upper part of the main diagonal, yielding $|\res
(f,g)| = \al_0^n \bt_n^2. $ \pSkip

\item More generally, when $g$ is $b$-primary and $b$ strictly
 dominates every root of $f$, then the same argument as in
(iv) shows that $|\res(f,g)|= \al_m^{n} \bt_0^m.$  Namely, we get
the most significant terms when we choose ~$\bt_0$ and~$\a_m$ in
computing the permanent. \pSkip

\item When $g$ is $b$-primary and $b$  dominates every
  root of $f$, then the same argument as in
(v) shows that $|\res(f,g)|\nucong \al_m^{n} \bt_0^m,$ but we
could have other terms yielding the same result. Note however that
any term contributing to $|\res(f,g)|$ must be products of
coefficients of the upper part of $f$ together with coefficients
of $g$. Thus, we would have the same resultant if we replaced $f$
by its upper part. \pSkip
\end{enumerate}
\end{exampl}

\begin{lem}\label{extraex} When $f$ and $g$ are both $a$-primary, i.e, $f = \sum_{i = 0}^m \xl{
a^i }{\ell _i} \lm^{m-i}$ and $g = \sum_{j=0}^{n} \xl{ a^j}{k_j}
\lm^{n-j} $, then $|\res{(f,g)}| = \xl{ a^{mn} }{\ell '} ,$ where
$\ell'$ is the permanent of the matrix of the ghost layers.
\end{lem} \begin{proof} Every possible term in the
permanent has the same $\nu$-value as $a^{mn}$, and so we add the
ghost layers of these terms. \end{proof}

This example turns out to be so instrumental that we introduce
some notation.

\begin{defn} The \textbf{layer matrix} of a matrix $A = (\xl{a_{i,j} }{\ell_ {i,j}})$
is the matrix $ (\ell_{i,j}) \in M_n(L).$ When $L$ is a ring, its
(classical) determinant is computed as an element of $L$.

 Given an a-primary polynomial $f =  \sum _{i=0}^m  \xl{ a^{i} }{k _i}\lm^{m-i} \in
R[\lm]$, we let   $\les_n(f) $ denote the {layer matrix} of~
$A_n(f)$, which is the $n\times (m+n)$ matrix
$$\left(
\begin{array}{ccccccccccc}
k_0 &k_1 &k_2 &k_3 & \dots &k_m &    &
  & &   \dots &
\\
   &k_0 &k_1 &k_2 & \dots &k_{m-1} &k_m &     &   &    &    \\
   &    & k_0 &k_1 & \dots &k_{m-2} &k_{m-1} &k_m &  &    &     \\
 \vdots &  &   & \ddots & \ddots & \ddots & \ddots
 & \ddots &   &  & \vdots \\
   &    &   &  &k_0 & \dots  & \dots  & \dots  & \dots   &k_m  &     \\
   &   \dots  &   &  &  &k_0  &k_1  & \dots  & \dots   &k_{m-1} &k_m \end{array} \right),$$
where the empty places stand for $0$.

Let us write $\Per{A}$ for the permanent of the matrix $A \in
M_n(L)$. For  polynomials $f,g$ of respective degrees~$m,n$, the
\textbf{layer Sylvester matrix}
  $\les(f,g)$ of $f$ and $g$ is the  matrix $  \left(
\begin{array}{c}
   \les_n(f) \\
  \les_m(g) \\
\end{array}
   \right),$ and
the \textbf{layer permanent}
 of $f$ and $g$ is
  $\Per{\les(f,g)}$.
 \end{defn}

We now can restate Lemma~\ref{extraex} more succinctly.
 \begin{lem}\label{primcomp} When $f$ and $g$ are $a$-primary,
 $$\Det{ \res (f,g) }=   \xl{ a }{\Per{\les (f,g)}}^{mn}.$$\end{lem}

\begin{exampl}\label{multresult01} Suppose $f =   \sum_{i=0}^m \xl{ a^{m-i} }{k _i} \la ^i
$  and $g = \la +  \xl{ a }{\ell }$. The permanent of the layer
 Sylvester matrix  is
$$\Per{ \les(f,g)} = \Per{\begin{array}{ccccccccccc}
 k_0 & k_ 1  & \dots & k_m
\\
  \ell & 1  &  &  \\
   &    &  \ddots &    \\ &
 &  \ell  & 1    \end{array}} = k_0 + k_1 \ell + k_2\ell^2 +\cdots = \sum_{i=0}^m k_{i}\ell^i,$$
 seen by expanding along the first row, and this equals $\tilde
 f(\ell),$ where $\tilde f = \sum_{i=0}^m k_{i}\la^i.$ In
 particular, if  $\ell \ge 1$ and $ k_0, k_1, k_2 \ge 1$, then $\Per{ \les(f,g)}$ is
 2-ghost.
\end{exampl}

The  computations of Example~\ref{res1} might lead us to expect
that the resultant is multiplicative, especially in view of
\cite[Theorem~4.12]{IzhakianRowen2008Resultants}. However,
Example~\ref{extraex} already leads us to a counterexample. To
simplify notation, we write
 $\xl{ a^i }{k }$ to denote
  $\xl{ (a^i)}{k}$,
 i.e., $a^i$ given layer $k$.

\begin{exampl}\label{multresult02} Suppose $f =    \la^{2}
  + \xl{ a }{k _1} \la + \xl{ a^2 }{k _0} ,$ $g = \la
+  \xl{ a }{\ell }$, and $h = \la + \xl{ a }{\hat \ell}.$ (Thus
$f,g,h$ are all $a$-primary.) Let $\Per{\les (f,g)}$ and
$\Per{\les (f,h)}$  be the permanents of the layer  Sylvester
matrices, which are
$$\Per{\les (f,g)} = \Per{\begin{array}{ccccccccccc}
 k_0 & k_ 1  & 1
\\
  \ell & 1  &    \\
 &  \ell  & 1    \end{array}} = \ell^2 + k_1\ell + k_0,$$
$$\Per{\les (f,h)} =   \Per{\begin{array}{ccccccccccc}
 k_0 & k_ 1  & 1
\\
  \hat \ell & 1  &    \\
 &    \hat \ell  & 1    \end{array}}  =   \hat \ell^2 + k_1  \hat \ell + k_0,$$
  Then $$\Det{ \res (f,g) }=   \xl{ a
}{\Per{\les (f,g)} }, \qquad \Det{ \res (f,h) }=   \xl{ a
}{\Per{\les (f,h)}} ,$$ so their product is $$  \xl{ a }{\Per{\les
(f,g)} \Per{\les (f,h)}}= \xl{ a }{p}$$ where
$$
\begin{array}{lll}
p & =  &(\ell^2 + k_1\ell + k_0)(\hell^2 + k_1\hell+
k_0) \\[2mm]
& = & \ell^2(\hell^2 + k_1 \hell +
k_0)+\ell(k_1\hell^2+k_1^2\hell+k_0k_1) + k_0\hell^2 +
k_0k_1\hell+ k_0^2. \end{array}
$$
On the other hand,
$$\Per{\les (f,gh)}  =  \Per{\begin{array}{ccccccccccc}
 k_0 & k_ 1  & 1 &
\\   &k_0 & k_ 1  & 1
\\
  \ell \hat \ell   & \ell +\hat  \ell  & 1  &    \\
 &  \ell \hat \ell   & \ell +\hat  \ell  & 1
       \end{array}} ,$$
       which clearly is a sum of 13 terms, and turns out to be
       $ p + 4 k_0 \ell \hell.$
       (In the determinant computation, these would appear twice
       with $+$ sign and twice with $-$ sign, and thus cancel.)
       It follows that  when $k_0 \ell \hell \ne 0$ we have
$$  |\res
   (f,gh)| \ne |\res
   (f,g)||\res
   (f,h)|.$$

   It is worth comparing these computations with the fact that in
   the ``classical'' world, for the standard  determinant, $\Det{\res(f,gh)}= \Det{\res(f,g)}\Det{\res(f,h)}.$
   Since the terms involved in computing the permanent and the
   determinant are the same (just with a change of sign), one
   might be surprised that we had this new term~$k_0\ell\hat
   \ell.$ This is clarified when we factor $f$ into linear
   factors, i.e., $f = (\la + a_1)(\la + a_2).$ Then $k_0 = s(a_1
   a_2)$ and $k_1 = s(a_1 + a_2),$ so the ``extra'' term
   $s(a_1a_2) \ell\hell$ now is achieved in
   $$\Per{\les (f,g)} \Per{\les (f,h)}=  (\ell^2 + s(a_1 + a_2)\ell + s(a_1
   a_2))(\hell^2 +
s(a_1 + a_2)\hell+ s(a_1
   a_2)),$$ although with a smaller coefficient.
\end{exampl}

\begin{exampl}\label{multresult03} For $f =    \la^{2}
  + \xl{ a^2 }{k _1}\la + \xl{ a^3 }{k _0} ,$ $g = \la ^2
+  \xl{ a }{\ell_1 } \la  +  \xl{ a^2 }{\ell_0 }$, and $h = \la^2
+  \xl{ a }{\hat \ell_1}\la + \xl{ a^2 }{\hat \ell_0},$
$$\Per{\les (f,g)}  = \Per{\begin{array}{ccccccccccc}
 k_0 & k_ 1  & 1 &
\\  & k_0 & k_ 1  & 1
\\
  \ell_0 &   \ell_1 & 1  &    \\
 & \ell_0 &   \ell_1 & 1    \end{array}}\qquad \text{and} \qquad
 \Per{\les (f,h)} =   \Per{\begin{array}{ccccccccccc}
  k_0 & k_ 1  & 1 &
\\  & k_0 & k_ 1  & 1
\\
  \hat \ell_0 &  \hat \ell_1 & 1  &     \\
 &    \hat \ell_0 &  \hat \ell_1 & 1     \end{array}},$$
and we have checked on the computer, using Matematika, that every
term in their product is subsumed in~$\Per{\les (f,gh)}$.
\end{exampl}

\begin{exampl}\label{multresult04} For $f =    \la^{3}  + \xl{ a }{k _2}\la^2
  + \xl{ a^2 }{k _1}\la + \xl{ a^3 }{k _0} ,$ and $g$ and $h$ as in Example~\ref{multresult03}, again we
have checked on the computer that every term in the product is
subsumed in $\Per{\les (f,gh)}$.
\end{exampl}

After the initial shock of these examples, one can find the
following consolations: \begin{enumerate}   \item  $$|\res
   (f,gh)| \nucong |\res
   (f,g)|  |\res
   (f,h)|,$$

  \item $$ |\res (fg,h)| \nucong |\res
   (f,h)|  |\res
   (g,h)|,$$

\item  $$\bigg |\res
   \bigg(\prod _a f_a, \prod _b g_b \bigg) \bigg|=   \prod _{a,b} |\res (f_a,
   g_b)|,
 $$\noindent the products taken over the homogeneous parts of $f$ and $g.$

 \item
   $$  |\res
   (f,gh)| \lmodWL | \res
   (f,g)||\res
   (f,h)|$$ when $L$ is small enough (including the standard supertropical case).
\end{enumerate}

   We aim for these results.
 We need a method of factoring out the bottom part of a polynomial. Note that in
the computations in Example~\ref{res1}, we could disregard all the
homogeneous parts of $g$ except the one in which $a$ is a root.

Our main objective is to compare $| \res(f,g)|$ with $ \prod
_{a,b} |\res(f_a,g_b)|,$  where $f = \prod _a f_a$ is the primary
decomposition of $f$ and $g = \prod _b g_b$ is the primary
decomposition of $g$.

Already we have the following special case, which was indicated in
Example~\ref{res1}:

\begin{lem}\label{firstcase1} Suppose $f = \sum _{i=0}^m \al_i  \la^{i}$ is $a$-primary and  $g =
\sum _{j=0}^n \bt_j  \la^{j}$ is $b$-primary. Then $$|\res(f,g)|
\nucong g(a)^m \nucong (a+b)^{mn}\nucong f(b)^n.$$

If $a < _\nu b,$ then $$|\res(f,g)|= \bt _0 ^m = g(a)^m \nucong
b^{mn} = (a+b)^{mn}= f(b)^n.$$ If $a >_\nu b,$ then $$|\res(f,g)|=
\al _0 ^n = f(b)^n \nucong a^{mn} =(a+b)^{mn}= g(a)^m .$$
 \end{lem}

\begin{proof}By
Remark~\ref{permcom}, $|\res(f, g)|$ is a sum of terms of the form
 $$\a_{i_1}\cdots \a_{i_n} \bt_{j_1}\cdots  \bt_{j_m}.$$ When $b >_\nu a$, noting
 that $$\bt_j a^j <_{\nu} \bt_0 \nucong b^n,$$ we get
 the $\nu$-dominant term when we increase the
weight of $b$ in this term by choosing each $\bt_{j_u}$ to be $
\bt_0$, i.e., we choose the term in $|\res(f, g)|$ involving $\bt
_0 ^m$. There is only one such nonzero term, and this has each $\a
_i = \a_m.$ But $\a_m \nucong \rone$ since $f$ is monic. Thus, the
$\nu$-dominant term of $|\res(f, g)|$ is $\bt _0 ^m = b^{mn}$,
since all the other terms $a^i b^{mn-i}$ are dominated by it. The
second assertion follows by symmetry, since then $g(a) \nucong a^n
= (a+b)^n.$ When $a \nucong b,$ we still have the same
$\nu$-dominant terms, but perhaps have others as well, so we only
get $\nu$-equality.
\end{proof}

 To generalize this observation, we turn to an idea from
\cite{IzhakianRowen2008Resultants}.

\begin{defn} Given a polynomial $f = \sum _{i=0}^m \a_i \la ^i$, we define
$$\spol{f}{u} = \sum _{i=u}^m \a_i \la ^{i-u}, \qquad u =
1,\dots, m.$$  $\spol{f}{1}$ is   called the \textbf{reduction
along} $\bar a$, where $\bar a$ is the corner root $\frac
{\a_0}{\a_1}.$
\end{defn}

\begin{lem} Over a layered 1-\semifield0, suppose $\bar a$ is the root of $f$ having lowest $\nu$-value,
and $\deg f_{\bar a} = u$.  Then $f = \spol{f}{u} f_{\bar a}.$
\end{lem}
\begin{proof} This is just a restatement of
Proposition~\ref{factortan1}, since  $\spol{f}{u} = \tilde f$ and
$ f_{\bar a} =  \sum _{i=0}^u \a_i \la ^{i}.$
\end{proof}

\begin{exampl}\label{spolcomp}  $f = \la \spol{f}{1} + \a_0.$ For $\al_1$ invertible,  $f \nucong (\la + \frac
{\al_0}{\al_1})\spol{f}{1}.$  \end{exampl}

 Our main tool   is the following computation:

\begin{lem}\label{prodres2} If $f =  \sum _{i=0}^m \a_i \la ^i$ and $g =  \sum _{j=0}^n \bt_j \la ^j$, then
\begin{equation}\label{sumres1} |\res(f,g)| \nucong \al_0 |\res(f,\spol{g}{1})| + \bt_0 |\res(\spol{f}{1},g)|.\end{equation}
\end{lem}
\begin{proof} We expand the resultant $$ |\res(f,g)| = \left|
\begin{array}{ccccccc}
             \al_0  & \al_1 & \al_2 & \dots &  &  &  \\
            & \al_0  & \al_1 & \al_2 & \dots &   &   \\
              & & \al_0  & \al_1 & \al_2 & \dots   &   \\
                &   &  &  \ddots  &\dots &\ddots &  \vdots \\
                     \bt_0 & \bt_1 & \bt_2 & \dots & \bt_n &  &  \\
                   &  \bt_0 & \bt_1 & \bt_2 & \dots & \bt_n &     \\
                    &   &  \bt_0 & \bt_1 & \bt_2 & \dots  &    \\
                       &&      &\ddots &\dots & &\vdots
                  \end{array} \right|$$ along the first column, to get
\begin{equation}\label{sumdets} \al_0 \left|  \begin{array}{ccccccc}
            \al_0  & \al_1 & \al_2 & \dots &  &  & \\
            & \al_0  & \al_1 & \al_2 & \dots &   \\
              &  &   \ddots  &\dots &\ddots & \dots \\
                    \bt_1 & \bt_2 & \dots & \bt_n &  &  \\
                    \bt_0 & \bt_1 & \bt_2 & \dots & \bt_n  &     \\
                      & \ddots & \dots    &\ddots & & \ddots
                  \end{array} \right| + \bt_0 \left|  \begin{array}{cccccc}
            \al_1 & \al_2 & \dots &  &  &  \\
              \al_0  & \al_1 & \al_2 & \dots &    &   \\
              & \ddots &  \ddots  &\dots &\dots &  \dots\\
                     \bt_0 & \bt_1   & \dots & \bt_n &   & \\
                   &  \bt_0 & \bt_1 & \dots & \bt_n &    \\
                      &   &  \ddots  &\dots &\ddots &  \ddots
                  \end{array} \right|
                  \end{equation}
                  In computing the second \permanent\   of
Equation~\eqref{sumdets} by expanding along the first column,  the
occurrence of $\a_0$ in the second row must be multiplied by some
$\a_i$ in the first row, whereas, switching the first two rows, we
also have $\a_1 \a_{i-1}$. But   $\a_1 \a_{i-1} \ge _\nu \a_0
\a_{i}$, so the term with $\a_0 \a_i$ is not relevant to the
computation of the \permanent. Thus the occurrence of $\a_0$ in
the second row cannot strictly dominate  the second \permanent\
of~\eqref{sumdets}, and we may erase it.

 By the same token, each occurrence of $\bt_0$ does not strictly dominate   the first \permanent\
 of~
Equation~\eqref{sumdets}. Thus, \eqref{sumdets} is
$\nu$-equivalent to
\begin{equation*}\label{sumdets1}
\begin{array}{l}
 \al_0 \left|  \begin{array}{cccccc}
         \al_0 &   \al_1  & \al_2 &  \dots &  &  \\
        &  \al_0    & \al_1  &  & \dots &   \\
               &     &\ddots &\dots & &\\
                    \bt_1 & \bt_2 & \dots & \bt_n &  &    \\
                      & \bt_1 & \bt_2 & \dots & \bt_n  &     \\
                      &      &\ddots &\dots & &
                  \end{array} \right| + \bt_0 \left|  \begin{array}{cccccc}
            \al_1 & \al_2 &\dots  &  \\
            & \al_1 & \al_2 & \dots    \\
               &     &\ddots &\dots  &\\
                    \bt_0 & \bt_1 & \dots & \bt_n &  &    \\
                      & \bt_0 & \bt_1 & \dots &      \\
                      &      &\ddots &\dots &
                  \end{array} \right|
 \\
 \\[2mm]
                   = \al_0 |\res(f,\spol{g}{1})| + \bt_0
                  |\res(\spol{f}{1},g)|.
                  \end{array}
        \end{equation*}
\end{proof}


We want strict equality in \eqref{sumres1} Towards this end, we
have.

\begin{cor}\label{prodres0} If $f =  \sum _{i=0}^m \a_i \la ^i$ and
$g =  \sum _{j=0}^n \bt_j \la ^j$ satisfy $\al_0
|\res(f,\spol{g}{1})| >_\nu \bt_0 |\res(\spol{f}{1},g)|$, then
$$|\res(f,g)|\nucong \a_0 |\res(f,\spol{g}{1})|,$$ and
 \begin{equation}\label{taneq16} |\res(f,g)| = \al_0 \left|  \begin{array}{ccccccc}
            \al_0  & \al_1 & \al_2 & \dots &  &  & \\
            & \al_0  & \al_1 & \al_2 & \dots &   \\
              &  &   \ddots  &\dots &\ddots & \dots \\
                    \bt_1 & \bt_2 & \dots & \bt_n &  &  \\
                    \bt_0 & \bt_1 & \bt_2 & \dots & \bt_n  &     \\
                      & \ddots & \dots    &\ddots & & \ddots
                  \end{array} \right|.\end{equation}
\end{cor}

\begin{cor}\label{lamb0} If $f = \sum _{i=0}^m \a_i \la ^i$, then $$|\res(f,\la g)|= \a_0 |\res(f,g)|.$$
 \end{cor}
 \begin{proof} By the proposition applied to $\la g,$ noting that $g = \spol{(\la g)}{1}.$
 \end{proof}

\begin{cor}\label{lamb} If $f = \sum _{i=0}^m \a_i \la ^i$, then $$|\res(f,\la^t g)|= \a_0 ^t |\res(f,g)|,$$
and likewise if $g = \sum   _{i=0}^n \bt_i \la ^i$, then
$$|\res(\la^t f, g)|= \beta_0 ^t |\res(f,g)|.$$

 \end{cor}
  \begin{proof} By induction on $t,$ applying Corollary~\ref{lamb0}
  repeatedly.
 \end{proof}

\begin{thm}\label{multresult11} Any monic polynomials $f$ and $g$   satisfy:
 \begin{equation}\label{taneq13}| \res(f,g)| \nucong \prod
_{a,b} |\res(f_a,g_b)|\nucong \prod _{a,b} (a+b)^{m_a
n_b},\end{equation} where $f = \prod _a f_a$ is the primary
decomposition of $f$ and $g = \prod _b g_b$ is the primary
decomposition of $g$, and $m_a = \deg f_a$ and $n_b = \deg g_b$.

Furthermore, write $\bar a$ (resp.~$\bar b$) for the root of $f$
(resp.~$g$) having smallest $\nu$-value. Then in \eqref{sumres1},
 \begin{equation}\label{taneq131} |\res(f,g)| = \al_0 |\res(f,\spol{g}{1})|\end{equation} if $\bar a
<_{\nu} \bar b,$   \begin{equation}\label{taneq1311} |\res(f,g)| =
\bt_0 |\res(\spol{f}{1}),g)|\end{equation} if $\bar a
>_{\nu} \bar b,$ and
\begin{equation}\label{taneq132}|\res(f,g)|= \al_0
|\res(f,\spol{g}{1})| ^{\nu}=\bt_0
|\res(\spol{f}{1},g)|^{\nu}\end{equation} if  $\bar a \nucong \bar
b.$
\end{thm}

\begin{proof} We prove Theorem~\ref{multresult11} by
double induction on $\deg f$ and $\deg g,$ and carry this
inductive assumption for Equation~\eqref{taneq13} throughout. The
base of the induction is Lemma~\ref{firstcase1}. Let $$P := \prod
_{a,b} |\res(f_a,g_b)|.$$ We   first prove the inequality
\begin{equation}\label{taneq132}| \res(f,g)| \le_{\nu} P.\end{equation}
 The reverse inequality to
\eqref{taneq132}  will will follow by considering the leading
monomial that is $\nu$-equivalent to~$P$.

 Since we could
only increase the layered resultant by passing to the full forms
of~$f$ and $g$, we may replace~$f$ and~$g$ by their full forms. On
the other hand, since we are only interested in $\nu$-values at
this stage, we may replace all the coefficients by tangible
coefficients of the same $\nu$-value. Thus, we assume that every
power of~$\la$ up to $\deg f$ has a tangible coefficient in $f$,
and likewise for $g$.

 Assume for convenience that $\bar a
\le _\nu \bar b$. Thus,
 $$\frac {\beta_0}{\beta_1} \nucong \bar b \ \ge_\nu \ \bar a =  \frac {\al
 _0}{\al_1}.$$
 Appealing to \eqref{sumres1}, we want
to show that $$\text{$\al_0 |\res(f,\spol{g}{1})| \le P$ \ds {and}
$\bt_0 |\res(\spol{f}{1},g)|\le P$.}$$ Note that $\spol{g}{1} =
\spol{g_{\bar b}}{1}\prod _{b\ne {\bar b}} g_b ,$ so by induction
$$ | \res(f,\spol{g}{1})| \nucong \prod
_{a}\bigg( |\res(f_a,\spol{g_{\bar b}}{1})| \prod_{b\ne {\bar
b}}|\res(f_a,g_b)|\bigg).$$ But writing $\bar m$ for $m_{\bar a}$
and $\bar n$ for $n_{\bar b}$, and $f_{\bar a} = \sum _{i=0}^{\bar
m} \a_{i,\bar a} \la^i,$ we have $\a_{\bar m, \bar a} \nucong
\rone,$ and thus
 $$|\res(f_a,\spol{g_{\bar b}}{1})|  \nucong  \beta_1^{\bar m} =
 \beta_0^{\bar m} \left(\frac{\beta _1}{\beta
_0}\right)^{\bar m}$$ by Lemma~\ref{firstcase1}, and
$$ \a_{0,\bar a}\nucong \bar a^{\bar m} \nucong \left(\frac {\a _0}{\a _1}\right)^{\bar m},$$
yielding
$$ \a_{0,\bar a}\res(f_a,\spol{g_{\bar b}}{1}) \le _\nu  \beta_0^{\bar m} \left(\frac {\a
_0}{\a _1}\frac{\beta _1}{\beta _0}\right)^{\bar m} \le _\nu
\beta_0^{\bar m} ,$$  with strict inequality when $\bar a < _\nu
\bar b.$ Hence, we have \begin{equation}\label{ind11}\al_0
|\res(f,\spol{g}{1})|\le  |\res(f_{ a},g_{\bar b})|\prod _{b \ne
\bar b} |\res(f_{ a},g_{  b})| = P,
\end{equation}  with strict inequality when $\bar a < _\nu \bar  b.$

Likewise, by induction,

 $$ | \res( \spol{f}{1},g)| \nucong \prod
_{b}\bigg( |\res(\spol{f_{\bar a}}{1} ,g_b)| \prod_{a\ne {\bar
a}}|\res(f_a,g_b)|\bigg).$$ But, by induction,
$|\res(\spol{f_{\bar a}}{1} ,g_b)| \nucong \a_m^{n} \beta_0^{m-1}
\nucong  \beta_0^{m-1}$  since $\a _m \nucong  \rone,$ so
$$\bt _0 \res(f_{\bar a},g_b) \le _\nu  \beta_0^m \left(\frac {\a
_0}{\a _1}\frac{\beta _1}{\beta _0}\right)^m \le _\nu  \beta_0^m
,$$ implying $\bt_0 |\res(\spol{f}{1},g)| \le P.$

 Thus,
 we have proved inductively that $\al_0 |\res(f,\spol{g}{1})| \le_\nu P$ with strict $\nu$-inequality if
 $\bar a<_\nu \bar b,$ and with equality if
 $\bar a \nucong \bar b,$  and also that $\bt_0 |\res(\spol{f}{1},g)| \le_\nu P$. This establishes Equation~\eqref{taneq132}.
 On the other hand, as noted in Lemma~\ref{firstcase1}, it is easy to find the dominant term $\al_{m,a}^n \bt_{0,b}^m$
in $\res(f_{a},g_b) $, so $\res(f_{a},g_b) = \al_{m,a}^n
\bt_{0,b}^m$, and this term also occurs when we use the essential
forms of $f$ and $g$. Equation ~\eqref{taneq132}
 follows.
\end{proof}

\begin{cor}\label{prodres} If $f =  \sum _{i=0}^m \a_i \la ^i$ and
$g =  \sum _{j=0}^n \bt_j \la ^j$ have the property that  $\bar a
\le _\nu \bar b$, notation as in the theorem, then
$$|\res(f,g)|\nucong \bt_0 |\res(\spol{f}{1},g)|;$$ if $\bar a
<_\nu \bar b$,  then
$$|\res(f,g)|= \bt_0 |\res(\spol{f}{1},g)|.$$
\end{cor}
 \begin{proof} Apply  Equation \eqref{ind11} to Corollary~\ref{prodres0}.
 \end{proof}

\begin{cor}\label{prodres1} If $f =  \sum _{i=0}^m \a_i \la ^i$ and
$g =  \sum _{j=0}^n \bt_j \la ^j$ have the property that $\bar a
<_\nu \bar b$, notation as in the theorem, then
$$|\res(f,g)|= \bt_0^{\bar m} |\res(\spol{f}{1},g)| = g(\bar a)^m |\res(\spol{f}{u},g)|= |\res(f_{\bar a},g)||\res(\spol{f}{u},g)| .$$
\end{cor}
 \begin{proof} Iterate Corollary~\ref{prodres}.
 \end{proof}

\begin{cor} $|\res(f,\prod _b g_b)| \nucong \prod_b f(b)^{\deg g_b},$
for any product of $b$-primary polynomials $g_b.$
\end{cor}
 \begin{proof} Apply  Theorem~\ref{multresult11} to the primary
 decompositions.
 \end{proof}

\begin{cor}\label{primfac}  If $f = \prod _a f_a$ and $g = \prod _b g_b$
written as products of primary polynomials, then $$|\res(f, g)|
\nucong \prod_{a,b} |\res(\la +a, \la+b)|^{\deg f_a \deg g_b}.$$
\end{cor}
  \begin{proof} Replace $f$ and $g$ by their full forms.
  \end{proof}


\begin{thm}\label{multresult5} Any polynomials $f,g,$ and
$h$ satisfy the relations: \begin{equation}\label{taneq03}|
\res(f,gh)| \nucong |\res(f,g)| |\res(f,h)|  \quad \text{and }
\quad |\res(fg,h)|\nucong |\res(f,h)| |\res(g,h)| .\end{equation}
 \end{thm}
\begin{proof} By symmetry, we need only prove the first equivalence.
Taking a   decomposition as in Remark~\ref{essconv1}, we write $f
= f_a    \tilde f,$ $g = g_b  \tilde g,$ and $h = h_c   \tilde h,$
and we may assume that $b \ge _{\nu} c.$ We induct on the degree
of $f$.

 In
view of Theorem~\ref{multresult11}, factoring $f$ and $g h$ into
their primary factors via \cite[Theorem~3.12]{Erez1} we may assume
that $f$   and $gh$ are primary. But then we conclude by means of
Lemma~\ref{firstcase1}.
\end{proof}

Note that Theorem~\ref{multresult5} was proved independently of
\cite{IzhakianRowen2008Resultants}, so the layered structure
actually gives a stronger result than in
\cite{IzhakianRowen2008Resultants}, and with a more direct proof.
We can improve Theorem~\ref{multresult11} in certain cases.

\begin{prop}\label{multresult41} For  monic polynomials $f$ and $g$,
assume that the corner roots of
 $f$ are distinct from the corner roots of $g$. Then:
\begin{equation}\label{taneq133}| \res(f,g)| = \prod _{a,b}
|\res(f_a,g_b)|,\end{equation}  where $f = \prod _a f_a$ is the
primary decomposition of $f$ and $g = \prod _b g_b$ is the primary
decomposition of $g$.
\end{prop}\begin{proof} Let $m_a = \deg f_a$ and $n_b = \deg g_b$,
and assume that $\bar a$ (resp.~$\bar b$) is the $\nu$-lowest of
all the corner roots of the $f_a$ (resp.~ of the~$g_b$). To
illustrate the idea, we first assume that $\bar a <_\nu \bar b$.
We apply Corollary~\eqref{prodres} repeatedly, alternating from
$f$ to $g$ when necessary, but at each time having a unique root
of lowest $\nu$-value to remove.
\end{proof}

 This also
gives us   multiplicativity results.

\begin{cor}\label{multip}   In case $f$, $g$ and $h$ have  no
 corner  roots in common,
\begin{equation}\label{taneq02}| \res(f,gh)| = |\res(f,g)|
|\res(f,h)|  \quad \text{and } \quad |\res(fg,h)| = |\res(f,h)|
|\res(g,h)| .\end{equation}
\end{cor}
\begin{proof} Write $f = \prod_a f_a$ where the $f_a$ are $a$-primary, and
likewise $g= \prod_b g_b$  and $h= \prod_c h_c$. Then, in view of
Proposition~\ref{multresult41}, $$\left| \res(f,gh)\right| =
\bigg| \res \bigg(\prod_a
 f_a ,\prod_{b,c} g_b h_c \bigg)\bigg| = \prod_{a,b,c} \left| \res(f_a,g_b)\right|\left| \res(f_a,
 h_c)\right|=  \left|\res(f,g)\right|
\left|\res(f,h)\right|.$$ The second assertion is proved
analogously.
 \end{proof}

\begin{prop} \label{firstcase}
  Suppose  $f = \sum _{i=0}^m \al_i\la^{i}$ and $g =
\sum _{j=0}^n \bt_j \la^{j},$ where  $f$ is $a$-primary. If $a$ is
strictly dominated by all the corner roots of $g$, then
$$|\res(f,g)|= \bt _0 ^m.$$  If $a$  strictly dominates  all the  corner
roots of $g$, then
$$|\res(f,g)|= \al _0 ^n.$$   \end{prop}
\begin{proof} By
Proposition~\ref{multresult41}. Letting $\beta _{0,b}$ denote the
constant term of $g_b,$ we have $\beta _0 = \prod _b \beta
_{0,b}$, and for the first assertion Lemma~\ref{firstcase1} then
yields
$$|\res(f,g)|= \prod _b \beta _{0,b}^m = \bt _0 ^m. $$
The second assertion is proved analogously. \end{proof}

\begin{lem} Suppose $\bar a$ is smaller than any corner root of
$f$ or $g$. Then $$\big|\res((\la +  \xl{\bar a}{ k} ) f, (\la +
 \xl{\bar a}{ \ell}) g)\big|= \xl{\bar a}{k +\ell} \a_0 \bt_0
 |\res(f,g)|,$$ where $\a_0, \bt_0$ are the respective constant
 terms of $f$ and $g$.
  \end{lem}
  \begin{proof} Note that the constant terms of $(\la + \xl{\bar a}{k}  ) f$ and $(\la +
 \xl{\bar a}{ \ell})) g$ are respectively $\xl{\bar a}{k}\a_0$ and $ \xl{\bar a}{ \ell} \bt_0$.  Thus, applying \eqref{taneq132} on the left and then
  \eqref{taneq131} on the right, taking Example~\ref{spolcomp} into account, we see that the two remaining
  dominant terms are $\xl{\bar a}{k}\a_0   \bt_0
  |\res(f,g)|$ and $  \xl{\bar a}{ \ell} \a_0 \bt_0
  |\res(f,g)|,$ yielding the assertion.
   \end{proof}

The dual assertion holds when we take the resultant of two
polynomials whose leading components are linear with the same
$\nu$-values.

\begin{exampl}\label{multresult031} In logarithmic notation, we compute $|\res(\la^2 + 5\la + 7, \la^2 + 4\la + 6)|$ to
be
$$ \left|\begin{array}{ccccccccccc}
7 & 5  & 0 &
\\  &7 & 5  & 0
\\
  6 &   4 & 0  &    \\
 & 6 &   4 & 0    \end{array}\right|,$$ which is

$$
7\left|\begin{array}{ccccccccccc} 7 & 5  & 0
\\
 4 & 0  &    \\
 & 4 & 0     \end{array} \right| + 6\left|\begin{array}{ccccccccccc}
5& 0  & \\ & 5 &  0  \\
6 &  4  & 0    \end{array} \right| = 7(5 \cdot 4 \cdot 0) + 6(5
\cdot 5 \cdot 0) = \xl{16 }{2}.$$

These polynomials factor to $(\la +5)(\la+2)$ and $(\la
+4)(\la+2)$, so the lemma yields $$\xl{2 }{2} \cdot 5 \cdot 4
\cdot 5 = \xl{16 }{2}.$$
\end{exampl}

\begin{prop} \label{firstcase2}
  Suppose any common  corner root of   $f $ and $gh$ is a simple root for each of $f$ and $gh$. (In particular, this is the
  case if $f$ and $gh$ are separable.) Then
$$|\res(f,gh)|= |\res(f,g)||\res(f,h)|.$$   \end{prop}
\begin{proof} Factoring to primary polynomials, we may assume that
$f$ is $a$-primary and $gh$ is $b$-primary. 
 Write $f = \sum _{i=0}^m \al_i
\la^{i}$, $g = \sum _{i=0}^n \bt_i \la^{i},$  and $h = \sum
_{i=0}^t \gamma_i \la^{i}.$ \pSkip

If $a <_\nu b, $ then $ |\res(f,gh)| =  (\bt _0 \gamma_0) ^m = \bt
_0 ^m \gamma_0 ^m =  |\res(f,g)||\res(f,h)|. $ \pSkip

If $a >_\nu b, $ then $ |\res(f,gh)| =  \al_0^{n+t} = \al_0^n
\al_0^t = |\res(f,g)||\res(f,h)|. $

Thus we may assume that $a \nucong b,$ so $b$ is a simple root of
$gh$, implying by hypothesis that $g$ or $h$ is constant, and
again we are done.
\end{proof}


Our next obstacle is when $f$ and $g$ have $\nu$-equivalent
corner roots of lowest $\nu$-value.

\begin{lem}\label{multresult412} For  monic polynomials $f = \sum _{i=0}^m \al_i
\la^{i}$ and  $g = \sum _{i=0}^n \bt_i \la^{i},$  having the same
root $\bar a$ of lowest $\nu$-value, we have
\begin{equation}\label{taneq134}| \res(f,g)| = | \res(f_{\bar a},g_{\bar
a})|\a_0^{\bar n} \bt_0^{\bar m} |\res(\spol{f}{m} ,\spol{g}{n}
|).\end{equation} where $\bar m = \deg f_a$ and $\bar n = \deg
g_a$.
\end{lem}
 \begin{proof} By Corollary~\ref{prodres1}, we get $\nucong$. We need to show that all the  dominant terms of $|
 \res(f,g)|$ come from  $| \res(f_{\bar a},g_{\bar
a})|\a_0^{\bar n} \bt_0^{\bar m} |\res(\spol{f}{m} ,\spol{g}{n}
|$. This was done in the simpler case where the roots of lowest
$\nu$-value differ, in Corollary~\ref{prodres1}, and we want to
 adapt the proof of Corollary~\ref{prodres1} to get equality.
Although the spirit is exactly the same, the notation seems to be
 a bit more cumbersome, since Corollary~\ref{prodres0} shows that
 we get more terms in the case at hand, which are not as easy to
 eliminate by the trick of Lemma~\ref{prodres2}.

The straightforward approach would be to apply the expansion along
 $\bar m+\bar n$ columns simultaneously, rather than along the first column.

 Namely, writing the Sylvester matrix $ \res(f,g)$ as $(c_{i,j}),$ any term
 in the summation of $| \res(f,g)|$ has the
  form $$s_\pi = \prod c_{1,j_1}
c_{2,j_2}\cdots c_{ {m+n},j_{m+n}},$$ where $j_i = \pi(i).$ We say
that $c_{i,j_i}$ has \textbf{type} 1 if $i \le \bar m$ and $j_i
\le \bar m$ (in which case $c_{i,j_i }$ is a coefficient of
$f_a$), and $c_{i,j_i }$ has \textbf{type} 2 if $m < i \le m +
\bar n$, (in which case $c_{i,j_i}$ is a coefficient of $g_a$). We
need that all dominant terms for $ \res(f,g)$ have types 1 or 2.

%

One can see this directly, but the notation is cumbersome.
Instead, we take an inductive procedure, expanding
Equation~\ref{taneq16} along the first column, and note that for
any term \begin{equation}\label{taneq1351} c = c_{1,j_1}
c_{2,j_2}\cdots c_{ {m+n},j_{m+n}},\end{equation} with $
c_{n+2,j_1}= \beta_0$ (along the $n+2$ row) we have the
corresponding term  \begin{equation}\label{taneq1352} c_{1,j'_1}
c_{2,j_2}\cdots c_{ {m+n},j'_{m+n}},\end{equation} where $j'_{n+1}
= 1$ so that $c'_{n+1,1} = \beta_1$, and $j'_{n+2} = j_{n+1}$ so
that $c'_{n+2,1} = c_{n+2,j_{n+1}} = \bt _{j_{n+1} -(n+2)}$, and
all the other $j'_i = j_i.$ Then the   \eqref{taneq1352} is the
same as \eqref{taneq1351} except that $\bt _0 \bt_{j_{n+1}
-(n+1)}$ has been replaced by $\bt _1 \bt_{j_{n+1} -(n+2)},$ for
which the $\nu$-value could only increase. If $c$ is a dominant
term for $| \res(f,g)|$, then it has been associated with a
dominant term for $| \res(f,g)|$ coming from $\bt_0
                  |\res(\spol{f}{1},g)|$, and now we can continue the induction
                  to show that we only get dominant terms when we do the reduction
                  along the smallest corner root at each stage, which means the first $\bar m + \bar n$ reductions
                  must be $\bar n$ for $f$ and $\bar m$ for $g$, as desired.
 \end{proof}

 \begin{thm}\label{multresult413} For  monic polynomials $f$ and
 $g$, we have
\begin{equation}\label{taneq135}| \res(f,g)| =\prod _{a,b}
|\res(f_a,g_b)|,\end{equation}  where $f = \prod _a f_a$ is the
primary decomposition of $f$ and $g = \prod _b g_b$ is the primary
decomposition of $g$.\end{thm}
 \begin{proof} Iterate Lemma~\ref{multresult412} and
 Proposition~\ref{firstcase}.
 \end{proof}

In view of Proposition~\ref{firstcase}, we know $|\res(f_a,g_b)|$
except when $a\nucong b.$ Thus, the determination of  $|
\res(f,g)|$ has been reduced to the case where $f$ and $g$ are
both $a$-primary. This case has already been considered in
Lemma~\ref{primcomp}, in which it was seen that $| \res(f,g)| =
   \xl{ a }{\Per{\les (f,g)}}^{mn},$ which however is tricky to
   compute, as observed in Example~\ref{multresult02} and the ensuing discussion.

   \begin{lem}\label{multresult05} Suppose $f =   \sum_{i=0}^m \xl{ a^{m-i} }{k _i} \la ^i
$  and $g = (\la +  \xl{ a }{\ell' })h$, where $h = \la^n +
\sum_{i=0}^{n-1} \xl{ a^{n-i} }{\ell _i} \la ^i $, with each
$\ell_i \ge 0$. The permanent of the layer
 Sylvester matrix satisfies $\Per{ \les(f,g)} \ge \tilde
 f(\ell')\Per{ \les(f,h)} ,$ where $\tilde f = \sum_{i=0}^m k_{i}\la^i\in L[\lambda].$
\end{lem}
\begin{proof} Write  $\tilde h = \la^n + \sum_{i=1}^{n-1} \ell_{i}\la^i.$ Then $(\la +   \ell' )\tilde h = \la ^{n+1} + \sum _{j=1}^{n-1} (\ell' \ell_{j} + \ell _{j-1})\la^j
+\ell_0,$ so
$$\Per{ \les(f,g)} = \Per{\begin{array}{ccccccccccc}
 k_0 & k_ 1  & \dots & k_m &
\\ & k_0 & k_ 1  & \dots & k_m
\\
   &    &  \ddots &    &\\
  \ell' \ell_0 &  \ell'\ell_1+ \ell_0  & \dots &  &\\ & \ell' \ell_0 &  \ell'\ell_1+ \ell_0  &  \dots  &\\
   &    &  \ddots &    & \end{array}} = k_0 + k_1 \ell' + k_2{\ell'}^2 +\cdots .$$
Expanding the permanent  along the first row shows that we get all
the terms of $\sum_{i=0}^m k_{i}{\ell'}^i\Per{ \les(f,h)}$ (and
possibly others).
\end{proof}

\begin{prop}\label{permcon1} Suppose $L = L_{\ge 0},$ and $f =   \sum_{i=0}^m \xl{ a^{m-i} }{k _i} \la ^i
$  and  $g = \la^n + \sum_{i=0}^{n-1} \xl{ a^{n-i} }{\ell _i} \la
^i $.  Write  $\tilde g = \la^n + \sum_{j=1}^{n-1} \ell_{j}\la^j,$
which as a polynomial over $\mathbb C$ we factor as $$\tilde g =
\prod _{j=1}^n (\la - u_j).$$ Then the permanent of the layer
 Sylvester matrix satisfies $\Per{ \les(f,g)}  \ge
 \prod _{j=1}^{n-1} \tilde f(u_j),$ where $\tilde f = \sum_{i=0}^m k_{i}\la^i.$
\end{prop}
\begin{proof} Iterate Lemma~\ref{multresult05}.
\end{proof}

The following question still remains.

\begin{conjecture} \begin{equation}\label{taneq01L}| \res(f,gh)| \lmodWL
|\res(f,g)| |\res(f,h)|  \quad \text{and } \quad |\res(fg,h)|
\lmodWL  |\res(f,h)|  |\res(g,h)|.\end{equation} \end{conjecture}

Note that by Theorem~\ref{multresult413} it is enough to show this
in the case where $f,g,h$ are all $a$-primary for some~$a$, in
which case the conjecture reduces to proving that the
\textbf{layer Sylvester matrices} satisfy
  $\les(f,gh) \ge \les(f,g)\les(f,h)$.
 But the permanent is congruent to the determinant modulo 2. Thus,
 replacing the layers by indeterminates, we might hope
 that formally no term in $\les(f,g)\les(f,h)$ repeats. This is
 easy to see for $g$ or $h$ linear, but in general we have the following
 example.

 \begin{exampl}\label{permcon3} Suppose $f =  \xl{\a_2 }{k_ 2}  \la^2 + \xl{\a_1 }{k_ 1}
 \la + \xl{\a_0}{k_ 0} $ and $g=  \xl{\bt_2 }{\ell_ 2}  \la^2 + \xl{\bt_1 }{\ell_ 1}
 \la + \xl{\bt_0}{\ell_ 0} $. Then $\Per{\les (f,g)}$ is the permanent
 of the matrix
$$\Per{\begin{array}{ccccccccccc}
 k_0 & k_ 1  & k_2 &
\\  & k_0 & k_ 1  & k_2
\\
  \ell_0 &   \ell_1 & \ell_2  &    \\
 &  \ell_0 &   \ell_1 & \ell_2    \end{array}} ,$$
 which has $k_0 k_2 \ell_0 \ell_2$ occurring twice.
 \end{exampl}

 We
resolve the conjecture in the standard supertropical case:

\begin{thm}\label{multresult6} For $L = L_{\ge 0},$ any polynomials $f,g,$ and
$h$ satisfy:
\begin{equation}\label{taneq01}|
\res(f,gh)| \lmodl |\res(f,g)| |\res(f,h)|  \quad \text{and }
\quad |\res(fg,h)| \lmodl |\res(f,h)|  |\res(g,h)|,\quad \text{for
} \ell \in \{ 1,  2, 3, 4 \} .\end{equation} In particular, the
conjecture holds when $L$ is truncated to level 4 or less (which
includes the supertropical case).
 \end{thm}
\begin{proof} By Theorem~\ref{multresult413}, we may assume that $f,$ $g,$ and $h$ are
$a$-primary. Then if $g$ or $h$ is linear we are done by
Lemma~\ref{multresult05}, so we may assume that $g$ and $h$ are of
degree $\ge 2,$ and for $\deg f \le 4$ we are done by
Example~\ref{multresult04}.
\end{proof}

\section{Layered derivatives and the discriminant}\label{diff}

In this section we assume that $R$ contains a zero element
$\rzero.$

\begin{defn}
The \textbf{layered derivative} $f'$ of $f$ on $R[\la ]$ is given
by:
\begin{equation}\label{diff2}
 \bigg(\sum_{j=0}^n \xl{\al _j}{\ell_j} \lm ^j\bigg)' := \sum_{j=1}^n  \xl{\al _j}{j \ell_j}  \lm
 ^{j-1} .\end{equation}
\end{defn}

 In particular, for $\a =\xl{\al }{1} \in R_1,$
 $$( \al  \lm^j)' := \xl{\al }{j} \lm^{j-1}
\quad (j\ge 2), \qquad (\al \lm)' := \al , \qquad \text{and } \
\al ' := \rzero.$$ We have the familiar formulas:
\begin{enumerate}
 \item  $(f+g)' = f' + g'$; \pSkip

  \item $(fg)' = f'g + fg'$.
\end{enumerate}

\begin{rem}\label{essfor} It is clear from \eqref{diff2} that if $f$ is essential then
so is $f',$ since the coefficients have the same respective
$\nu$-values (with the power of $\la$ decreased by 1).
\end{rem}

\begin{Note} Sheiner~\cite{Erez2} has noted that
the natural homomorphism from the polynomial \semiring0 in $n$
indeterminates to its image in $\Fun (R^{(n)}, R)$ (viewing a
polynomial as a function) does not commute with taking the
 layered derivative, even in the standard supertropical
setting. For example, take $f_a = \la^2 + a\la +2.$ Then the $f_a$
are all equal as functions whenever $a <_\nu 1.$ But $f_a' =
\la^\nu + a$, which all differ. Thus, when working with  $R[\la],$
one needs to choose a particular polynomial representative in
order for our definition of derivative to be well-defined.
Fortunately, the natural candidate, the essential form of the
polynomial, works well, by~Remark~\ref{essfor}.\end{Note}

\begin{exampl}\label{primarypart}  The  layered derivative of an $a$-primary
polynomial is an $a$-primary polynomial since the leading
coefficient is multiplied by $ \xl{\rone}{m}.$ The ghost layer of
$f'(a)$ can decrease dramatically when the ghost layer $\ell_ 0$
of the constant term is large enough.
\end{exampl}



\begin{rem}\label{antider}  Since the  layered derivative of $\xl{\a}{\ell}\la ^m $
is $\xl{\a}{m\ell}\la^{m-1}$, we would expect the anti-derivative
 of~$\xl{\a}{\ell}\la^m$   to be
$\xl{\a}{\frac\ell{m+1}}\la^{m+1}$. But in general  this only
makes sense when the sorting \semiring0~$L$ is $\mathbb
N$-divisible, for example when $L = \Q_{>0}$.\end{rem}

The   layered derivative of a separable polynomial can be factored
rather easily when $L = \Q_{>0}$.

\begin{prop}\label{sep0} If $f = (\la + a_m)\cdots (\la + a_1)$ is separable, then
$$f' = \big (\xl{\la}{m /(m-1)}+a_m \big)\big (\xl{\la}{(m-1) /(m-2)}+a_{m-1}\big) \ds \cdots
\big (\xl{\la}{2}+a_2\big)= \prod_{k=2}^m  \big(\xl{\la}{\frac{k}
{k-1}}+a_k\big).$$
\end{prop}
\begin{proof} Write $f = \sum_{i=0}^{m} \xl{\al_i }{\ell_ i} \lm ^{i },$ and then using
Equation~\eqref{diff2}, we write $f' = \sum_{i=1}^m \xl{\al_i
}{i\ell_ i} \lm ^{i -1}$ and factor it.
 \end{proof}

 \begin{cor}\label{sep1} The   layered derivative of a separable polynomial is
 separable.\end{cor}

 The
layered derivative enables one to define the \textbf{layered
discriminant} of a polynomial $f$ as the layered resultant
$|\res{(f,f')}|$ of $f$ and $f'$. Although the discriminant may be
difficult to compute in general, it is easy for separable
polynomials in view of Propositions ~\ref{firstcase2} and
~\ref{sep0}, for which ghost layer depends only on the degree
 of $f$ and not on the particular corner roots. Thus, we can
determine whether or not a polynomial is separable by checking the
 layer of its layered discriminant.

\begin{thm}\label{discsort} The sort of the layered discriminant of a tangible separable polynomial
$f(\la)$ of degree $m$ is $m^{m-1}\prod _{k =2}^m {\frac{(2k-1)}
{k(k-1)}}.$
\end{thm}
\begin{proof} Write $f = (\la + a_m)\cdots (\la + a_1)$; then  $f' = \prod_{k=2}^m  (\xl{\la}{\frac{k}
{k-1}}+a_k).$ Now, $$s\left(\left|\res(\la + a_j,
\xl{\la}{\frac{k}
{k-1}}+a_k)\right|\right)= \begin{cases}   {\frac{j} {j-1}}, & \quad j >k ; \\
{\frac{k} {k-1}} +1, & \quad j = k ; \\ 1, & \quad j <k .
\end{cases}$$
For each $k$, taking the product as $j$ runs from $1$ to $m$ gives
${\frac{m} {k}}({\frac{k} {k-1}} +1)1\cdots 1 = {\frac{m} {k-1}} +
{\frac{m} {k}} = {\frac{(2k-1)m} {k(k-1)}} $.
\end{proof}

When the polynomial is not separable, its multiple root makes the
sort higher. Thus, the existence of a multiple corner root of a
polynomial~$f$ can be recognized via the ghost layer of its
layered discriminant, without actually computing the roots of $f$.

\section{Major examples of layered
\domains0}\label{majex}

Having set forth the general layered structure, we pause to
indicate how specific  examples for $L$ fit into this context. For
convenience we usually take the sorting set $L$ to be a \semiring0
with $L = L_+$. One could also take the more general case where $L
= L_+ \cup \{ 0 \}.$  We consider how the choice of the sorting
\semiring0~$L$ affects the mathematical structure of~$R$. In all
but the last example (\ref{comp0}), $L$ is totally ordered, and
usually non-negative.

\subsection {Examples of uniform $L$-layered \domains0 for $L$ totally ordered under $(\ge)$}

Since the main goal of this paper is to enrich the ghost
structure, we turn first to describe various choices of $L$ in
terms of the  $L$-layered \domain0 construction here. We list them
and indicate their strengths and weaknesses.

\subsubsection{Examples for $L= L_{\ge 1}$}

The following examples are relatively easy to describe, and suit
most of our purposes. We already considered the max-plus situation
in Example~\ref{maxpl}. In all other examples, we assume that
$1+1>1$ in $L$.

\begin{exampl}\label{supertropst1} Taking $L=\{1,\infty\}$ yields
the (standard)
 supertropical \domains0, as noted above in Example~\ref{supertropst}.
  Although
 this structure has nice properties   obtained in  \cite{IzhakianRowen2008Matrices},  \cite{IzhakianRowen2009MatricesII},
 \cite{IKR1},  and \cite{IKR3}, especially in connection with linear
 algebra,
 it has the deficiency that some basic algebraic properties fail
 such as unique factorization of polynomials (cf.~\cite[Theorem~8.53]{IzhakianRowen2007SuperTropical}), and other
 properties (such as \cite[Lemma~2.2]{IzhakianRowen2007SuperTropical}) require case-by-case analysis.
 Furthermore,  since roots of polynomials are defined in terms of ghost values and there is only
 one layer of ghosts, it is difficult to study multiple roots of
 polynomials.

 Differentiation of polynomials is not very useful in this setting, as explained
 in \S\ref{diff}.
\end{exampl}

 \begin{rem} The standard supertropical case (Example~\ref{supertropst}) does have one advantage over the general $L$-layered
 structure --
 it satisfies the Frobenius property $(a+b)^{m} = a^m + b^m$ for all $m$, whereas for the general $L$-layered
 structure this holds only up to
 $\nu$-equivalence.
\end{rem}

\begin{exampl}\label{supertrop} In Example~\ref{supertropst}, we
could take instead   $L=\{1,2\}$ with $1<2,$   $ 1 \cdot 1=1,$ and
all other sums and products are $2,$ i.e., formally replacing the
index $\infty$ by 2. This fits in better with our notation of
truncation~(\S\ref{trunc1}), as to be explained
shortly.\end{exampl}

\begin{exampl}\label{ntrunc}  More generally, we may choose the
$q$-\textbf{truncated \semiring0}
$$L=[1,q]:=\{1,2,\dots,q\}$$ of Example~\ref{trunc0}. The elements of $R_q$ are all $q$-ghosts.  We then
obtain an $L$-layered \domain0 with $q$ layers. We recover
Example~\ref{supertrop} when $q=2$.
\end{exampl}

\begin{rem}\label{ntruncinf} Note in Example~\ref{ntrunc} that $q$ takes on the role of
the infinite element   in $L$.
 Thus, we may
relabel~$q$ as $\infty$, and call this \semiring0
$$L=[1,q-1]^{\infty}:=\{1,2,\dots,q-1, \infty\}.$$  This notation corresponds
better with Example~\ref{supertropst}.
\end{rem}

When $q>2$ we can describe when $\bfa \in R^{(n)}$ is a multiple
root   of a tangible polynomial $f$, in terms of when $f(\bfa)$ is
2-ghost.

\begin{rem}\label{Frob2} The Frobenius property $(a+b)^{m} = a^m + b^m$ does
hold in the $q$-truncated \domain0 for all $m\ge q$, since then
both $(a+b)^{m}$ and $a^m + b^m$ have maximal possible layer
$q$.\end{rem}

 The truncated layered \domain0
generalizes the situation given in \cite{CC}, as follows.
\begin{prop}\label{Frob3} (Notation as in  Remark~\ref{char2}.) Suppose
$R$ is a $q$-truncated layered \domain0 and also satisfies the
property:
$$\nu_{q,k} (a) =  e_q \quad \text{iff}  \quad  a \nucong
\rone, \qquad \forall a \in R_k.$$ (For example, this holds when
$R$ is uniform.) Then $\tlds^{q-1}(a) = \tlds ^q (a)$ for all
$a\nucong \rone,$ implying $\tlds^{q+1} = \tlds^{q}.$
\end{prop}
\begin{proof} If $a\nucong \rone$, then $$\tlds^{q-1}(a) = e_q = e_{q+1} = \tlds
(e_q) = \tlds(\tlds^{q-1}(a)) = \tlds ^q (a).$$ But in view of
Remark~\ref{char2}, if $a \not \nucong \rone,$ then either $\tlds
(a) = a,$ in which case $\tlds^{q+1}(a) = \tlds^{q}(a),$ or $\tlds
(a) = 1,$ in which case $\tlds^{q+1}(a) = \tlds^{q}(\rone)= \tilde
s^{q-1}(\rone)= \tlds^{q}(a).$
\end{proof}

\begin{exampl}\label{integ}  Taking $L = \bbN $ enables us to
deal with arbitrary multiplicities of corner roots, and also deal
with   layered derivatives, since we can apply the
formula~\eqref{diff2}. Thus, this situation is useful for studying
geometry.

There are difficulties from the algebraic perspective. We still
have irreducible non-primary polynomials such as $\xl{0}{1}\la ^2
+ \xl{2}{2} \la + \xl{3}{1}.$ Unique factorization of polynomials
still fails in one indeterminate, cf.~\cite{Erez1}.

 Also, often one cannot integrate since the
antiderivative $\xl{\la}{\frac \ell m}^{m+1}$ of
$\xl{\la}{\ell}^m$ described in~Remark~\ref{antider} does not
exist unless $m$ divides $\ell$.
\end{exampl}


\subsubsection{Examples for $L \neq  L_{\ge 1}$}
We  expand $L$ further, to handle more sophisticated mathematical
analysis, such as integration.

\begin{exampl}\label{ration1} Suppose $L = \{ \frac m{2^n}: m,n \in \bbN \}$.
 Sheiner \cite{Erez1} has pointed out that the polynomial
$$\la^2 + \xl{b}{2} + \xl{ab}{1}\, \qquad a <_\nu b,$$ which is irreducible
over the (standard) supertropical \semiring0, now has the
factorization
\begin{equation}\label{fact12}  \la^2 +
\xl{b}{2} + \xl{ab}{1} = (\la +   \xl{a}{\frac 12}) (\la +
\xl{b}{2}),
\end{equation}
which enables one to resolve the different factorizations in
$R[\la]$  given in \cite[Example
8.38(iii)]{IzhakianRowen2007SuperTropical}. On the other hand, as
noted above,
 unique factorization of primary polynomials still fails.

 Note that the ghost elements no longer form an
ideal, since we have elements of layer $<1$. The ghost valuation
\semiring0\ corresponds to $L = \{ \frac m{2^n}: m\ge 2^n \}$,
which contains the sub-\semiring0 $\Net$.
\end{exampl}

The situation is even better when $L$ is a multiplicative group.

\begin{exampl}\label{ration} Taking $L = \bbQ _{>0}$ enables
us to factorize polynomials in one indeterminate uniquely into
primary polynomials, as described in Theorem~\ref{Erez1}. Unique
factorization of $ \bbQ _{\ge 0}$-layered  primary polynomials
into irreducibles almost always holds, the only exception
involving the 0-layer,
 occurring either in the leading
monomial or the lowest order monomial, cf.~\cite{Erez1}.

  The ghost valuation
\semiring0\ corresponds to $L = \bbQ _{\ge 1}$, which contains the
sub-\semiring0 $\Net$, and this example should be a useful tool in
 geometric applications.  Also, one can integrate in this
setting, since the antiderivative of Remark~\ref{antider} now
makes sense.

The situation in several variables is not yet understood
completely, because of subtleties in the geometry. Sheiner
(cf.~\cite{Erez1}) gives an example of a polynomial with multiple
factorizations; this corresponds to a tropical hypersurface which
can be decomposed into different unions of irreducible
supertropical hypersurfaces, even taking layers into account.

\end{exampl}

\begin{exampl}\label{posreals} One could also take $L = \mathbb
R_{>0}$, which provides better factorizations of some primary
polynomials, although we do not yet see much advantage over
Example~\ref{ration}. In this case, the ghost valuation
\semiring0\ corresponds to $L = \bbR  _{\ge 1}$.
\end{exampl}

\subsection{Non-positive examples}

Other relevant examples   are more esoteric.

\begin{exampl}\label{reals} Taking $L = \mathbb
R $ provides unique factorization of primary polynomials into
linear and quadratic factors, but polynomials having a 0 component
need not be factorizable into primary polynomials. One could
interpret negative layers as ``antilayers,'' since $ \xl{a}{\ell}+
\xl{a}{-\ell} = \xl{a}{0} \in R_0.$

\end{exampl}

\begin{exampl}\label{comp0} Taking $L = \bbC$ provides unique factorization of primary polynomials into
linear factors, at the cost of losing the total order of the
reals. (Note that $\mathbb C$ could be pre-ordered via the
absolute value.)


We expect this example to be a useful intermediate tool in
tropical calculus, since one can  pass later to the layered
sub-\domain0 $\FR_{\ge 1}$.
\end{exampl}

\begin{exampl}\label{comp05} More generally, Sheiner \cite{Erez2} has an
interesting example taking $L = F,$ notation as in~
Example~\ref{analyt0}. Let $R = R(L,\tG)$, and define the map $K
\to R$ by $p \mapsto \xl{v(p)}{\alpha}$ where $\alpha$ is the
coefficient of the lowest monomial of the Puiseux series $p$. This
map, generalizing the Kapranov map, keeps track of the ``leading
coefficient'' of the Puiseux series $p$ in terms of when the image
of $p$ has layer 0, and provides a layered version of \cite{Par},
as to be explained further in Appendix A.
\end{exampl}

\begin{exampl}\label{comp06} If one is willing to forego
integration, we could take $L$ to be a finite field, with the
trivial pre-order.\end{exampl}

\section{Appendix A: Layered \domains0 with
symmetry, and patchworking}\label{basicsym1}

Akian, Gaubert, and Guterman \cite[Definition~4.1]{AGG} introduced
an involutory operation on semirings, which they call a
\textbf{symmetry}, to unify the supertropical theory with
classical ring theory. In this appendix, we put their symmetry in
the context of $L$-layered \domains0, where here $L$ is partially
pre-ordered. The main example for our construction is the
``patchworking'' given in Example~\ref{doubntrunc1} below.

\begin{defn}\label{negmap} A \textbf{negation map} on a \semiring0
$L$ is a function $\tau: L \to L$ satisfying the properties:

\boxtext{
\begin{itemize}
\item[N1.] $\tau (k\ell ) = \tau(k)\ell  = k\tau(\ell )$; \pSkip

\item[N2.] $\tau^2 (k) = k$; \pSkip

\item[N3.] $\tau (k+\ell ) = \tau(k) + \tau(\ell )$.
\end{itemize}}
\end{defn}

Suppose the \semiring0  $L$ has a negation  map $\tau$ of order
$\le 2$. We say that   an $L$-quasi-layered \domain0
$(R;\tau,\sig)$ is an $(L,\tau)$-\textbf{quasi-layered \semiring0\
with symmetry} $\sig$ when $R$ is a \semiring0\ together with a
map
$$\sig : R\to R$$ and a negation map $\tau$ on $L$, together with the extra axiom (for all
$a\in \R _k,$ $b\in \R _{\ell}$):

\boxtext{
\begin{enumerate}\eroman

\item[S1.]   $s( \sig(a)) = \tau(s(a)).$
\end{enumerate} }

Note that when $\sigma$ and $\tau$ are the identity maps, we are
back to $L$-layered \domains0.

\begin{rem}\label{det7}
One big advantage of the symmetry is that it enables one to return
to a more classical definition of determinant of a matrix $A =
(a_{ij})$,   defined as \begin{equation}  \sum_{\pi \in S_n} \sig
^{\operatorname{sgn}(\pi)} (a_{1, \pi(1)} \cdots
a_{n,\sig(n)}).\end{equation}
\end{rem}

\begin{rem} Concerning truncation in the context of symmetries, we  observe briefly that when $L$ has a
given negation map  $\tau$, we should require our upper ideal $Q$
to be $\tau$-invariant; i.e., $\tau(Q) \subseteq Q$. Then $\tau$
induces a negation map on $\bar L$, which can be used to define a
natural symmetry on $\olR,$ and a truncation that works in
parallel to Definition~\ref{defn6}.\end{rem}

\subsection{Examples of
$(L,\tau)$-layered \domains0 with symmetry}\label{majex1}

\begin{exampl}\label{comp06} Whenever $L$ is a ring,
one can define a negation map by putting $\tau(\ell) = -\ell$, and
then define the layered symmetry via $\sig(\! \xl a {\ell}) = \xl
a {-\ell}.$ Applying this to Example~\ref{comp05},
 Sheiner~\cite{Erez2} has
exploited Remark~\ref{det7} to study the linear algebra of this
structure via the 0-layer. For example,  a matrix is singular iff
\eqref{det7} has layer 0.\end{exampl}

 We conclude this appendix
with an example motivated from Viro's theory of patchworking in
tropical geometry, as developed in \cite[Chapter 2]{IMS}, in which
the sorting \semiring0\ $L$ is more intricate, with a partial
order which is not total.

\begin{exampl}\label{doubntrunc}  Suppose
$L$ is an ordered \semiring0. We mimic the construction of
$\mathbb Z$ from $\Net.$ Define the \textbf{doubled \semiring0}
 $$D(L) = L_1 \times L_{-1},$$ the direct product of two
copies $L_1$ and $L_{-1}$, where addition is defined
componentwise, but multiplication is given by
$$(k,\ell)\cdot (k',\ell') = (kk'+\ell\ell',k \ell'+\ell k').$$ In other words, $D(L)$ is multiplicatively
graded by $\{ \pm 1 \}.$

$D(L)$ is endowed with the product partial order, i.e.,
$(k',\ell') \ge ( k,\ell)$ when  $k'\ge k$ and $\ell' \ge \ell$.
To see this, note that if $(k',\ell')\ge (k,\ell)$, then
multiplying by $(m,n)$ gives $$(k'm + \ell' n, k'n + \ell' m)\ge
(km + \ell n, kn + \ell m).$$ Furthermore, $D(L)$ has the negation
map $\tau$ of order 2, given by $\tau(k,\ell) = (\ell,k)$.

In case $L$ is truncated, as in Example~\ref{ntrunc}, with maximal
  element $n $, then
$(n, n)\ge (k,\ell)$ for all $k$  and $\ell$, so $(n,n)$ is the
unique maximal element of $D(L)$. On the other hand, one could
take infinitely many layers, such as $L = \bbN $ as in
Example~\ref{integ}.
\end{exampl}

Here is the $D(L)$-layered \domain0 with symmetry  of greatest
interest to us.

\begin{exampl}\label{doubntrunc1} Suppose $\tG$ is an ordered abelian monoid,   viewed as a \semiring0\ as in
Construction~\ref{defn5}.  Define the \textbf{double layered
\domain0}
$$R = \R({D(L)},\tG) = \{ ((k,\ell),a): (k, \ell) \ne
(0,0), \ a \in \tG\},$$  but with addition and multiplication
given by the following rules:

\begin{equation*}\label{141}
\begin{array}{rll}
((k,\ell),a)+ ((k',\ell'),b) & = & \begin{cases} ((k,\ell),a)& \quad\text{if}\ a> b,\\
((k',\ell'),b)& \quad\text{if}\ a< b,\\
((k+k',   \ell + \ell'),\,a)& \quad\text{if}\ a= b.
\end{cases}\\ \\
((k,\ell),a)\cdot  ((k',\ell'),b) & = & ((k k'+ \ell \ell',  k
\ell'+k'\ell),\,ab). \end{array}
\end{equation*}
\end{exampl}

One can check routinely that this is a commutative \semiring0.
When $L = \{ 1, \infty\},$ we note that $$D(L) = \{ (1,1),
(1,\infty), (\infty,1), (\infty,\infty) \},$$ which is applicable
to Viro's theory of patchworking, where the ``tangible'' part
could be viewed as those elements of layer $(1,1), (1,\infty),$ or
$(\infty,1)$. Explicitly, comparing with Viro's use of hyperfields
in \cite[\S~3.5]{V}, we can identify these three layers
respectively with $0, 1,$ and $ -1$ in his terminology, and the
element $(\infty,\infty)$ with the set $\{0, 1, -1\}$.

\begin{rem} In the doubled layered \domain0 $R =
\R({D(L)},\tG),$ we consider the symmetry $\sigma: R \to R$ given
by $\sigma : ((k,\ell),a) \mapsto ( (\ell,k),a).$ This symmetry is
analogous to the one described in \cite{AGG}, and behaves much
like the negation.
\end{rem}

\begin{Note} When the order on $L$ is only partial, $L$ could have several multiplicative
 idempotents other than~$1$ and $\infty$, cf.~Example~\ref{doubntrunc1}.
 Thus, one would want to define tangible elements more generally,
 in terms of these idempotents, and Lemma~\ref{tang11} needs to
 be modified. Otherwise, the theory pretty much follows the same
 lines given there.\end{Note}

\section{Appendix B: Weakening the structure of $L$ and $R$}\label{basicsym}

 Strictly speaking, we have only generalized the
notion of supertropical \domain0, not supertropical semiring,
since Axiom A2 says  that $a,b \in R_1$ implies $ab \in R_1.$ We
take a brief  excursion  to consider a slight generalization that
covers this case also.


\begin{Note}\label{arith} To generalize the notion
``supertropical semiring'' from the standard supertropical theory,
we would weaken Axiom A2 to: \boxtext{
\begin{enumerate}\item[wA2.] If $a\in
\R _k$ and $b\in \R _\ell,$ then $ab\in \R _m$ for some $m \ge {k
\ell }.$\end{enumerate}}

Now we have to modify Axiom A3 to make it compatible; i.e.,
multiplication commutes with the sort transition maps.
Technically, this  says:

 \boxtext{
\begin{enumerate}\item[wA3.] If $a\in \R _k$ and $a'\in \R _{k'},$ with $aa' \in R_{k''}$ and
$\nu_{\ell,k}(a)\cdot\nu_{\ell',k'}(a')\in R_{\ell''}$ and
$\nu_{m,\ell}(a)\cdot\nu_{m',\ell'}(a'')\in R_{m''},$ for $m \ge
\ell$, $m' \ge \ell'$, and $m'' \ge mm'$,
 then
\qquad
$\nu_{q,\ell''}(aa')=\nu_{q,m''}(\nu_{m,\ell}(a)\cdot\nu_{m',\ell'}(a'))$
 for all $q \ge \ell'',m''$.  \end{enumerate}}
\end{Note}

This weakening is of arithmetic interest, since we now have a
version of Example~\ref{trun0001} without requiring a zero layer.

\begin{exampl}[The weakly layered truncated \semiring0]\label{trun000} Suppose $R$ is $L$-quasi-layered. Fix $q >0,$ and for any \semiring0 $L$ we
formally adjoin an infinite sort $\infty$, letting $L_\infty = L
\cup \infty.$ Define
$$\htR(L_\infty ,[1,q]) := \{ \xl{a}{k}: k\in L, \ a \in \{1,
\dots, q-1\}\} \cup \{\xl{q}{\infty}\}, $$ where addition is
defined as in Construction~\ref{defn5}, and the product $
\xl{a}{k} \xl{b}{\ell}$ is given as in Equation~\eqref{13} except
for $ab=q$, in which case $ \xl{a}{k} \xl{b}{\ell}
=\xl{q}{\infty}$ for any $k,\ell \in L.$
  Addition and multiplication by $ \xl{q}{\infty}$ are given by:
$$\xl{a}{k} +   \xl{q}{\infty} =   \xl{q}{\infty} = \xl{a}{k}   \xl{q}{\infty}.$$

One checks as before that $\htR(L_\infty ,[1,q])$ is indeed a
\semiring0. The sort transition maps are as in
Construction~\ref{defn5}, except that we define $\nu_{\infty,
k}(\xl{a}{k}) =  \xl{q}{\infty}$ for all $(k,a)$. Thus, $
\xl{q}{\infty}$ is the special infinite element.

\end{exampl}
%
%

When we forego $\nu$-bipotence, we do not need $L$ to be a
\semiring0, but merely a directed, partially pre-ordered
multiplicative monoid (without addition). Although this material
is not needed for our current applications to tropical
mathematics, it yields an intriguing parallel between the
\semiring0 $R$ and the sorting set $L$ (since any ordered monoid
becomes a \semiring0 when addition is taken to be the maximum),
and may provide guidance for future research.

\begin{rem}\label{remove1}
Since $L$ now is only assumed to be a multiplicative monoid, we
need to remove references to addition in $L$. Thus, we need
 a formal ``doubling function'' $\ell \mapsto 2\ell$ on
$L$, and use strong ghosts, eliminate Axiom A4, and weaken Axiom B
to:

\boxtext{
\begin{enumerate}
\item[wB.]  (weak supertropicality) If $a\in \R _k$ and $b\in \R
_{\ell }$ with $a \nucong b$, then $a+b \in R_m$ for some
 $m \ge k, \ell, \min\{2k, 2\ell\}$  with $a+b \nucong b$.
\end{enumerate}}
\end{rem}

%


It is easy  to check that the sorting map $s:R\to L$ still exists
and satisfies Equation~\eqref{sortval}.


%

\end{document}